\documentclass[12pt]{amsart}

\usepackage{amsfonts, dutchcal, fancyhdr, qtree, amsmath, amscd, tipa, amssymb, fullpage, hyperref, url, ifsym, amsthm, subfigure, enumitem, wasysym, tikz, ytableau, mathtools, tabularx, ltablex}
\usetikzlibrary{matrix, arrows}
\allowdisplaybreaks

\newtheorem{thm}{Theorem}[section]
\newtheorem{lem}[thm]{Lemma}

\newtheorem{prop}[thm]{Proposition}

\newtheorem{cor}[thm]{Corollary}
\theoremstyle{definition}
\newtheorem{df}[thm]{Definition}
\newtheorem{exam}[thm]{Example}
\theoremstyle{remark}
\newtheorem{rem}[thm]{Remark}
\theoremstyle{question}

\numberwithin{equation}{section}

\title{Computing the Lusztig--Vogan Bijection}
\author{David B Rush}
\email{dbr@alum.mit.edu}
\dedicatory{To the memory of Prof. John W. Chun (1930--2016)}
\date{\today}

\begin{document}

\begin{abstract}
Let $G$ be a connected complex reductive algebraic group with Lie algebra $\mathfrak{g}$.  The Lusztig--Vogan bijection relates two bases for the bounded derived category of $G$-equivariant coherent sheaves on the nilpotent cone $\mathcal{N}$ of $\mathfrak{g}$.  One basis is indexed by $\Lambda^+$, the set of dominant weights of $G$, and the other by $\Omega$, the set of pairs $(\mathcal{O}, \mathcal{E})$ consisting of a nilpotent orbit $\mathcal{O} \subset \mathcal{N}$ and an irreducible $G$-equivariant vector bundle $\mathcal{E} \rightarrow \mathcal{O}$.  The existence of the Lusztig--Vogan bijection $\gamma \colon \Omega \rightarrow \Lambda^+$ was proven by Bezrukavnikov, and an algorithm computing $\gamma$ in type $A$ was given by Achar.  Herein we present a combinatorial description of $\gamma$ in type $A$ that subsumes and dramatically simplifies Achar's algorithm.  
\end{abstract}

\maketitle 

\tableofcontents

\section*{Overview}
In 1989, Lusztig concluded his landmark four-part study of cells in affine Weyl groups \cite{Lusztig1, Lusztig2, Lusztig3, Lusztig4} with an almost offhand remark:

\begin{quote}
``\ldots we obtain a (conjectural) bijection between $X_{\text{dom}}$ and the set of pairs $(u, \rho)$, (up to $G$-conjugacy) with $u \in G$ unipotent and $\rho$ an irreducible representation of $Z_G(u)$.''
\end{quote}

By $X_{\text{dom}}$, Lusztig meant the set of dominant weights of a connected complex reductive algebraic group $G$.  (We refer to this set as $\Lambda^+$.)  We denote by $\Omega$ the set of pairs $(\mathcal{C}, V)$, where $\mathcal{C} \subset G$ is a unipotent conjugacy class and $V$ is an irreducible representation of the centralizer $Z_G(u)$ for $u \in \mathcal{C}$, which is uniquely determined by $\mathcal{C}$ up to inner isomorphism.

So elementary an assertion was Lusztig's claim of a bijection between $\Lambda^+$ and $\Omega$ that its emergence from so deep an opus was in retrospect an obvious indication that the close connection between the sets in question transcends the setting in which it was first glimpsed.  

Indeed, Vogan's work on associated varieties \cite{Vogan} led him to the same supposition only two years later.  Let $\mathfrak{g}$ denote the Lie algebra of $G$, and let $\mathcal{N}^*$ denote the nilpotent cone of the dual space $\mathfrak{g}^*$.  Fixing a compact real form $K$ of $G$ with Lie algebra $\mathfrak{k}$, let $\mathfrak{C}$ be the category of finitely generated $(S(\mathfrak{g}/\mathfrak{k}), K)$-modules for which each prime ideal in the support corresponds under the Nullstellensatz to a subvariety of $(\mathfrak{g}/\mathfrak{k})^* \subset \mathfrak{g}^*$ contained in $\mathcal{N}^*$.  In 1991, Vogan \cite{Vogan} showed that $\Omega$ --- in an alternate incarnation as the set of pairs $(\mathcal{O}, V)$, where $\mathcal{O} \subset \mathcal{N}^*$ is a coadjoint orbit and $V$ is an irreducible representation of the stabilizer $G_X$ for $X \in \mathcal{O}$ --- indexes a basis for the Grothendieck group $K_0(\mathfrak{C})$.  That $\Lambda^+$ also indexes such a basis pointed to an uncharted bijection.  

Further evidence for the existence of what has come to be known as the \textit{Lusztig--Vogan bijection} was uncovered by Ostrik \cite{Ostrik}, who was first to consider $\Omega$ and $\Lambda^+$ in the context in which the conjecture was ultimately confirmed by Bezrukavnikov \cite{Bezrukav} --- that of the equivariant $K$-theory of the nilpotent cone of $\mathfrak{g}$.  Let $\mathcal{N}$ denote this nilpotent cone.  Ostrik examined $(G \times \mathbb{C}^*)$-equivariant coherent sheaves on $\mathcal{N}$.  Subsequently, Bezrukavnikov examined $G$-equivariant coherent sheaves on $\mathcal{N}$ and proved Lusztig and Vogan's claim.

Let $\mathfrak{D} := \textbf{D}^b(\operatorname{Coh}^G(\mathcal{N}))$ be the bounded derived category of $G$-equivariant coherent sheaves on $\mathcal{N}$.  Bezrukavnikov \cite{Bezrukav} showed not only that $\Omega$ and $\Lambda^+$ both index bases for the Grothendieck group $K_0(\mathfrak{D})$, but also that there exists a bijection $\gamma \colon \Omega \rightarrow \Lambda^+$ uniquely characterized by the following property: For any total order $\leq$ on $\Lambda^+$ compatible with the root order, if $\leq$ is imposed on $\Omega$ via $\gamma^{-1}$, then the change-of-basis matrix is upper triangular.  

In his proof, Bezrukavnikov did not construct $\gamma$.  Instead, he exhibited a $t$-structure on $\mathfrak{D}$, the heart of which is a quasi-hereditary category with irreducible objects indexed by $\Omega$ and costandard objects indexed by $\Lambda^+$.  This entailed the existence of $\gamma$, but left open the question of how $\gamma$ is computed.\footnote{In type A, the existence of the Lusztig--Vogan bijection also follows from Xi's work on the based ring of the affine Weyl group \cite{Xi}, in which he proved a more general conjecture of Lusztig \cite{Lusztig4}.  }  

In his 2001 doctoral thesis \cite{Achart}, Achar set $G := GL_n(\mathbb{C})$ and formulated algorithms to compute inverse maps $\Omega \rightarrow \Lambda^+$ and $\Lambda^+ \rightarrow \Omega$ that yield an upper triangular change of basis in $K_0(\mathfrak{C})$.  Then, in a follow-up article \cite{Acharj}, he showed that his calculations carry over to $K_0(\mathfrak{D})$ and therefore that his bijection agrees with Bezrukavnikov's.      

Achar's algorithm for $\gamma^{-1}$ is elegant and simple.  Unfortunately, his algorithm for $\gamma$ is a series of nested while loops, set to terminate upon reaching a configuration satisfying a list of conditions.  Progress is tracked by a six-part monovariant, which is whittled down as the algorithm runs.  Achar \cite{Achart, Acharj} proved that his algorithm halts on every input after finitely many steps.  But it does not directly describe the image of a given pair $(\mathcal{O}, V) \in \Omega$.    

In this article, we present an algorithm that directly describes the terminal configuration returned by Achar's algorithm on an input in $\Omega$, bypassing all of Achar's while loops and obviating the need for an accompanying monovariant.  The upshot is a combinatorial algorithm to compute $\gamma$ for $G = GL_n(\mathbb{C})$ that encompasses and expedites Achar's algorithm and holds the prospect of extension to other classical groups.\footnote{A conjectural algorithm, to compute $\gamma$ for \textit{even} nilpotent orbits in type $C$, is featured in Chapter 7 of the author's 2017 doctoral thesis \cite{Rush}.  }  
\vfill \eject

\section*{Index of Notation}

\begin{tabularx}{6.0in}{l X l} 
	$G$ & connected complex reductive algebraic group & \S 1.1 \\
	$\mathfrak{g}$ & Lie algebra of $G$ & \S 1.1 \\
	$\mathcal{N}$ & nilpotent cone of $\mathfrak{g}$ & \S 1.1 \\
	$\mathfrak{D}$ & bounded derived category of $G$-equivariant coherent sheaves on $\mathcal{N}$ & \S 1.1 \\
	$X$ & nilpotent element & \S 1.1 \\
	$\mathcal{O}_X$ & nilpotent orbit of $X$ & \S 1.1 \\
	$G_X$ & stabilizer of $X$ & \S 1.1 \\
	$(\mathcal{O}_X, V)$ & pair consisting of nilpotent orbit $\mathcal{O}_X$ and irreducible $G_X$-representation $V$ & \S 1.1 \\
	$IC_{(\mathcal{O}_X, V)}$ & intersection cohomology complex associated to $(\mathcal{O}_X, V)$ & \S 1.1 \\
	$\Omega$ & equivalence classes of pairs $(\mathcal{O}_X, V)$ & \S 1.1 \\
	$A_{\lambda}$ & complex associated to weight $\lambda$ via Springer resolution & \S 1.1 \\
	$\Lambda$ & weight lattice of $G$ & \S 1.1 \\
	$\Lambda^+$ & dominant weights of $G$ & \S 1.1 \\
	$\gamma(\mathcal{O}_X,  V)$ & Lusztig--Vogan bijection & \S 1.1 \\
	$A^P_{\lambda}$ & complex associated to weight $\lambda$ via $T^*(G/P) \rightarrow \overline{\mathcal{O}}$ & \S 1.1 \\
	$[\alpha_1, \ldots, \alpha_{\ell}]$ & partition associated to $X$ & \S 1.2 \\
	$[k_1^{a_1}, \ldots, k_m^{a_m}]$ & distinct parts of $\alpha$ with multiplicities & \S 1.2 \\
	$G_X^{\text{red}}$ & reductive part of $G_X$ & \S 1.2 \\
	$[\alpha^*_1, \ldots, \alpha^*_s]$ & conjugate partition to $\alpha$ & \S 1.2 \\
	$P_X$ & parabolic subgroup associated to $X$ & \S 1.2 \\
	$L_X$ & Levi factor of $P_X$ & \S 1.2 \\
	$L_X^{\text{ref}}$ & Levi subgroup of $L_X$ containing $G_X^{\text{red}}$ & \S 1.2 \\
	$X_{\alpha}$ & representative element of $\mathcal{O}_X$ & \S 1.3 \\
	$\mathcal{O}_{\alpha}$ & $\mathcal{O}_{X_{\alpha}}$ & \S 1.3 \\
	$V^{\nu(t)}$ & irreducible $GL_{a_t}$-representation with highest weight $\nu(t)$ & \S 1.3 \\
	$V^{(\nu(1), \ldots, \nu(m))}$ & $V^{\nu(1)} \boxtimes \cdots \boxtimes V^{\nu(m)}$ & \S 1.3 \\
	$[\nu_1, \ldots, \nu_{\ell}]$ & integer sequence & \S 1.3 \\
	$G_{\alpha}$ & $G_{X_{\alpha}}$ & \S 1.3 \\
	$G_{\alpha}^{\text{red}}$ & $G_{X_{\alpha}}^{\text{red}}$ & \S 1.3 \\
	$V^{(\alpha, \nu)}$ & $G_{\alpha}$-representation arising from $V^{(\nu(1), \ldots, \nu(m))}$ & \S 1.3 \\
	$P_{\alpha}$ & $P_{X_{\alpha}}$ & \S 1.3 \\
	$L_{\alpha}$ & $L_{X_{\alpha}}$ & \S 1.3 \\
	$\Lambda^+_{\alpha}$ & dominant weights of $L_{\alpha}$ & \S 1.3 \\
	$W^{\lambda^j}$ & irreducible $GL_{\alpha^*_j}$-representation with highest weight $\lambda^j$ & \S 1.3 \\
	$W^{\lambda}$ & $W^{\lambda^1} \boxtimes \cdots \boxtimes W^{\lambda^s}$ & \S 1.3 \\
	$A^{\alpha}_{\lambda}$ & $A^{P_{\alpha}}_{\lambda}$ & \S 1.3 \\
	$W_{\alpha}$ & Weyl group of $L_{\alpha}$ & \S 1.3 \\
	$\rho_{\alpha}$ & half-sum of positive roots of $L_{\alpha}$ & \S 1.3 \\
	$W$ & Weyl group of $G$ & \S 1.3 \\
	$\operatorname{dom}(\mu)$ & unique dominant weight of $G$ in $W$-orbit of $\mu$ & \S 1.3 \\
	$\Omega_{\alpha}$ & dominant integer sequences with respect to $\alpha$ & \S 1.3 \\
	$\Lambda^+_{\alpha, \nu}$ & dominant weights $\mu$ of $L_{\alpha}$ such that $V^{(\alpha, \nu)}$ occurs in decomposition of $W^{\mu}$ as direct sum of irreducible $G_{\alpha}^{\text{red}}$-representations & \S 1.3 \\
	$\mathfrak{A}(\alpha, \nu)$ & integer-sequences version of algorithm & \S 1.5 \\
	$\mathsf{A}(\alpha, \nu)$ & Achar's algorithm & \S 1.5 \\
	$\mathcal{A}(\alpha, \nu)$ & weight-diagrams version of algorithm & \S 1.5 \\
	$\operatorname{dom}(\iota)$ & rearrangement of entries of $\iota$ in weakly decreasing order & \S 2.1 \\
	$\mathcal{C}_{-1}(\alpha, \nu, i, I_a, I_b)$ & candidate-ceiling function & \S 2.2 \\ 
	$\mathcal{R}_{-1}(\alpha, \nu)$ & ranking-by-ceilings function & \S 2.2 \\
	$\sigma$ & permutation & \S 2.2 \\
	$\mathbb{Z}^{\ell}_{\text{dom}}$ & weakly decreasing integer sequences of length $\ell$ & \S 2.2 \\
	$\mathcal{U}_{-1}(\alpha, \nu, \sigma)$ & column-ceilings function & \S 2.2 \\
	$\mathfrak{A}_{\operatorname{iter}}(\alpha, \nu)$ & iterative integer-sequences version of algorithm & \S 2.2 \\
	$D_{\alpha}$ & weight diagrams of shape-class $\alpha$ & \S 3 \\
	$X$ & weight diagram & \S 3 \\
	$X^j_i$ & $i^{\text{th}}$ entry from top in $j^{\text{th}}$ column of $X$ & \S 3 \\
	$EX$ & map $D_{\alpha} \rightarrow D_{\alpha}$ & \S 3 \\
	$(X, Y)$ & diagram pair & \S 3 \\
	$\kappa(X)$ & map $D_{\alpha} \rightarrow \Omega_{\alpha}$ & \S 3 \\
	$h(X)$ & map $D_{\alpha} \rightarrow \Lambda^+_{\alpha}$ & \S 3 \\
	$\eta(Y)$ & map $D_{\alpha} \rightarrow \Lambda^+$ & \S 3 \\
	$D_{\ell}$ & weight diagrams with $\ell$ rows & \S 4.1 \\
	$X_{i,j}$ & entry of $X$ in $i^{\text{th}}$ row and $j^{\text{th}}$ column & \S 4.1 \\
	$\mathcal{S}(\alpha, \sigma, \iota)(i)$ & row-survival function & \S 4.1 \\
	$\mathcal{k}$ & number of branches & \S 4.1 \\
	$\ell_x$ & number of rows surviving into $x^{\text{th}}$ branch & \S 4.1 \\
	$\mathcal{C}_1(\alpha, \nu, i, I_a, I_b)$ & candidate-floor function & \S 4.2 \\
	$\mathcal{R}_1(\alpha, \nu)$ & ranking-by-floors function & \S 4.2 \\
	$\mathcal{U}_1(\alpha, \nu, \sigma)$ & column-floors function & \S 4.2 \\
	$\alpha^*_j$ & $|\lbrace i: \alpha_i \geq j \rbrace|$ & \S 4.2 \\
	$\#(X,i)$ & number of boxes in $i^{\text{th}}$ row of $X$ & \S 4.3 \\
	$\Sigma(X,i)$ & sum of entries in $i^{\text{th}}$ row of $X$ & \S 4.3 \\
	$\mathcal{P}(\alpha, \iota)(i)$ & row-partition function & \S 5 \\
	$\operatorname{Cat}$ & diagram-concatenation function & \S 5 \\
	$\mathcal{T}_j(X)$ & column-reduction function & \S 5 \\

\end{tabularx}
\vfill \eject

\section{Introduction}

\subsection{Sheaves on the nilpotent cone}
Let $G$ be a connected complex reductive algebraic group with Lie algebra $\mathfrak{g}$.  An element $X \in \mathfrak{g}$ is \textit{nilpotent} if $X \in [\mathfrak{g}, \mathfrak{g}]$ and the endomorphism $\operatorname{ad} X \colon \mathfrak{g} \rightarrow \mathfrak{g}$ is nilpotent.  The \textit{nilpotent cone} $\mathcal{N}$ comprises the nilpotent elements of $\mathfrak{g}$.  Since $\mathcal{N}$ is a subvariety of $\mathfrak{g}$ (cf. Jantzen \cite{Jantzen}, section 6.2), we may consider the bounded derived category $\mathfrak{D} := \textbf{D}^b(\operatorname{Coh}^G(\mathcal{N}))$ of $G$-equivariant coherent sheaves on $\mathcal{N}$.   

Let $X \in \mathfrak{g}$ be nilpotent, and write $\mathcal{O}_X \subset \mathcal{N}$ for the orbit of $X$ in $\mathfrak{g}$ under the adjoint action of $G$.  We refer to $\mathcal{O}_X$ as the \textit{nilpotent orbit} of $X$.  

Write $G_X$ for the stabilizer of $X$ in $G$.  To an irreducible representation $V$ of $G_X$ corresponds the $G$-equivariant vector bundle \[E_{(\mathcal{O}_X, V)} := G \times_{G_X} V \rightarrow \mathcal{O}_X\] with projection given by $(g, v) \mapsto \operatorname{Ad}(g) (X)$.  Its sheaf of sections $\mathcal{E}_{(\mathcal{O}_X, V)}$ is a $G$-equivariant coherent sheaf on $\mathcal{O}_X$.  To arrive at an object in the derived category $\mathfrak{D}$, we build the complex $\mathcal{E}_{(\mathcal{O}_X, V)}[\frac{1}{2} \dim \mathcal{O}_X]$ consisting of $\mathcal{E}_{(\mathcal{O}_X, V)}$ concentrated in degree $-\frac{1}{2} \dim \mathcal{O}_X$.  Then we set \[IC_{(\mathcal{O}_X, V)} := j_{!*}\left(\mathcal{E}_{(\mathcal{O}_X, V)}\left[\frac{1}{2} \dim \mathcal{O}_X\right]\right) \in \mathfrak{D},\] where $j_{!*}$ denotes the Goresky--Macpherson extension functor obtained from the inclusion $j \colon \mathcal{O}_X \rightarrow \mathcal{N}$ and Bezrukavnikov's $t$-structure on $\mathfrak{D}$.  

Let ${\Omega}^{\text{pre}}$ be the set of pairs $\lbrace (\mathcal{O}_X, V) \rbrace_{X \in \mathcal{N}}$ consisting of a nilpotent orbit $\mathcal{O}_X$ and an irreducible representation $V$ of the stabilizer $G_X$.  We assign an equivalence relation to ${\Omega}^{\text{pre}}$ by stipulating that $(\mathcal{O}_X, V) \sim (\mathcal{O}_Y, W)$ if there exists $g \in G$ and an isomorphism of vector spaces $\pi \colon V \rightarrow W$ such that $\operatorname{Ad}(g) X = Y$ and the group isomorphism $\operatorname{Ad}(g) \colon G_X \rightarrow G_Y$ manifests $\pi$ as an isomorphism of $G_X$-representations.  

Note that $(\mathcal{O}_X, V) \sim (\mathcal{O}_Y, W)$ implies $\mathcal{O}_X = \mathcal{O}_Y$ and $E_{(\mathcal{O}_X, V)} \cong E_{(\mathcal{O}_Y, W)}$.  Thus, the map associating the intersection cohomology complex $IC_{(\mathcal{O}_X, V)}$ in $\mathfrak{D}$ to the equivalence class of $(\mathcal{O}_X, V)$ in $\Omega^{\text{pre}}$ is well-defined.  Set $\Omega := \Omega^{\text{pre}} / \sim$.  Then $\Omega$ indexes the family of complexes $\lbrace IC_{(\mathcal{O}_X, V)} \rbrace_{(\mathcal{O}_X, V) \in \Omega}$.  (The notation $(\mathcal{O}_X, V) \in \Omega$ is shorthand for the equivalence class represented by $(\mathcal{O}_X, V)$ belonging to $\Omega$.)  

On the other hand, weights of $G$ also give rise to complexes in $\mathfrak{D}$.  To see this, let $B$ be a Borel subgroup of $G$, and fix a maximal torus $T \subset B$.  A weight $\lambda \in \operatorname{Hom}(T, \mathbb{C}^*)$ is a character of $T$, from which we obtain a one-dimensional representation $\mathbb{C}^{\lambda}$ of $B$ by stipulating that the unipotent radical of $B$ act trivially.  Then \[L_{\lambda} := G \times_B \mathbb{C}^{\lambda} \rightarrow G/B \] is a $G$-equivariant line bundle on the flag variety $G/B$.  Its sheaf of sections $\mathcal{L}_{\lambda}$ is a $G$-equivariant coherent sheaf on $G/B$ which may be pulled back to the cotangent bundle $T^*(G/B)$ along the projection $p \colon T^*(G/B) \rightarrow G/B$.  

From the Springer resolution of singularities $\pi \colon T^*(G/B) \rightarrow \mathcal{N}$, we obtain the direct image functor $\pi_{*}$, and then the total derived functor $R\pi_{*}$.  We set \[A_{\lambda} := R\pi_{*} p^{*} \mathcal{L}_{\lambda} \in \mathfrak{D}.\]

Let $\Lambda := \operatorname{Hom}(T, \mathbb{C}^*)$ be the weight lattice of $G$, and let $\Lambda^+ \subset \Lambda$ be the subset of dominant weights with respect to $B$.  The family of complexes $\lbrace A_{\lambda} \rbrace_{\lambda \in \Lambda^+}$ is sufficient to generate the Grothendieck group $K_0(\mathfrak{D})$, so it is this family which we compare to $\lbrace IC_{(\mathcal{O}_X, V)} \rbrace_{(\mathcal{O}_X, V) \in \Omega}$.  Entailed in the relationship is the Lusztig--Vogan bijection.  

\begin{thm}[Bezrukavnikov \cite{Bezrukav}, Corollary 4] \label{bez}
The Grothendieck group $K_0(\mathfrak{D})$ is a free abelian group for which both the sets $\lbrace [IC_{(\mathcal{O}_X, V)}] \rbrace_{(\mathcal{O}_X, V) \in \Omega}$ and $\lbrace [A_{\lambda}] \rbrace_{\lambda \in \Lambda^+}$ form bases.  There exists a unique bijection $\gamma \colon \Omega \rightarrow \Lambda^+$ such that \[\left[IC_{(\mathcal{O}_X, V)}\right] \in \operatorname{span} \lbrace [A_{\lambda}] : \lambda \leq \gamma(\mathcal{O}_X, V) \rbrace, \] where the partial order on the weights is the root order, viz., the transitive closure of the relations $\upsilon \lessdot \omega$ for all $\upsilon, \omega \in \Lambda$ such that $\omega - \upsilon$ is a positive root with respect to $B$.  

Furthermore, the coefficient of $[A_{\gamma(\mathcal{O}_X, V)}]$ in the expansion of $[IC_{(\mathcal{O}_X, V)}]$ is $\pm 1$.  
\end{thm}

The association of the complex $A_{\lambda}$ to the weight $\lambda$ evinces a more general construction of objects in $\mathfrak{D}$ that is instrumental in identifying the bijection $\gamma$.  Let $P \supset B$ be a parabolic subgroup, and let $U_P$ be its unipotent radical.  Denote the Lie algebra of $U_P$ by $\mathfrak{u}_P$.  The unique nilpotent orbit $\mathcal{O}$ for which $\mathcal{O} \cap \mathfrak{u}_P$ is an open dense subset of $\mathfrak{u}_P$ is called the \textit{Richardson orbit} of $P$, and there exists a canonical map $\pi \colon T^*(G/P) \rightarrow \overline{\mathcal{O}}$ analogous to the Springer resolution.

Let $L$ be the Levi factor of $P$ that contains $T$.  From a weight $\lambda \in \Lambda$ dominant with respect to the Borel subgroup $B_L := B \cap L$ of $L$, we obtain an irreducible $L$-representation $W^{\lambda}$ with highest weight $\lambda$, which we may regard as a $P$-representation by stipulating that $U_P$ act trivially.  Then \[M_{\lambda} := G \times_P W^{\lambda} \rightarrow G/P\] is a $G$-equivariant vector bundle on $G/P$.  Pulling back its sheaf of sections $\mathcal{M}_{\lambda}$ to the cotangent bundle $T^*(G/P)$ along the canonical projection $p \colon T^*(G/P) \rightarrow G/P$, and then pushing the result forward onto $\mathcal{N}$, we end up with the complex \[A^P_{\lambda} := R\pi_* p^*\mathcal{M}_{\lambda} \in \mathfrak{D}.\]

Note that the Richardson orbit of $B$ is the \textit{regular nilpotent orbit} $\mathcal{O}^{\text{reg}}$, uniquely characterized by the property $\overline{\mathcal{O}^{\text{reg}}} = \mathcal{N}$.  The Levi factor of $B$ containing $T$ is $T$ itself.  Thus, for all $\lambda \in \Lambda$, the complex $A^B_{\lambda}$ is defined and coincides with $A_{\lambda}$, meaning that the above construction specializes to that of $\lbrace A_{\lambda} \rbrace_{\lambda \in \Lambda}$, as we claimed.  

\subsection{The nilpotent cone of $\mathfrak{gl}_n$}
Henceforward we set $G := GL_n(\mathbb{C})$.  Then $\mathfrak{g} = \mathfrak{gl}_n(\mathbb{C})$.  Let $X \in \mathfrak{g}$ be nilpotent.  The existence of the Jordan canonical form implies the existence of positive integers $\alpha_1 \geq \cdots \geq \alpha_{\ell}$ summing to $n$ and vectors $v_1, \ldots, v_{\ell}$ such that \[\mathbb{C}^n = \operatorname{span} \lbrace X^j v_i : 1 \leq i \leq \ell, 0 \leq j \leq \alpha_i -1 \rbrace \] and $X^{\alpha_i} v_i = 0$ for all $i$ (cf. Jantzen \cite{Jantzen}, section 1.1).

Express the partition $\alpha := [\alpha_1, \ldots, \alpha_{\ell}]$ in the form $[k_1^{a_1}, \ldots, k_m^{a_m}]$, where $k_1 > \cdots > k_m$ are the distinct parts of $\alpha$ and $a_t$ is the multiplicity of $k_t$ for all $1 \leq t \leq m$.  Let $V_t$ be the $a_t$-dimensional vector space spanned by the set $\lbrace v_i : \alpha_i = k_t \rbrace$.  Define a map \[\varphi_X \colon GL(V_1) \times \cdots \times GL(V_m) \rightarrow G_X\] by $\varphi_X(g_1, \ldots, g_m)(X^j v_i) := X^j g_t v_i$ for $v_i \in V_t$.  

Note that $\varphi_X$ is injective.  Let $G_X^{\text{red}}$ be the image of $\varphi_X$, and let $R_X$ be the unipotent radical of $G_X$.  From Jantzen \cite{Jantzen}, sections 3.8--3.10, we see that $G_X^{\text{red}}$ is reductive and $G_X = G_X^{\text{red}} R_X$.  Since $R_X$ acts trivially in any irreducible $G_X$-representation, specifying an irreducible representation of $G_X$ is equivalent to specifying an irreducible representation of $G_X^{\text{red}}$, which means specifying an irreducible representation of $GL_{a_1, \ldots, a_m}  := GL_{a_1} \times \cdots \times GL_{a_m}$.  

Let $\alpha^* = [\alpha^*_1, \ldots, \alpha^*_s]$ be the conjugate partition to $\alpha$, where $s := \alpha_1$.  For all $1 \leq j \leq s$, let $V(j)$ be the $\alpha^*_j$-dimensional vector space spanned by the set $\lbrace X^{\alpha_i - j} v_i : \alpha_i \geq j \rbrace$, and set $V^{(j)} := V(1) \oplus \cdots \oplus V(j)$.  

Define subgroups $L_X \subset P_X \subset G$ by \[P_X := \lbrace g \in G : g\big(V^{(j)}\big) = V^{(j)} \text{ for all } 1\leq j \leq s \rbrace\] and \[L_X := \lbrace g \in G : g\big(V(j)\big) = V(j) \text{ for all } 1 \leq j \leq s \rbrace.\]

Since $P_X$ is the stabilizer of the partial flag \[\lbrace 0 \rbrace \subset V^{(1)} \subset \cdots \subset V^{(s)} = \mathbb{C}^n,\] it follows immediately that $P_X \subset G$ is a parabolic subgroup and $L_X \subset P_X$ is a Levi factor.  Furthermore, the Richardson orbit of $P_X$ is none other than $\mathcal{O}_X$ (cf. Jantzen \cite{Jantzen}, section 4.9).  For general $G$, this implies that the connected component of the identity in $G_X$ is contained in $P_X$.  In our case $G = GL_n$, the conclusion is stronger: $G_X \subset P_X$, and $G_X^{\text{red}} \subset L_X$.  (That we could find $P_X$ so that $\mathcal{O}_X$ is its Richardson orbit is also due to the assumption that $G$ is of type $A$.)

The claim $G_X \subset P_X$ follows from the observation that $V^{(j)}$ is the kernel of $X^j$ for all $1 \leq j \leq s$.  To see $G_X^{\text{red}} \subset L_X$, we find a Levi subgroup of $L_X$ that contains $G_X^{\text{red}}$.  Since $X^{k_t - j}V_t \subset V(j)$, the direct sum decomposition \[\mathbb{C}^n = \bigoplus_{t = 1}^m \bigoplus_{j = 1}^{k_t} X^{k_t - j} V_t\] is a refinement of the decomposition $\mathbb{C}^n = \bigoplus_{j=1}^s V(j)$.  Set \[L_X^{\text{ref}} := \lbrace g \in G : g(X^{k_t-j} V_t) = X^{k_t-j} V_t \text{ for all } 1 \leq t \leq m, 1 \leq j \leq k_t \rbrace.\]  Then $L_X^{\text{ref}} \subset L_X$, and the inclusion $G_X^{\text{red}} \subset L_X^{\text{ref}}$ follows directly from the definition of $\varphi_X$.  

Let $\chi_X$ be the isomorphism \[L_X^{\text{ref}} \rightarrow \prod_{t = 1}^m \prod_{j=1}^{k_t} GL(X^{k_t -j} V_t)\] given by $g \mapsto \prod_{t=1}^m (g|_{X^{k_t-1} V_t}, \ldots, g|_{V_t})$, and let $\psi_X$ be the isomorphism \[L_X \rightarrow GL(V(1)) \times \cdots \times GL(V(s))\] given by $g \mapsto \left(g|_{V(1)}, \ldots, g|_{V(s)}\right)$.  

From the analysis above, we may conclude that the composition \[\psi_X \varphi_X \colon GL_{a_1, \ldots, a_m} \rightarrow GL_{\alpha^*_1, \ldots, \alpha^*_s}\] factors as the composition \[\chi_X \varphi_X \colon GL_{a_1, \ldots, a_m}  \rightarrow \prod_{t=1}^m (GL_{a_t})^{k_t}\] (which coincides with the product, over all $1 \leq t \leq m$, of the diagonal embeddings $GL_{a_t} \rightarrow (GL_{a_t})^{k_t}$), followed by the product, over all $1 \leq j \leq s$, of the inclusions $\prod_{t : k_t \geq j} GL_{a_t} \rightarrow GL_{\alpha^*_j}$.  This description of $\psi_X \varphi_X$ allows us to detect the appearance of certain $[IC_{(\mathcal{O}_X, V)}]$ classes in the expansion on the $\Omega$-basis of a complex arising via the resolution $T^*(G/P_X) \rightarrow \overline{\mathcal{O}_X}$ (cf. Lemma~\ref{omega}).  

\begin{exam} \label{colors}
	\setcounter{MaxMatrixCols}{20}
Set $n := 11$.  Then $G = GL_{11}$.  Set 
\[X := 
\begin{bmatrix}
	0 & 1 & 0 & 0 &  &  &  &  &  &  &  \\
	0 & 0 & 1 & 0 &  &  &  &  &  &  &  \\
	0 & 0 & 0 & 1 &  &  &  &  &  &  &  \\
	0 & 0 & 0 & 0 &  &  &  &  &  &  &  \\
	  &  &  &  & 0 & 1 & 0 &  &  &  &  \\
	  &  &  &  & 0 & 0 & 1 &  &  &  &  \\
	  &  &  &  & 0 & 0 & 0 &  &  &  &  \\
	  &  &  &  &  &  &  & 0 & 1 &  &  \\
	  &  &  &  &  &  &  & 0 & 0 &  &  \\
   	  &  &  &  &  &  &  &  &  & 0 &  \\
	  &  &  &  &  &  &  &  &  &  & 0 
\end{bmatrix}.\]

The partition encoding the sizes of the Jordan blocks of $X$ is $\alpha = [4,3,2,1,1]$.  The Young diagram of $\alpha$ is depicted in Figure~\ref{rowleng}.  Each Jordan block of $X$ corresponds to a row of $\alpha$.  

\begin{figure}[h]
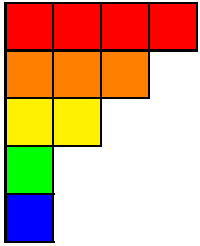

	\begin{ytableau}
		*(red) & *(red) & *(red) & *(red) \\
		*(orange) & *(orange) & *(orange) \\
		*(yellow) & *(yellow)\\
		*(green)\\
		*(blue)
	\end{ytableau}
	\caption{The Young diagram of $\alpha$, colored by rows}
	\label{rowleng}

\end{figure}

We may express $\alpha$ in the form $[4^1,3^1,2^1,1^2]$, where $4>3>2>1$ are the distinct parts of $\alpha$.  Then $G_X^{\text{red}}$ is the image under the isomorphism $\varphi_X$ of \[ GL_1 \times GL_1 \times GL_1 \times GL_2.\]  Each factor of the preimage corresponds to a distinct part of $\alpha$ (cf. Figure~\ref{rowmult}).  

\begin{figure}[h]
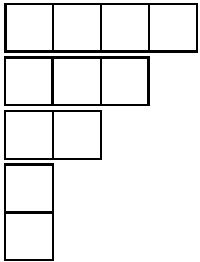

	\ytableausetup{centertableaux}
	\begin{ytableau}
		{} &  &  &  \\ 
	\end{ytableau}
	\\ \vspace{0.01in} \hspace{-0.31in}
	\begin{ytableau}
		{} &  &  \\
	\end{ytableau}
	\\ \vspace{0.01in} \hspace{-0.56in}
	\begin{ytableau}
		{} & \\
	\end{ytableau}
	\\ \vspace{0.01in} \hspace{-0.81in}
	\begin{ytableau}
		{} \\
		{}
	\end{ytableau}
\caption{The Young diagram of $\alpha$, partitioned by distinct parts}  \label{rowmult}

\end{figure}

The conjugate partition $\alpha^*$ is $[5,3,2,1]$.  The isomorphism $\psi_X$ maps $L_X$ onto \[GL_5 \times GL_3 \times GL_2 \times GL_1.\]  Each factor of the image corresponds to a column of $\alpha$ (cf. Figure~\ref{col}).  

\begin{figure}[h]
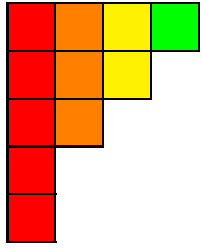

	\begin{ytableau}
	*(red) & *(orange) & *(yellow) & *(green) \\
	*(red) & *(orange) & *(yellow) \\
	*(red) & *(orange)\\
	*(red)\\
	*(red)
\end{ytableau}
	\caption{The Young diagram of $\alpha$, colored by columns} \label{col}
\end{figure} 

The group $L_X^{\text{ref}}$ lies inside $L_X$ and contains $G_X^{\text{red}}$.  The isomorphism $\chi_X$ maps $L_X^{\text{ref}}$ isomorphically onto \[(GL_1)^4 \times (GL_1)^3 \times (GL_1)^2 \times (GL_2)^1.\]  Each factor of the image corresponds to an ordered pair consisting of a distinct part of $\alpha$ \textit{and} a column of $\alpha$ (cf. Figure~\ref{both}).  

\begin{figure}[h]
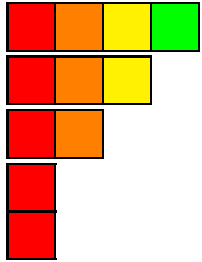

	\ytableausetup{centertableaux}
	\begin{ytableau}
		*(red) & *(orange) & *(yellow) & *(green) \\ 
	\end{ytableau}
\\ \vspace{0.01in} \hspace{-0.305in}
	\begin{ytableau}
		*(red) & *(orange) & *(yellow) \\
	\end{ytableau}
\\ \vspace{0.01in} \hspace{-0.555in}
	\begin{ytableau}
		*(red) & *(orange)\\
	\end{ytableau}
\\ \vspace{0.01in} \hspace{-0.805in}
	\begin{ytableau}
		*(red)\\
		*(red)
	\end{ytableau}
	\caption{The Young diagram of $\alpha$, partitioned by distinct parts and colored by columns} \label{both}
\end{figure} 

The composition \[\psi_X \varphi_X \colon GL_{1,1,1,2} \rightarrow GL_{5,3,2,1}\] factors as the product of diagonal embeddings \[\chi_X \varphi_X \colon GL_{1,1,1,2} \rightarrow (GL_1)^4 \times (GL_1)^3 \times (GL_1)^2 \times (GL_2)^1,\] followed by the product of the inclusions \[
GL_{1,1,1,2} \rightarrow GL_5, \quad
GL_{1,1,1} \rightarrow GL_3, \quad
GL_{1,1} \rightarrow GL_2, \quad \text{and} \quad
GL_1 \rightarrow GL_1.\]
\end{exam}

\subsection{Sheaves on the nilpotent cone of $\mathfrak{gl}_n$}
Let $e_1, \ldots, e_n$ be the standard basis for $\mathbb{C}^n$.  From the nilpotent orbit $\mathcal{O}_X$, we choose the representative element $X_{\alpha} \in \mathfrak{g}$ given by \[e_i \mapsto 0\] for all $1 \leq i \leq \alpha^*_1$ and \[e_{\alpha^*_1 + \cdots + \alpha^*_{j-1} + i} \mapsto e_{\alpha^*_1 + \cdots + \alpha^*_{j-2} + i}\] for all $2 \leq j \leq s$, $1 \leq i \leq \alpha^*_j$.

\begin{exam}
Maintain the notation of Example~\ref{colors}.  Then \[X_{\alpha} = \begin{bmatrix}
	0 & 0 & 0 & 0 & 0 & 1 & 0 & 0 &  &  &  \\
	0 & 0 & 0 & 0 & 0 & 0 & 1 & 0 &  &  &  \\
	0 & 0 & 0 & 0 & 0 & 0 & 0 & 1 &  &  &  \\
	0 & 0 & 0 & 0 & 0 & 0 & 0 & 0 &  &  &  \\
	0 & 0 & 0 & 0 & 0 & 0 & 0 & 0 &  &  &  \\
	 &  &  &  &  & 0 & 0 & 0 & 1 & 0 &  \\
	 &  &  &  &  & 0 & 0 & 0 & 0 & 1 &  \\
	 &  &  &  &  & 0 & 0 & 0 & 0 & 0 &  \\
	 &  &  &  &  &  &  &  & 0 & 0 & 1 \\
	 &  &  &  &  &  &  &  & 0 & 0 & 0 \\
	 &  &  &  &  &  &  &  &  &  & 0 
	\end{bmatrix}.\]
\end{exam}

To see that $X_{\alpha} \in \mathcal{O}_X$, let $g \in G$ be given by $X^{\alpha_i - j} v_i \mapsto e_{\alpha^*_1 + \cdots + \alpha^*_{j-1} + i}$, and observe that $X_{\alpha} = gXg^{-1}$.  Thus, $\mathcal{N} = \bigcup_{\alpha \vdash n} \mathcal{O}_{X_{\alpha}}$.  For $\alpha$ a partition of $n$, we write $\mathcal{O}_{\alpha}$ for the orbit $\mathcal{O}_{X_{\alpha}}$.  The uniqueness of the Jordan canonical form implies that the orbits $\mathcal{O}_{\alpha}$ and $\mathcal{O}_{\beta}$ are disjoint for distinct partitions $\alpha$ and $\beta$, so $\lbrace \mathcal{O}_{\alpha} \rbrace_{\alpha \vdash n}$ constitutes the set of nilpotent orbits of $\mathfrak{g}$.   

For each factor $GL_{a_t}$ of $GL_{a_1, \ldots, a_m}$, we identify the weight lattice with the character lattice $\mathbb{Z}^{a_t}$ of the maximal torus $(\mathbb{C}^*)^{a_t}$ of invertible diagonal matrices, and we assign the partial order induced by the Borel subgroup of invertible upper triangular matrices.  Then the isomorphism classes of irreducible $GL_{a_1, \ldots, a_m}$-representations are indexed by $m$-tuples of integer sequences $(\nu(1), \ldots, \nu(m))$ such that $\nu(t)$ is a dominant weight of $GL_{a_t}$ for all $1 \leq t \leq m$.  The $m$-tuple $(\nu(1), \ldots, \nu(m))$ corresponds to the representation \[V^{(\nu(1), \ldots, \nu(m))} := V^{\nu(1)} \boxtimes \cdots \boxtimes V^{\nu(m)},\] where $V^{\nu(t)}$ denotes the irreducible $GL_{a_t}$-representation with highest weight $\nu(t)$.  

We say that an integer sequence $\nu = [\nu_1, \ldots, \nu_{\ell}]$ is \textit{dominant} with respect to $\alpha$ if $\alpha_i = \alpha_{i+1}$ implies $\nu_i \geq \nu_{i+1}$.  Note that the dominance condition holds precisely when $\nu$ is the concatenation of an $m$-tuple of dominant weights $(\nu(1), \ldots, \nu(m))$.  For such $\nu$, we denote by $V^{(\alpha, \nu)}$ the representation of $G_{\alpha} := G_{X_{\alpha}}$ (or of $G_{\alpha}^{\text{red}} := G_{X_{\alpha}}^{\text{red}}$, depending on context) arising from the representation $V^{(\nu(1), \ldots, \nu(m))}$ of $GL_{a_1, \ldots, a_m}$.  Via the association \[(\alpha, \nu) \mapsto \big(\mathcal{O}_{\alpha}, V^{(\alpha, \nu)}\big),\] we construe $\Omega$ as consisting of pairs of integer sequences $(\alpha, \nu)$ such that $\alpha = [\alpha_1, \ldots, \alpha_{\ell}]$ is a partition of $n$ and $\nu = [\nu_1, \ldots, \nu_{\ell}]$ is dominant with respect to $\alpha$.  

Let $B \subset G$ be the Borel subgroup of invertible upper triangular matrices, and let $T \subset B$ be the maximal torus of invertible diagonal matrices.  The weight lattice $\Lambda = \operatorname{Hom}(T, \mathbb{C}^*) \cong \mathbb{Z}^n$ comprises length-$n$ integer sequences $\lambda = [\lambda_1, \ldots, \lambda_n]$.  Those weights $\lambda \in \Lambda$ which are weakly decreasing are dominant with respect to $B$ and belong to $\Lambda^+$.  

Set $P_{\alpha} := P_{X_{\alpha}}$ and $L_{\alpha} := L_{X_{\alpha}}$.  Then \[P_{\alpha} = \left\lbrace g \in G : ge_{\alpha^*_1 + \cdots + \alpha^*_{j-1} + i} \in \operatorname{span} \lbrace e_1, \ldots, e_{\alpha^*_1 + \cdots + \alpha^*_j} \rbrace \right\rbrace \] and \[L_{\alpha} = \left\lbrace g \in G : ge_{\alpha^*_1 + \cdots + \alpha^*_{j-1} + i} \in \operatorname{span} \lbrace e_{\alpha^*_1 + \cdots + \alpha^*_{j-1} + 1}, \ldots, e_{\alpha^*_1 + \cdots + \alpha^*_j} \rbrace \right\rbrace. \]

We see immediately that $P_{\alpha} \supset B$ and $L_{\alpha} \supset T$.  Thus, $\Lambda$ doubles as the weight lattice of $L_{\alpha}$.  Given a weight $\lambda \in \Lambda$, let $\lambda^j$ be its restriction to the factor $GL_{\alpha^*_j}$ of $L_{\alpha} \cong GL_{\alpha^*_1, \ldots, \alpha^*_s}$.  This realizes $\lambda$ as the concatenation of the $s$-tuple of integer sequences $(\lambda^1, \ldots, \lambda^s)$.  If $\lambda^j$ is weakly decreasing for all $1 \leq j \leq s$, then $\lambda$ is dominant with respect to the Borel subgroup $B_{\alpha} := B_{L_{\alpha}}$, in which case $\lambda$ belongs to $\Lambda^+_{\alpha}$, the set of dominant weights of $L_{\alpha}$ with respect to $B_{\alpha}$.  For $\lambda \in \Lambda^+_{\alpha}$, we denote by $W^{\lambda^j}$ the irreducible $GL_{\alpha^*_j}$-representation with highest weight $\lambda^j$, and we set \[W^{\lambda} := W^{\lambda^1} \boxtimes \cdots \boxtimes W^{\lambda^s},\] which indeed has highest weight $\lambda$.  

We rely on the complexes $A^{\alpha}_{\lambda} := A^{P_{\alpha}}_{\lambda}$ associated to weights $\lambda \in \Lambda^+_{\alpha}$ to interpolate between the $\Omega$- and $\Lambda^+$-bases for $K_0(\mathfrak{D})$.  Weights of $L_{\alpha}$ are also weights of $G$, so it is reasonable to expect that the expansion of $[A^{\alpha}_{\lambda}]$ on the $\Lambda^+$-basis be easy to compute.  On the other hand, representations of $L_{\alpha}$ restrict to representations of $G_{\alpha}^{\text{red}}$, and it turns out that this relationship lifts to the corresponding objects in $\mathfrak{D}$.  The following results of Achar \cite{Acharj} encapsulate these statements formally.  

\begin{lem}[Achar \cite{Acharj}, Corollary 2.5] \label{omega} 
Let $(\alpha, \nu) \in \Omega$, and let $\lambda \in \Lambda^+_{\alpha}$.  Suppose that $V^{(\alpha, \nu)}$ occurs in the decomposition of the $L_{\alpha}$-representation $W^{\lambda}$ as a direct sum of irreducible $G_{\alpha}^{\emph{red}}$-representations.  Then, when $[A^{\alpha}_{\lambda}]$ is expanded on the $\Omega$-basis for $K_0(\mathfrak{D})$, the coefficient of $[IC_{(\alpha, \nu)}]$ is nonzero.  
\end{lem}

\begin{lem}[Achar \cite{Acharj}, Corollary 2.7] \label{lambda}
Let $W_{\alpha}$ be the Weyl group of $L_{\alpha}$, and let $\rho_{\alpha}$ be the half-sum of the positive roots of $L_{\alpha}$.  For all $\lambda \in \Lambda^+_{\alpha}$, the following equality holds in $K_0(\mathfrak{D})$: \[[A^{\alpha}_{\lambda}] = \sum_{w \in W_{\alpha}} (-1)^w [A_{\lambda + \rho_{\alpha} - w \rho_{\alpha}}].\]
\end{lem}

Let $W$ be the Weyl group of $G$, and, for all $\mu \in \Lambda$, let $\operatorname{dom}(\mu) \in \Lambda^+$ be the unique dominant weight in the $W$-orbit of $\mu$.  When $[A_{\mu}]$ is expanded on the $\Lambda^+$-basis for $K_0(\mathfrak{D})$, the coefficient of $[A_{\lambda}]$ is zero unless $\lambda \leq \operatorname{dom}(\mu)$ (cf. Achar \cite{Acharj}, Proposition 2.2).  Thus, if $\mu \in \Lambda^+_{\alpha}$, it follows from Lemma~\ref{lambda} that $[A^{\alpha}_{\mu}] \in \lbrace  \operatorname{span} [A_{\lambda}] : \lambda \leq \operatorname{dom}(\mu + 2 \rho_{\alpha}) \rbrace$.  

Let $\Omega_{\alpha}$ be the set of all dominant integer sequences $\nu$ with respect to $\alpha$.  Given $\nu \in \Omega_{\alpha}$, set \[\Lambda^+_{\alpha, \nu} := \left \lbrace \mu \in \Lambda^+_{\alpha} : \dim \operatorname{Hom}_{G_{\alpha}^{\text{red}}} \big(V^{(\alpha, \nu)}, W^{\mu}\big) > 0 \right \rbrace.\]  On input $(\alpha, \nu)$, our algorithm finds a weight $\mu \in \Lambda^+_{\alpha, \nu}$ such that $||\mu + 2 \rho_{\alpha}||$ is minimal.  As demonstrated by Achar \cite{Achart, Acharj}, this guarantees that $\gamma(\alpha, \nu) = \operatorname{dom}(\mu + 2 \rho_{\alpha})$.\footnote{This follows from Claim 2.3.1 in \cite{Achart}, except that $\gamma$ is defined differently.  In \cite{Acharj}, Theorem 8.10, Achar shows that the bijection $\gamma$ constructed in \cite{Achart} coincides with the bijection in Theorem~\ref{bez}.  }

The intuition behind this approach is straightforward.  For all $\mu \in \Lambda^+_{\alpha, \nu}$, the expansion of $[A^{\alpha}_{\mu}]$ on the $\Omega$-basis takes the form \[\big[A^{\alpha}_{\mu}\big] = \dim \operatorname{Hom}_{G_{\alpha}^{\text{red}}} \big(V^{(\alpha, \nu)}, W^{\mu}\big) \big[IC_{(\alpha, \nu)}\big] + \sum_{\upsilon \in \Omega_{\alpha} : \upsilon \neq \nu} c_{\alpha, \upsilon} \big[IC_{(\alpha, \upsilon)}\big] + \sum_{(\beta, \xi) \in \Omega : \beta \vartriangleleft \alpha} c_{\beta, \xi} \big[IC_{(\beta, \xi)}\big],\] where $\trianglelefteq$ denotes the dominance order on partitions of $n$.  On the other hand, the expansion of $[A^{\alpha}_{\mu}]$ on the $\Lambda^+$-basis takes the form \[\big[A^{\alpha}_{\mu}\big] = \pm \big[A_{\operatorname{dom}(\mu + 2 \rho_{\alpha})}\big] + \sum_{\lambda < \operatorname{dom}(\mu + 2 \rho_{\alpha})} c_{\lambda} \big[A_{\lambda} \big].\]

We compare the equations.  There is a single maximal-weight term in the right-hand side of the second equation.  It follows that there is a single maximal-weight term in the expansion of the right-hand side of the first equation on the $\Lambda^+$-basis.  By Theorem~\ref{bez}, the maximal weight must be $\gamma(\alpha, \nu)$ or among the sets $\lbrace \gamma(\alpha, \upsilon) : \upsilon \neq \nu \rbrace$ and $\lbrace \gamma(\beta, \xi) : \beta \vartriangleleft \alpha \rbrace$.  In the former case, we may conclude immediately that $\gamma(\alpha, \nu) = \operatorname{dom}(\mu + 2 \rho_{\alpha})$.  It turns out that mandating the minimality of $||\mu + 2 \rho_{\alpha}||$ suffices to preclude the latter possibility.  

\subsection{The Lusztig--Vogan bijection for $GL_2$}
Set $n := 2$.  Then $G = GL_2$.  The weight lattice $\Lambda$ comprises ordered pairs $[\lambda_1, \lambda_2] \in \mathbb{Z}^2$, and $\Lambda^+ = \lbrace [\lambda_1, \lambda_2] \in \mathbb{Z}^2 : \lambda_1 \geq \lambda_2 \rbrace$.  

The variety $\mathcal{N} \subset \mathfrak{g}$ is the zero locus of the determinant polynomial.  Each matrix of rank $1$ in $\mathfrak{g}$ is similar to $\begin{bmatrix} 0 & 1 \\ 0 & 0 \end{bmatrix}$, so $\mathcal{N}$ is the union of $\begin{bmatrix} 0 & 0 \\ 0 & 0\end{bmatrix}$ (the \textit{zero orbit}) and the $G$-orbit of $\begin{bmatrix} 0 & 1 \\ 0 & 0 \end{bmatrix}$ (the \textit{regular orbit}).  

To the zero orbit corresponds the partition $[1,1]$.  Note that $G_{[1,1]}^{\text{red}} = L_{[1,1]} = G$.  Hence \[\Omega_{[1,1]} = \lbrace [\nu_1, \nu_2] \in \mathbb{Z}^2 : \nu_1 \geq \nu_2 \rbrace \quad \text{and} \quad \Lambda^+_{[1,1]} = \lbrace [\mu_1, \mu_2] \in \mathbb{Z}^2 : \mu_1 \geq \mu_2 \rbrace.\]

For all $[\mu_1, \mu_2] \in \Lambda^+_{[1,1]}$, the irreducible $L_{[1,1]}$-representation $W^{[\mu_1, \mu_2]}$ is isomorphic as a $G_{[1,1]}^{\text{red}}$-representation to $V^{([1,1], [\mu_1, \mu_2])}$.  Thus, for all $[\nu_1, \nu_2] \in \Omega_{[1,1]}$, \[\Lambda^+_{[1,1], [\nu_1, \nu_2]} = \lbrace [\nu_1, \nu_2] \rbrace.\]  Our algorithm sets $[\mu_1, \mu_2] := [\nu_1, \nu_2]$.  

On the $\Omega$-basis, $[A^{[1,1]}_{[\mu_1, \mu_2]}]$ expands as \[ \left[A^{[1,1]}_{[\nu_1, \nu_2]}\right] = \left[IC_{([1,1], [\nu_1, \nu_2])}\right].\]

Since $W_{[1,1]} = W = \mathfrak{S}_2$ and $\rho_{[1,1]} = [\frac{1}{2}, - \frac{1}{2}]$, it follows that $[A^{[1,1]}_{[\mu_1, \mu_2]}]$ expands on the $\Lambda^+$-basis as \[\left[A^{[1,1]}_{[\nu_1, \nu_2]}\right] = - \left[A_{[\nu_1 + 1, \nu_2 - 1]}\right] + \left[A_{[\nu_1, \nu_2]}\right].\]

Hence $\gamma([1,1], [\nu_1, \nu_2]) = [\nu_1 + 1, \nu_2 - 1] = \operatorname{dom}([\mu_1, \mu_2] + 2 \rho_{[1,1]})$, which confirms that the output is correct.  

We turn our attention to the regular orbit, to which corresponds the partition $[2]$.  Recall that $G_{[2]}^{\text{red}} \cong GL_1$ and $L_{[2]} \cong GL_1 \times GL_1$.  Hence \[\Omega_{[2]} = \lbrace [\nu_1] \in \mathbb{Z}^1 \rbrace \quad \text{and} \quad \Lambda^+_{[2]} = \lbrace [\mu_1, \mu_2] \in \mathbb{Z}^2 \rbrace.\]  

Furthermore, the composition $\psi_{X_{[2]}} \varphi_{X_{[2]}}$ of the isomorphisms $\varphi_{X_{[2]}} \colon GL_1 \rightarrow G_{[2]}^{\text{red}}$ and $\psi_{X_{[2]}} \colon L_{[2]} \rightarrow GL_1 \times GL_1$ coincides with the diagonal embedding $GL_1 \rightarrow GL_1 \times GL_1$.  For all $[\mu_1, \mu_2] \in \Lambda^+_{[2]}$, the irreducible $L_{[2]}$-representation $W^{[\mu_1, \mu_2]}$ is isomorphic as a $G_{[2]}^{\text{red}}$-representation to $V^{([2], [\mu_1 + \mu_2])}$.  

Thus, for all $[\nu_1] \in \Omega_{[2]}$, \[\Lambda^+_{[2], [\nu_1]} = \lbrace [\mu_1, \mu_2] \in \Lambda^+_{[2]} : \mu_1 + \mu_2 = \nu_1 \rbrace.\]  Our algorithm sets \[\left[\mu_1, \mu_2 \right] := \left[ \left \lceil \frac{\nu_1}{2} \right \rceil, \left \lfloor \frac{\nu_1}{2} \right \rfloor \right].\] 

On the $\Omega$-basis, $[A^{[2]}_{[\mu_1, \mu_2]}]$ expands as \[ \left[A^{[2]}_{\left[ \left \lceil \frac{\nu_1}{2} \right \rceil, \left \lfloor \frac{\nu_1}{2} \right \rfloor \right]} \right] = \left[IC_{([2], [\nu_1])}\right] + \sum_{[\xi_1, \xi_2] \in \Omega_{[1,1]}} c_{[1,1], [\xi_1, \xi_2]}\left[IC_{([1,1], [\xi_1, \xi_2])}\right].\]

Since $W_{[2]}$ is trivial and $\rho_{[2]} = [0, 0]$, it follows that $[A^{[2]}_{[\mu_1, \mu_2]}]$ expands on the $\Lambda^+$-basis as \[\left[A^{[2]}_{\left[ \left \lceil \frac{\nu_1}{2} \right \rceil, \left \lfloor \frac{\nu_1}{2} \right \rfloor \right]} \right] = \left[A_{\left[ \left \lceil \frac{\nu_1}{2} \right \rceil, \left \lfloor \frac{\nu_1}{2} \right \rfloor \right]} \right].\]  

From our analysis above, we know that $\gamma([1,1], [\xi_1, \xi_2]) = [\xi_1 + 1, \xi_2 - 1]$, so there cannot exist $[\xi_1, \xi_2] \in \Omega_{[1,1]}$ such that $\gamma([1,1], [\xi_1, \xi_2]) = [\lceil \frac{\nu_1}{2} \rceil, \lfloor \frac{\nu_1}{2} \rfloor]$.  

Hence $\gamma([2], [\nu_1]) = [\lceil \frac{\nu_1}{2} \rceil, \lfloor \frac{\nu_1}{2} \rfloor] = \operatorname{dom}([\mu_1, \mu_2] + 2 \rho_{[2]})$.\footnote{It follows immediately from Theorem~\ref{bez} that $c_{[1,1],[\xi_1, \xi_2]} = 0$ for all $[\xi_1, \xi_2] \in \Omega_{[1,1]}$.  }  

\subsection{Outline}
The cynosure of this article is the \textit{integer-sequences version} of our algorithm, which admits as input a pair $(\alpha, \nu) \in \Omega$ and yields as output a weight $\mathfrak{A}(\alpha, \nu) \in \Lambda^+_{\alpha}$.  The output, which consists of a weight of each factor of $L_{\alpha}$, is obtained recursively: The weight of the first factor $GL_{\alpha^*_1}$ is computed; then the input is adjusted accordingly, and the algorithm is called on the residual input to determine the weight of each of the remaining factors.  

The algorithm design is guided by the objective of locating $\mathfrak{A}(\alpha, \nu)$ in $\Lambda^+_{\alpha, \nu}$ and keeping $||\mathfrak{A}(\alpha, \nu) + 2 \rho_{\alpha}||$ as small as possible.  Our main theorem is the following.  

\begin{thm} \label{main}
Let $(\alpha, \nu) \in \Omega$.  Then $\gamma(\alpha, \nu) = \operatorname{dom}(\mathfrak{A}(\alpha, \nu) + 2 \rho_{\alpha})$.  
\end{thm}

We prove the main theorem by verifying that \[||\mathfrak{A}(\alpha, \nu) + 2 \rho_{\alpha}|| = \min \lbrace ||\mu + 2 \rho_{\alpha}|| : \mu \in \Lambda^+_{\alpha, \nu} \rbrace.\]  However, our approach is indirect and relies on a combinatorial apparatus introduced by Achar \cite{Achart, Acharj} --- \textit{weight diagrams}.  

A weight diagram $X$ of shape-class $\alpha$ encodes several integer sequences, including a weight $h(X) \in \Lambda^+_{\alpha}$.  On input $(\alpha, \nu)$, Achar's algorithm outputs a weight diagram $\mathsf{A}(\alpha, \nu)$ of shape-class $\alpha$ such that $h \mathsf{A}(\alpha, \nu) \in \Lambda^+_{\alpha, \nu}$ and $||h \mathsf{A}(\alpha, \nu) + 2 \rho_{\alpha}||$ is minimal (cf. \cite{Acharj}, Corollary 8.9).  Achar's conclusion (cf. \cite{Acharj}, Theorem 8.10) is that Theorem~\ref{main} holds with $h \mathsf{A}(\alpha, \nu)$ in place of $\mathfrak{A}(\alpha, \nu)$.  

The minimality of $||h \mathsf{A}(\alpha, \nu) + 2 \rho_{\alpha}||$ is basic to Achar's algorithm, which maintains a candidate output $X$ at each step, and performs only manipulations that do not increase $||hX + 2 \rho_{\alpha}||$.  In contrast, $\mathfrak{A}(\alpha, \nu)$ is computed one entry at a time.  The minimality of $||\mathfrak{A}(\alpha, \nu) + 2 \rho_{\alpha}||$ is an emergent property, which we prove by comparison of our algorithm with Achar's.  

Rather than attempt to connect $\mathfrak{A}$ to $\mathsf{A}$, we introduce a third algorithm $\mathcal{A}$, built with the same tools as $\mathfrak{A}$, but configured to output weight diagrams rather than integer sequences.\footnote{$\mathcal{A}$ actually outputs pairs of weight diagrams, so what we refer to in the introduction as $\mathcal{A}(\alpha, \nu)$ is denoted in the body by $p_1 \mathcal{A}(\alpha, \nu)$.  }  The relationship between this \textit{weight-diagrams version} and Achar's algorithm is impossible to miss: $\mathcal{A}(\alpha, \nu)$ always exactly matches $\mathsf{A}(\alpha, \nu)$.  Hence $||h\mathcal{A}(\alpha, \nu) + 2 \rho_{\alpha}||$ is minimal.  

While it is not the case that $\mathfrak{A}(\alpha, \nu)$ always coincides with $h\mathcal{A}(\alpha, \nu)$,\footnote{In the author's thesis \cite{Rush}, the integer-sequences version $\mathfrak{A}$ is defined so that $\mathfrak{A}(\alpha, \nu) = h \mathcal{A}(\alpha, \nu)$, but the proof that this equation holds is laborious and not altogether enlightening (cf. Chapter 5).  Relaxing this requirement allows us to simplify the definition of $\mathfrak{A}$ and focus on proofs more pertinent to $\gamma$.  } we show nonetheless that 
\begin{equation} \label{minimality}
||\mathfrak{A}(\alpha, \nu) + 2 \rho_{\alpha}|| = ||h \mathcal{A}(\alpha, \nu) + 2 \rho_{\alpha}||,
\end{equation}
which implies that \[\operatorname{dom}(\mathfrak{A}(\alpha, \nu) + 2 \rho_{\alpha}) = \operatorname{dom}(h \mathcal{A}(\alpha, \nu) + 2 \rho_{\alpha}),\] confirming that $\mathfrak{A}$ is a bona fide version of $\mathcal{A}$.  The main theorem follows immediately.   

In summary, the algorithm $\mathfrak{A}$ is a bee-line for computing $\gamma$, akin to an ansatz, which works because $\mathfrak{A}(\alpha, \nu) \in \Lambda^+_{\alpha, \nu}$ such that $||\mathfrak{A}(\alpha, \nu) + 2 \rho_{\alpha}||$ is minimal.  The minimality of $||\mathfrak{A}(\alpha, \nu) + 2 \rho_{\alpha}||$ is a consequence of the minimality of $||h \mathcal{A}(\alpha, \nu) + 2 \rho_{\alpha}||$, and we deduce the latter by identifying $\mathcal{A}(\alpha, \nu)$ with $\mathsf{A}(\alpha, \nu)$.  

The rest of this article is organized as follows.  In section 2, we present the integer-sequences version of our algorithm, along with several example calculations.  

In section 3, we define weight diagrams.  A weight diagram of shape-class $\alpha$ encodes an element each of $\Omega_{\alpha}$, $\Lambda^+_{\alpha}$, and $\Lambda^+$, and we give a correct proof of Proposition 4.4 in Achar \cite{Acharj} regarding the relations between the corresponding objects in $\mathfrak{D}$.  

In section 4, we present the weight-diagrams version of our algorithm and delineate its basic properties.  Then we prove Equation~\ref{minimality} holds, assuming that $||h \mathcal{A}(\alpha, \nu) + 2 \rho_{\alpha}||$ is minimal.

In section 5, we state Achar's criteria for a weight diagram to be \textit{distinguished}, and we prove that $\mathcal{A}$ outputs a distinguished diagram on any input.  As we explain, this implies that the diagrams $\mathcal{A}(\alpha, \nu)$ and $\mathsf{A}(\alpha, \nu)$ are identical for all $(\alpha, \nu) \in \Omega$.  

Finally, in the appendix, we cite Achar's algorithm for $\gamma^{-1}$ as heuristic evidence that our algorithm for $\gamma$ is the conceptually correct counterpart.  Achar's algorithm for $\gamma$ does not parallel his algorithm for $\gamma^{-1}$, but ours does.  
\vfill \eject

\section{The Algorithm, Integer-Sequences Version}

\subsection{Overview}
Fix a partition $\alpha = [\alpha_1, \ldots, \alpha_{\ell}]$ with conjugate partition $\alpha^* = [\alpha^*_1, \ldots, \alpha^*_s]$.  Given an integer sequence $\iota$ of any length, let $\operatorname{dom}(\iota)$ be the sequence obtained by rearranging the entries of $\iota$ in weakly decreasing order.  (This is consistent with the notation of section 1.3, for $\operatorname{dom}(\iota) \in W \iota \cap \Lambda^+$ if $\iota \in \Lambda$.)  

Let $\nu \in \Omega_{\alpha}$.  On input $(\alpha, \nu)$, our algorithm outputs an integer sequence $\mu$ of length $n$ satisfying the following conditions:
\begin{enumerate}
	\item $\mu$ is the concatenation of an $s$-tuple of weakly decreasing integer sequences $(\mu^1, \ldots, \mu^s)$ such that $\mu^j$ is of length $\alpha^*_j$ for all $1 \leq j \leq s$;
	\item There exists a collection of integers $\lbrace \nu_{i,j} \rbrace_{\substack{1 \leq i \leq \ell \\ 1 \leq j \leq \alpha_i}}$ such that \[\nu_i = \nu_{i, 1} + \cdots + \nu_{i, \alpha_i}\] for all $1 \leq i \leq \ell$ and $\mu^j = \operatorname{dom}([\nu_{1, j}, \ldots, \nu_{\alpha^*_j, j}])$ for all $1 \leq j \leq s$.  
\end{enumerate}
Recall that the first condition indicates $\mu \in \Lambda^+_{\alpha}$.  The second condition implies $\mu \in \Lambda^+_{\alpha, \nu}$ (cf. Corollary~\ref{decamp}).  

Although we could construct a collection $\lbrace \nu_{i,j} \rbrace_{\substack{1 \leq i \leq \ell \\ 1 \leq j \leq \alpha_i}}$ such that $\nu_i = \nu_{i,1} + \cdots + \nu_{i, \alpha_i}$ for all $i$ and obtain $\mu$ as a by-product (by setting $\mu^j := \operatorname{dom}([\nu_{1,j}, \ldots, \nu_{\alpha^*_j, j}])$ for all $j$), our algorithm instead computes each $\mu^j$ directly, alongside a permutation $\sigma^j \in \mathfrak{S}_{\alpha^*_j}$, so that $\nu_i = \mu^1_{\sigma^1(i)} + \cdots + \mu^{\alpha_i}_{\sigma^{\alpha_i}(i)}$ for all $i$.  (Then a collection fit to $\mu$ is given by $\nu_{i,j} := \mu^j_{\sigma^j(i)}$.)  

\begin{rem} \label{motiv}
Were we seeking to minimize $||\mu||$, it would suffice to choose, for all $i$, integers $\nu_{i, 1}, \ldots, \nu_{i, \alpha_i} \in \lbrace \lceil \frac{\nu_i}{\alpha_i} \rceil, \lfloor \frac{\nu_i}{\alpha_i} \rfloor \rbrace$ summing to $\nu_i$, and let the collection $\lbrace \nu_{i,j} \rbrace_{\substack{1 \leq i \leq \ell \\ 1 \leq j \leq \alpha_i}}$ induce the output $\mu$.  

However, our task is to minimize $||\mu + 2 \rho_{\alpha}||$, in which case we cannot confine each $\nu_{i,j}$ to the set $\lbrace \lceil \frac{\nu_i}{\alpha_i} \rceil, \lfloor \frac{\nu_i}{\alpha_i} \rfloor \rbrace$.\footnote{See section 2.4 for an example in which there exists $i,j$ such that $\nu_{i,j}$ must not belong to $\lbrace \lceil \frac{\nu_i}{\alpha_i} \rceil, \lfloor \frac{\nu_i}{\alpha_i} \rfloor \rbrace$.  }  Specifying the collection $\lbrace \nu_{i,j} \rbrace_{\substack{1 \leq i \leq \ell \\ 1 \leq j \leq \alpha_i}}$ straightaway, and learning the (numerical) order of the entries in each sequence $[\nu_{1, j}, \ldots, \nu_{\alpha^*_j, j}]$ post hoc, risks needlessly inflating \[\sum_{j=1}^s \left|\left|\operatorname{dom}([\nu_{1,j}, \ldots, \nu_{\alpha^*_j, j}]) + 2 \left[\frac{\alpha^*_j - 1}{2}, \ldots, \frac{1 - \alpha^*_j}{2} \right]\right|\right|^2 = ||\mu + 2 \rho_{\alpha}||^2.\]  

But how can we know what the order among the integers $\nu_{1,j}, \ldots, \nu_{\alpha^*_j, j}$ will be before their values are assigned?  Our answer is simply to stipulate the order, and pick values pursuant thereto --- by deciding $\sigma^j$, then $\mu^j$, and setting $[\nu_{1, j}, \ldots, \nu_{\alpha^*_j, j}] := [\mu^j_{\sigma^j(1)}, \ldots, \mu^j_{\sigma^j(\alpha^*_j)}]$.  
\end{rem} 

The algorithm runs by recursion.  Roughly: $\sigma^1$ is determined via a \textit{ranking} function, which compares \textit{candidate ceilings}, each measuring how the addition of $2 \rho_{\alpha}$ to $\mu$ might affect a subset of the collection $\lbrace \nu_{i,j} \rbrace_{\substack{1 \leq i \leq \ell \\ 1 \leq j \leq \alpha_i}}$, subject to a hypothesis about $\sigma^1$.  After $\sigma^1$ is settled, the corresponding candidate ceilings are tweaked (under the aegis of a \textit{column} function) to compute $\mu^1$.  Then $\mu^1$ is ``subtracted off,'' and the algorithm is called on the residual input $\nu'$, defined by $\nu'_i :=  \nu_i - \mu^1_{\sigma^1(i)}$, returning $\mu^2, \ldots, \mu^s$. 

\subsection{The algorithm}
Describing the algorithm explicitly requires us to introduce formally several preliminary functions.  
\begin{df}
Given a pair of integer sequences $(\alpha, \nu) \in \mathbb{N}^{\ell} \times \mathbb{Z}^{\ell}$, an integer $i \in \lbrace 1, \ldots, \ell \rbrace$, and an ordered pair of disjoint sets $(I_a, I_b)$ satisfying $I_a \cup I_b = \lbrace 1, \ldots, \ell \rbrace \setminus \lbrace i \rbrace$, we define the \textit{candidate-ceiling} function $\mathcal{C}_{-1}$ as follows:
	
\[\mathcal{C}_{-1}(\alpha, \nu, i , I_a, I_b) := \left \lceil \frac{\nu_i - \sum_{j \in I_a} \min \lbrace \alpha_i, \alpha_j \rbrace + \sum_{j \in I_b} \min \lbrace \alpha_i, \alpha_j \rbrace}{\alpha_i} \right \rceil.\]
\end{df}

\begin{df}
The \textit{ranking-by-ceilings} algorithm $\mathcal{R}_{-1}$ computes a function $\mathbb{N}^{\ell} \times \mathbb{Z}^{\ell} \rightarrow \mathfrak{S}_{\ell}$ iteratively over $\ell$ steps.  
	
Say $\mathcal{R}_{-1}(\alpha, \nu) = \sigma$.  On the $i^{\text{th}}$ step of the algorithm, $\sigma^{-1}(1), \ldots, \sigma^{-1}(i-1)$ have already been determined.  Set \[J_i := \lbrace \sigma^{-1}(1), \ldots, \sigma^{-1}(i-1) \rbrace \quad \text{and} \quad J'_i := \lbrace 1, \ldots, \ell \rbrace \setminus J_i.\]  Then $\sigma^{-1}(i)$ is designated the numerically minimal $j \in J'_i$ among those for which \[(\mathcal{C}_{-1}(\alpha, \nu, j, J_i, J'_i \setminus \lbrace j \rbrace), \alpha_j, \nu_j)\] is lexicographically maximal.    
\end{df}

\begin{df}
The \textit{column-ceilings} algorithm  $\mathcal{U}_{-1}$ is iterative with $\ell$ steps and computes a function $\mathbb{N}^{\ell} \times \mathbb{Z}^{\ell} \times \mathfrak{S}_{\ell} \rightarrow \mathbb{Z}^{\ell}_{\text{dom}}$, where $\mathbb{Z}^{\ell}_{\text{dom}} \subset \mathbb{Z}^{\ell}$ denotes the subset of weakly decreasing sequences.  
	
Say $\mathcal{U}_{-1}(\alpha, \nu, \sigma) = [\iota_1, \ldots, \iota_{\ell}]$.  On the $i^{\text{th}}$ step of the algorithm, $\iota_1, \ldots, \iota_{i-1}$ have already been determined.  Then \[\iota_i := \mathcal{C}_{-1}(\alpha, \nu, \sigma^{-1}(i), \sigma^{-1} \lbrace 1, \ldots, i-1 \rbrace, \sigma^{-1} \lbrace i+1, \ldots, \ell \rbrace) - \ell + 2i - 1\] unless the right-hand side is greater than $\iota_{i-1}$, in which case $\iota_i := \iota_{i-1}$.  
\end{df}
 
We assemble these constituent functions into a recursive algorithm $\mathfrak{A}$ that computes a map $\mathbb{Y}_{n, \ell} \times \mathbb{Z}^{\ell} \rightarrow \mathbb{Z}^{n}$, where $\mathbb{Y}_{n, \ell}$ denotes the set of partitions of $n$ with $\ell$ parts.  

On input $(\alpha, \nu)$, the algorithm sets \[\sigma^1 := \mathcal{R}(\alpha, \nu) \quad \text{and} \quad \mu^1 := \mathcal{U}(\alpha, \nu, \sigma^1).\]

If $\alpha_1 = 1$, it returns $\mu^1$.  

Otherwise, it defines $(\alpha', \nu') \in \mathbb{Y}_{n-\ell, \alpha^*_2} \times \mathbb{Z}^{\alpha^*_2}$ by setting \[\alpha'_i := \alpha_i - 1 \quad \text{and} \quad \nu'_i := \nu_i - \mu^1_{\sigma^1(i)}\] for all $1 \leq i \leq \alpha^*_2$.  

Then it prepends $\mu^1$ to $\mathfrak{A}(\alpha', \nu')$ and returns the result.  

\begin{rem} \label{iter}

The use of recursion makes our instructions for computing $\mathfrak{A}(\alpha, \nu)$ succinct.  At the cost of a bit of clarity, we can rephrase the instructions to use iteration, and thereby delineate every step in the computation.  

Consider the algorithm $\mathfrak{A}_{\operatorname{iter}} \colon \mathbb{Y}_{n, \ell} \times \mathbb{Z}^{\ell} \rightarrow \mathbb{Z}^n$ defined as follows.  

On input $(\alpha, \nu)$, it starts by setting $\alpha^1 := \alpha$, $\nu^1 := \nu$, $\sigma^1 := \mathcal{R}_{-1}(\alpha^1, \nu^1)$, and $\mu^1 := \mathcal{U}_{-1}(\alpha^1, \nu^1, \sigma^1)$.  

Then, for $2 \leq j \leq s$:
\begin{itemize}
\item It defines $\alpha^j$ by $\alpha^j_i :=\alpha^{j-1}_i - 1$ for all $1 \leq i \leq \alpha^*_j$;  
\item It defines $\nu^{j}$ by $\nu^{j}_i := \nu^{j-1}_i - \mu^{j-1}_{\sigma^{j-1}(i)}$ for all $1 \leq i \leq \alpha^*_{j}$;
\item It sets $\sigma^j := \mathcal{R}_{-1}(\alpha^j, \nu^j)$;
\item It sets $\mu^j := \mathcal{U}_{-1}(\alpha^j, \nu^j, \sigma^j)$.
\end{itemize}

Finally, it returns the concatenation of $(\mu^1, \ldots, \mu^s)$.  

It should be clear that $\mathfrak{A}_{\operatorname{iter}}(\alpha, \nu)$ agrees with $\mathfrak{A}(\alpha, \nu)$.  To see this, we induct on $s$.  For the inductive step, it suffices to show that $\mathfrak{A}(\alpha', \nu')$ is the concatenation of $(\mu^2, \ldots, \mu^s)$.  But $\mathfrak{A}(\alpha', \nu') =  \mathfrak{A}_{\operatorname{iter}}(\alpha^2, \nu^2)$ by the inductive hypothesis.  

\end{rem}

\subsection{Examples}

We study three examples.  First, to illustrate the workings of the ranking function, we consider the orbit $\mathcal{O}_{[2,1]}$.  Given $\nu \in \Omega_{[2,1]}$, the algorithm makes exactly one meaningful comparison --- to determine whether $\sigma^1$ is the trivial or nontrivial permutation in $\mathfrak{S}_2$.  

Second, to underscore the advantages of our approach, we consider an input pair $(\alpha, \nu)$ for which there exists only one collection $\lbrace \nu_{i,j} \rbrace_{\substack{1 \leq i \leq \ell \\ 1 \leq j \leq \alpha_i}}$ such that $\nu_{i,j} \in \lbrace \lceil \frac{\nu_i}{\alpha_i} \rceil, \lfloor \frac{\nu_i}{\alpha_i} \rfloor \rbrace$ for all $i,j$, and setting $\mu^j := \operatorname{dom}([\nu_{1,j}, \ldots, \nu_{\alpha^*_j,j}])$ for all $j$ yields an incorrect answer for $\gamma(\alpha, \nu)$.  The input pair is $([3,2,2,1], [15,8,8,4])$.  

Last, we revisit the orbit $\mathcal{O}_{[4,3,2,1,1]}$ featured in Example~\ref{colors} and compute $\mathfrak{A}$ on the input pair $([4,3,2,1,1], [15,14,9,4,4])$, taken from Achar's thesis \cite{Achart}.  We also discuss the computation of $\mathfrak{A}_{\operatorname{iter}}$.  

\begin{exam}
Set $\alpha := [2,1]$.  Then $\alpha^* = [2,1]$.  Reading $G_{\alpha}^{\text{red}}$ and $L_{\alpha}$ off the Young diagram of $\alpha$ (cf. Figure~\ref{rank}), we see that $G_{[2,1]}^{\text{red}} \cong GL_1 \times GL_1$ and $L_{[2,1]} \cong GL_2 \times GL_1$.  
\begin{figure}[h]
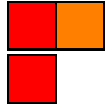

	\ytableausetup{centertableaux}
	\begin{ytableau}
		*(red) & *(orange) \\ 
	\end{ytableau}
	\\ \vspace{0.01in} \hspace{-0.305in}
	\begin{ytableau}
		*(red) \\
	\end{ytableau}
	\caption{The Young diagram of $[2,1]$} \label{rank}
\end{figure}

Note that \[\Omega_{[2,1]} = \lbrace [\nu_1, \nu_2] \in \mathbb{Z}^2 \rbrace \quad \text{and} \quad \Lambda^+_{[2,1]} = \lbrace [\lambda_1, \lambda_2, \lambda_3] \in \mathbb{Z}^3 : \lambda_1 \geq \lambda_2 \rbrace.\]

Let $\nu = [\nu_1, \nu_2] \in \Omega_{[2,1]}$.   On input $(\alpha, \nu)$, the algorithm computes $\sigma^1  := \mathcal{R}_{-1}(\alpha, \nu)$.  Since $\alpha_1 > \alpha_2$, the triple \[(\mathcal{C}_{-1}(\alpha, \nu, 1, \varnothing, \lbrace 2 \rbrace), \alpha_1, \nu_1)\] is lexicographically greater than the triple \[(\mathcal{C}_{-1}(\alpha, \nu, 2, \varnothing, \lbrace 1 \rbrace), \alpha_2, \nu_2)\] if and only if 
\begin{align} \label{pare}
\mathcal{C}_{-1}(\alpha, \nu, 1, \varnothing, \lbrace 2 \rbrace) \geq \mathcal{C}_{-1}(\alpha, \nu, 2, \varnothing, \lbrace 1 \rbrace).
\end{align}

Therefore, $(\sigma^1)^{-1}(1) = 1$ if and only if Inequality~\ref{pare} holds.  By construction of the ranking-by-ceilings algorithm, $(\sigma^1)^{-1}(2) \in \lbrace 1, 2 \rbrace \setminus \lbrace (\sigma^1)^{-1}(1) \rbrace$, so $\sigma^1$ is the identity in $\mathfrak{S}_2$ if Inequality~\ref{pare} holds, and transposes $1$ and $2$ otherwise.  

Evaluating the candidate ceilings, we find:
\begin{align*}
& \mathcal{C}_{-1}(\alpha, \nu, 1, \varnothing, \lbrace 2 \rbrace) = \mathcal{C}_{-1}([2,1], [\nu_1, \nu_2], 1, \varnothing, \lbrace 2 \rbrace) = \left \lceil \frac{\nu_1 + 1}{2} \right \rceil; \\ & \mathcal{C}_{-1}(\alpha, \nu, 2, \varnothing, \lbrace 1 \rbrace) = \mathcal{C}_{-1}([2,1], [\nu_1, \nu_2], 2, \varnothing, \lbrace 1 \rbrace) = \nu_2 + 1.
\end{align*}

Observe that \[\left \lceil \frac{\nu_1 + 1}{2} \right \rceil \geq \nu_2 + 1 \Longleftrightarrow \nu_1 \geq 2 \nu_2.\]

Hence \[\sigma^1 = \begin{cases} 
12 & \nu_1 \geq 2 \nu_2 \\ 21 & \nu_1 \leq 2 \nu_2 - 1 
\end{cases}.\]

We treat each case separately.  

\begin{enumerate}
\item Suppose $\nu_1 \geq 2 \nu_2$.  

The algorithm computes $\mu^1 := \mathcal{U}_{-1}(\alpha, \nu, \sigma^1)$.  By definition, \[\mu^1_1 = \mathcal{C}_{-1}(\alpha, \nu, 1, \varnothing, \lbrace 2 \rbrace) - 1 = \left \lceil \frac{\nu_1 - 1}{2} \right \rceil.\]

Since \[\mathcal{C}_{-1}(\alpha, \nu, 2, \lbrace 1 \rbrace,  \varnothing) + 1 = \nu_2,\] and $\lceil \frac{\nu_1 - 1}{2} \rceil \geq \nu_2$, it follows that \[\mu^1_1 \geq \mathcal{C}_{-1}(\alpha, \nu, 2, \lbrace 1 \rbrace,  \varnothing) + 1.\]

Hence \[\mu^1_2 = \mathcal{C}_{-1}(\alpha, \nu, 2, \lbrace 1 \rbrace,  \varnothing) + 1 = \nu_2.\]

Then the algorithm sets $\alpha' := [1]$, and it defines $\nu'$ by \[\nu'_1 := \nu_1 - \mu^1_1 = \nu_1 - \left \lceil \frac{\nu_1 - 1}{2} \right \rceil = \left \lfloor \frac{\nu_1+1}{2} \right \rfloor.\]

Clearly, \[\mathfrak{A}(\alpha', \nu') = \mathcal{C}_{-1}(\alpha', \nu', 1, \varnothing, \varnothing) = \nu'_1 = \left \lfloor \frac{\nu_1+1}{2} \right \rfloor.\]

Hence \[\mathfrak{A}([2,1], [\nu_1, \nu_2]) = \left[\left \lceil \frac{\nu_1 - 1}{2} \right \rceil, \nu_2, \left \lfloor \frac{\nu_1 + 1}{2} \right \rfloor \right].\]

\item Suppose $\nu_1 \leq 2 \nu_2 - 1$.  

The algorithm computes $\mu^1 := \mathcal{U}_{-1}(\alpha, \nu, \sigma^1)$.  By definition, \[\mu^1_1 = \mathcal{C}_{-1}(\alpha, \nu, 2, \varnothing, \lbrace 1 \rbrace) - 1 = \nu_2.\]

Since \[\mathcal{C}_{-1}(\alpha, \nu, 1, \lbrace 2 \rbrace,  \varnothing) + 1 = \left \lceil \frac{\nu_1 + 1}{2} \right \rceil\] and $\nu_2 \geq \lceil \frac{\nu_1 + 1}{2} \rceil$, it follows that \[\mu^1_1 \geq \mathcal{C}_{-1}(\alpha, \nu, 1, \lbrace 2 \rbrace,  \varnothing) + 1.\]

Hence \[\mu^1_2 = \mathcal{C}_{-1}(\alpha, \nu, 1, \lbrace 2 \rbrace,  \varnothing) + 1 = \left \lceil \frac{\nu_1+1}{2} \right \rceil.\]

Then the algorithm sets $\alpha' := [1]$, and it defines $\nu'$ by \[\nu'_1 := \nu_1 - \mu^1_2 = \nu_1 - \left \lceil \frac{\nu_1 + 1}{2} \right \rceil = \left \lfloor \frac{\nu_1 - 1}{2} \right \rfloor.\]   

Clearly, \[\mathfrak{A}(\alpha', \nu') = \mathcal{C}_{-1}(\alpha', \nu', 1, \varnothing, \varnothing) = \nu'_1 = \left \lfloor \frac{\nu_1 - 1}{2} \right \rfloor.\]

Hence \[\mathfrak{A}([2,1], [\nu_1, \nu_2]) = \left[\nu_2, \left \lceil \frac{\nu_1 + 1}{2} \right \rceil, \left \lfloor \frac{\nu_1 - 1}{2} \right \rfloor \right].\]
\end{enumerate}

We conclude that \[\mathfrak{A}([2,1], [\nu_1, \nu_2]) = \begin{cases} 
\left[\left \lceil \frac{\nu_1 - 1}{2} \right \rceil, \nu_2, \left \lfloor \frac{\nu_1 + 1}{2} \right \rfloor \right] & \nu_1 \geq 2 \nu_2 \\ 
\left[\nu_2, \left \lceil \frac{\nu_1 + 1}{2} \right \rceil, \left \lfloor \frac{\nu_1 - 1}{2} \right \rfloor \right] & \nu_1 \leq 2 \nu_2 - 1 \end{cases}.\]

Since $\rho_{[2,1]} = [\frac{1}{2}, -\frac{1}{2}, 0]$, assuming Theorem~\ref{main} holds, we find \[\gamma([2,1], [\nu_1, \nu_2]) = \begin{cases}
\left[\left \lceil \frac{\nu_1 + 1}{2} \right \rceil, \left \lfloor \frac{\nu_1 + 1}{2} \right \rfloor, \nu_2 - 1 \right] & \nu_1 \geq 2 \nu_2 \\
\left[\nu_2 + 1, \left \lceil \frac{\nu_1 - 1}{2} \right \rceil, \left \lfloor \frac{\nu_1 - 1}{2} \right \rfloor \right] & \nu_1 \leq 2 \nu_2 - 1 \end{cases}.\]
\end{exam}

\begin{exam}
Set $\alpha := [3,2,2,1]$.  Then $\alpha^* = [4,3,1]$.  Reading $G_{\alpha}^{\text{red}}$ and $L_{\alpha}$ off the diagram of $\alpha$ (cf. Figure~\ref{task}), we see that $G_{\alpha}^{\text{red}} \cong GL_1 \times GL_2 \times GL_1$ and $L_{\alpha} \cong GL_4 \times GL_3 \times GL_1$.  

\begin{figure}[h]
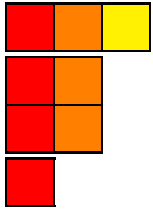

	\ytableausetup{centertableaux}
	\begin{ytableau}
		*(red) & *(orange) & *(yellow) \\ 
	\end{ytableau}
	\\ \vspace{0.01in} \hspace{-0.305in}
	\begin{ytableau}
		*(red) & *(orange) \\
		*(red) & *(orange)
	\end{ytableau}
	\\ \vspace{0.01in} \hspace{-0.555in}
	\begin{ytableau}
		*(red) \\
	\end{ytableau}
	\caption{The Young diagram of $[4,3,1]$} \label{task}
\end{figure} 

Note that \[\Omega_{\alpha} = \lbrace \nu \in \mathbb{Z}^4 : \nu_2 \geq \nu_3 \rbrace\] and \[\Lambda^+_{\alpha} = \lbrace \lambda \in \mathbb{Z}^{8} : \lambda_1 \geq \lambda_2 \geq \lambda_3 \geq \lambda_4; \lambda_5 \geq \lambda_6 \geq \lambda_7 \rbrace.\]

Set $\nu := [15, 8, 8, 4] \in \Omega_{\alpha}$.  On input $(\alpha, \nu)$, the algorithm computes \[\sigma^1 := \mathcal{R}_{-1}(\alpha, \nu) = 1234.\]

Next it computes \[\mu^1 := \mathcal{U}_{-1}(\alpha, \nu, \sigma^1) = [4, 4, 4, 4].\]

Then it sets \[\alpha' := [2, 1, 1] \quad \text{and} \quad \nu' := [11, 4, 4].\]

To finish off, it computes \[\mathfrak{A}(\alpha', \nu') = [5, 4, 4, 6].\]

Thus, \[\mathfrak{A}(\alpha, \nu) = [4, 4, 4, 4, 5, 4, 4, 6].\]

Since \[\rho_{\alpha} = \left[\frac{3}{2}, \frac{1}{2}, - \frac{1}{2}, - \frac{3}{2}, 1, 0, -1, 0 \right],\] assuming Theorem~\ref{main} holds, we find \[\gamma(\alpha, \nu) = [7,7,6,5,4,3,2,1].\]

Note that \[\frac{\nu_1}{\alpha_1} = 5 \quad \text{and} \quad \frac{\nu_2}{\alpha_2} = \frac{\nu_3}{\alpha_3} = \frac{\nu_4}{\alpha_4} = 4.\]

Therefore, if $\lbrace \nu_{i,j} \rbrace_{\substack{1 \leq i \leq 4 \\ 1 \leq j \leq \alpha_i}} \subset \mathbb{Z}$ is a collection such that $\nu_{i,j} \in \lbrace \lceil \frac{\nu_i}{\alpha_i} \rceil, \lfloor \frac{\nu_i}{\alpha_i} \rfloor \rbrace$ for all $i, j$, then \[\nu_{1,1} = \nu_{1,2} = \nu_{1,3} = 5 \quad \text{and} \quad \nu_{2,1} = \nu_{2,2} = \nu_{3,1} = \nu_{3,2} = \nu_{4,1} = 4.\]  

Setting \[\mu^1 := [5, 4, 4, 4], \quad \mu^2 := [5, 4, 4], \quad \mu^3 := [5],\] we arrive at the induced weight $\mu = [5, 4, 4, 4, 5, 4, 4, 5]$, which has smaller norm than the output $\mathfrak{A}(\alpha, \nu) = [4, 4, 4, 4, 5, 4, 4, 6]$.  

However, \[[5, 4, 4, 4, 5, 4, 4, 5] + 2 \rho_{\alpha} = [8, 5, 3, 1, 7, 4, 2, 5],\] which has larger norm than \[[4, 4, 4, 4, 5, 4, 4, 6] + 2 \rho_{\alpha} = [7, 5, 3, 1, 7, 4, 2, 6].\]

Thus, attempting to minimize $||\mathfrak{A}(\alpha, \nu)||$ leads to an incorrect answer for $\gamma(\alpha, \nu)$.  It is essential to minimize $||\mathfrak{A}(\alpha, \nu) + 2 \rho_{\alpha}||$, which is accomplished by our algorithm (cf. Remark~\ref{motiv}).  
\end{exam}

\begin{exam} \label{acharexam}
Set $\alpha := [4,3,2,1,1]$.  Then $\alpha^* = [5,3,2,1]$.  Recall from Example~\ref{colors} that \[G_{\alpha}^{\text{red}} \cong GL_1 \times GL_1 \times GL_1 \times GL_2 \quad \text{and} \quad L_{\alpha} \cong GL_5 \times GL_3 \times GL_2 \times GL_1.\]

Note that \[\Omega_{\alpha} = \lbrace \nu \in \mathbb{Z}^5 : \nu_4 \geq \nu_5 \rbrace\] and \[\Lambda^+_{\alpha} = \lbrace \lambda \in \mathbb{Z}^{11} : \lambda_1 \geq \lambda_2 \geq \lambda_3 \geq \lambda_4 \geq \lambda_5; \lambda_6 \geq \lambda_7 \geq \lambda_8; \lambda_9 \geq \lambda_{10} \rbrace.\]

Set $\nu := [15, 14, 9, 4, 4] \in \Omega_{\alpha}$.  On input $(\alpha, \nu)$, the algorithm computes \[\sigma^1 := \mathcal{R}_{-1}(\alpha, \nu) = 42135.\]  

Next it computes \[\mu^1 := \mathcal{U}_{-1}(\alpha, \nu, \sigma^1) = [4, 4, 4, 4, 4].\]

Then it sets \[\alpha' := [3, 2, 1] \quad \text{and} \quad \nu' := [11, 10, 5].\]  

To finish off, it computes \[\mathfrak{A}(\alpha', \nu') = [5, 5, 5, 5, 4, 2].\]

Thus, \[\mathfrak{A}(\alpha, \nu) = [4, 4, 4, 4, 4, 5, 5, 5, 5, 4, 2].\]

If we run $\mathfrak{A}_{\operatorname{iter}}$ on input $(\alpha, \nu)$, we obtain the following table.

\begin{center}
	\begin{tabular}{ |l|l|l|l|l| } 
		\hline
		$\alpha^1 = [4,3,2,1,1]$ & $\nu^1 = [15, 14, 9, 4, 4]$ & $\sigma^1 = 42135$ & $\mu^1 = [4, 4, 4, 4, 4]$ \\
		$\alpha^2 = [3,2,1]$ & $\nu^2 = [11, 10, 5]$ & $\sigma^2 = 312$ & $\mu^2 = [5, 5, 5]$ \\ 
		$\alpha^3 = [2,1]$ & $\nu^3 = [6, 5]$ & $\sigma^3 = 21$ & $\mu^3 = [5, 4]$ \\ 
		$\alpha^4 = [1]$ & $\nu^4 = [2]$ & $\sigma^4 = 1$ & $\mu^4 = [2]$ \\
		\hline
	\end{tabular}
\end{center}

Hence \[\mathfrak{A}_{\operatorname{iter}}(\alpha, \nu) = [4, 4, 4, 4, 4, 5, 5, 5, 5, 4, 2] = \mathfrak{A}(\alpha, \nu).\]

Since \[\rho_{\alpha} = \left[2, 1, 0, -1, -2, 1, 0, -1, \frac{1}{2}, -\frac{1}{2}, 0 \right],\] assuming Theorem~\ref{main} holds, we find \[\gamma(\alpha, \nu) = [8, 7, 6, 6, 5, 4, 3, 3, 2, 2, 0].\]

This agrees with Achar's answer (cf. \cite{Achart}, Appendix A).  
\end{exam}

\vfill \eject

\section{Weight Diagrams}
In this section, we define a class of combinatorial models, which Achar christened \textit{weight diagrams}.  In form akin to Young tableaux, weight diagrams in function capture at the level of integer sequences the interactions in $K_0(\mathfrak{D})$ described in Lemmas~\ref{omega} and ~\ref{lambda}.  A weight diagram of shape-class $\alpha \vdash n$ simultaneously depicts a dominant integer sequence $\kappa(X)$ with respect to $\alpha$ and a dominant weight $h(X)$ of $L_{\alpha}$.  We establish herein that $[IC_{(\alpha, \kappa(X))}]$ occurs in the decomposition of $[A^{\alpha}_{h(X)}]$ on the $\Omega$-basis.  

Let $\alpha = [\alpha_1, \ldots, \alpha_{\ell}]$ be a partition of $n$ with conjugate partition $\alpha^* = [\alpha^*_1, \ldots, \alpha^*_s]$.  Let $k_1 > \cdots > k_m$ be the distinct parts of $\alpha$, and $a_t$ be the multiplicity of $k_t$ for all $1 \leq t \leq m$.  

\begin{df} \label{blank}
A \textit{blank diagram} of \textit{shape-class} $\alpha$ is a collection of unit squares (referred to as boxes) arranged in $\ell$ left-justified rows, which differs from a Young diagram of shape $\alpha$ only by permutation of the rows.  
\end{df}

\begin{df} \label{diagram}
A \textit{weight diagram} of \textit{shape-class} $\alpha$ is a filling of a blank diagram of shape-class $\alpha$ by integer entries, with one entry in each box.
\end{df}

Let $D_{\alpha}$ be the set of all weight diagrams of shape-class $\alpha$.  For a weight diagram $X \in D_{\alpha}$, we denote by $X^j_i$ the $i^{\text{th}}$ entry from the top in the $j^{\text{th}}$ column from the left.  We next define a combinatorial map $E \colon D_{\alpha} \rightarrow D_{\alpha}$.  

\begin{df}
Let $X$ be a weight diagram of shape-class $\alpha$.  Set $EX$ to be the filling of the same blank diagram as $X$ given by $EX^j_i := X^j_i + \alpha^*_j - 2i + 1$ for all $1 \leq j \leq s$, $1 \leq i \leq \alpha^*_j$.  
\end{df}

For the sake of convenience, we consider weight diagrams in pairs for which the second diagram is obtained from the first via $E$.  The weight-diagrams version of our algorithm better stores simultaneously the combinatorial information pertinent to the corresponding elements in $\Omega_{\alpha}$ and $\Lambda^+$ when formulated to build diagram pairs, rather than individual diagrams. 

\begin{df}
Let $\overline{E} \colon D_{\alpha} \rightarrow D_{\alpha} \times D_{\alpha}$ denote the composition of the diagonal map $D_{\alpha} \rightarrow D_{\alpha} \times D_{\alpha}$ with the map $\operatorname{Id} \times E \colon D_{\alpha} \times D_{\alpha} \rightarrow D_{\alpha} \times D_{\alpha}$.  A \textit{diagram pair} of \textit{shape-class} $\alpha$ is an ordered pair of diagrams $(X, Y)$ in $\overline{E}(D_{\alpha})$.  
\end{df}

The nomenclature ``weight diagram'' is attributable to the natural maps $\kappa \colon D_{\alpha} \rightarrow \Omega_{\alpha}$, $h \colon D_{\alpha} \rightarrow \Lambda^+_{\alpha}$, and $\eta \colon D_{\alpha} \rightarrow \Lambda^+$, which we proceed to define.  

\begin{df}
Let $X$ be a weight diagram of shape-class $\alpha$.  For all $1 \leq t \leq m$, $1 \leq i \leq a_t$, $1 \leq j \leq k_t$, let $\kappa_X^j(t, i)$ be the entry of $X$ in the $j^{\text{th}}$ column and the $i^{\text{th}}$ row from the top among rows of length $k_t$.  Then set \[\kappa_X(t) := \operatorname{dom} \left(\sum_{j=1}^{k_t} [\kappa_X^j(t, 1), \ldots, \kappa_X^j(t, a_t)] \right).\]  Set $\kappa(X)$ to be the concatenation of the $m$-tuple $(\kappa_X(1), \ldots, \kappa_X(m))$.  
\end{df}

\begin{df}
Let $X$ be a weight diagram of shape-class $\alpha$.  For all $1 \leq j \leq s$, set $h_X^j := \operatorname{dom}([X^j_1, \ldots, X^j_{\alpha^*_j}])$.  Then set $h(X)$ to be the concatenation of the $s$-tuple $(h_X^1, \ldots, h_X^s)$.  
\end{df}

\begin{df}
Let $Y$ be a weight diagram of shape-class $\alpha$.  Set $\eta(Y) := \operatorname{dom}(h(Y))$.  
\end{df}

Suppose that the entries of $X$ are weakly decreasing down each column.  Then $E$ lifts the addition of $2 \rho_{\alpha}$ to the underlying $L_{\alpha}$-weight of $X$; in other words, $h(EX) = h(X) + 2 \rho_{\alpha}$.  Hence 
\begin{align} \label{compat}
\eta(EX) = \operatorname{dom}(h(X) + 2 \rho_{\alpha}).
\end{align}   

If $X$ is \textit{distinguished} (cf. Definition~\ref{dis}), then the pair $(\alpha, \kappa(X)) \in \Omega$ and the dominant weight $\eta(EX) \in \Lambda^+$ correspond under $\gamma$ (cf. Theorem~\ref{achar}), and both can be read off the diagram pair $(X, EX)$.  The task of the weight-diagrams version of our algorithm is to find, on input $(\alpha, \nu)$, a distinguished diagram $X$ such that $\kappa(X) = \nu$, and output $(X, EX)$.   

\begin{exam} \label{thes}
We present a diagram pair of shape-class $[4,3,2,1,1] \vdash 11$, taken from Achar's thesis \cite{Achart}.    

\begin{figure}[h]
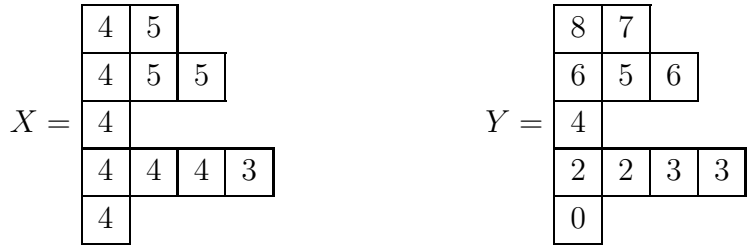

\ytableausetup{centertableaux, textmode}
\hspace{0.5in}
$X =$ \begin{ytableau}
4 & 5\\
4 & 5 & 5\\
4\\
4 & 4 & 4 & 3\\
4
\end{ytableau}
\hspace{1in}
$Y = $ \begin{ytableau}
8 & 7\\
6 & 5 & 6\\
4\\
2 & 2 & 3 & 3\\
0
\end{ytableau}
\caption{A diagram pair of shape-class $[4,3,2,1,1]$}
\end{figure}

We see that $\kappa(X) = [15, 14, 9, 4, 4]$ and $h(X) = [4, 4, 4, 4, 4, 5, 5, 4, 5, 4, 3]$.  Furthermore, $Y = EX$, and $\eta(Y) = [8, 7, 6, 6, 5, 4, 3, 3, 2, 2, 0]$.  As noted in Example~\ref{acharexam}, \[\gamma([4,3,2,1,1], \kappa(X)) = \eta(Y).\] 
\end{exam}

\begin{thm} \label{decomp}
Let $(X, Y) \in \overline{E}(D_{\alpha})$ be a diagram pair of shape-class $\alpha$.  Then $V^{(\alpha, \kappa(X))}$ occurs in the decomposition of $W^{h(X)}$ as a direct sum of irreducible $G_{\alpha}^{\text{red}}$-representations.  Furthermore, $[IC_{(\alpha, \kappa(X))}]$ occurs in the decomposition of $[A^{\alpha}_{h(X)}]$ on the $\Omega$-basis.    
\end{thm}

\begin{proof}
It suffices to prove the former statement, for the latter follows from the former in view of Lemma~\ref{omega}.  For all $1 \leq t \leq m$, $1 \leq j \leq k_t$, set \[\kappa_X^j(t) := \operatorname{dom}([\kappa_X^j(t, 1), \ldots, \kappa_X^j(t, a_t)]).\]  For all $1 \leq j \leq s$, let $\kappa_X^j$ be the concatenation of $\prod_{t : k_t \geq j} \kappa_X^j(t)$.  Finally, set $\kappa_X^{\text{ref}}$ to be concatenation of $(\kappa_X^1, \ldots, \kappa_X^s)$.  

Observe first that $\kappa_X^{\text{ref}}$ is a dominant weight of $L_{\alpha}^{\text{ref}} := L_{X_{\alpha}}^{\text{ref}}$ with respect to the Borel subgroup $B_{\alpha}^{\text{ref}} := B_{L_{\alpha}^{\text{ref}}}$.  To see this, note that $\kappa_X^j(t)$ is weakly decreasing for all $1 \leq t \leq m$, $1 \leq j \leq k_t$, and $L_{\alpha}^{\text{ref}}$ is included in $L_{\alpha}$ via the product, over all $1 \leq j \leq s$, of the inclusions $\prod_{t : k_t \geq j} GL_{a_t} \rightarrow GL_{\alpha^*_j}$ (cf. section 1.2).  

Since $\kappa_X^j$ is a permutation of $h_X^j$ for all $1 \leq j \leq s$, it follows that $\kappa_X^{\text{ref}}$ belongs to the $W_{\alpha}$-orbit of $h(X)$, so $\kappa_X^{\text{ref}}$ is a weight of the $L_{\alpha}$-representation $W^{h(X)}$.  Let $w \in W_{\alpha}$ be chosen so that $w(\kappa_X^{\text{ref}}) = h(X)$.  We claim that $\kappa_X^{\text{ref}}$ is a highest weight of the restriction of $W^{h(X)}$ to $L_{\alpha}^{\text{ref}}$.  

Let $\Phi_{\alpha} \subset \Lambda$ be the set of roots of $L_{\alpha}$, and let $\Phi_{\alpha}^{\text{ref}} \subset \Phi_{\alpha}$ be the subset of roots of $L_{\alpha}^{\text{ref}}$.  Assume for the sake of contradiction that there exists a root $\beta \in \Phi_{\alpha}^{\text{ref}}$, positive with respect to $B_{\alpha}^{\text{ref}}$, such that $\kappa_X^{\text{ref}} + \beta$ is a weight of $W^{h(X)}$.  Let $\beta^{\vee}$ denote the coroot corresponding to $\beta$.  Then $\langle \kappa_X^{\text{ref}}, \beta^{\vee} \rangle \geq 0$, which implies $\langle h(X), \beta_1^{\vee} \rangle \geq 0$, where $\beta_1 := w(\beta)$.    

However, $w(\kappa_X^{\text{ref}} + \beta) = h(X) + \beta_1$ is a weight of $W^{h(X)}$, so $\beta_1$ must be negative with respect to $B_{\alpha}$.  Since $h(X)$ is dominant with respect to $B_{\alpha}$, it follows that $\langle h(X), \beta_1^{\vee} \rangle \leq 0$.  

We conclude that $\langle h(X), \beta_1^{\vee} \rangle = 0$.  Let $s_{\beta_1} \in W_{\alpha}$ be the reflection corresponding to $\beta_1$.  Then $s_{\beta_1}(h(X)) = h(X)$.  Hence $s_{\beta_1}(h(X) + \beta_1) = h(X) - \beta_1$ is a weight of $W^{h(X)}$ that exceeds $h(X)$ in the root order.  (Contradiction.)  

Let $V$ be the $(GL_{a_1})^{k_1} \times \cdots \times (GL_{a_m})^{k_m}$-representation given by \[ V := \left(V^{\kappa_X^1(1)} \boxtimes \cdots \boxtimes V^{\kappa_X^{k_1}(1)}\right) \boxtimes \cdots \boxtimes \left(V^{\kappa_X^1(m)} \boxtimes \cdots \boxtimes V^{\kappa_X^{k_m}(m)}\right).\]

What we have just shown implies that $V$ occurs in the decomposition of $W^{h(X)}$ as a direct sum of irreducible $L_{\alpha}^{\text{ref}}$-representations.  Recall from section 1.2 that $G_{\alpha}^{\text{red}}$ is embedded in $L_{\alpha}^{\text{ref}}$ via the product, over all $1 \leq t \leq m$, of the diagonal embeddings $GL_{a_t} \rightarrow (GL_{a_t})^{k_t}$.  It follows that the restriction of $V$ to $G_{\alpha}^{\text{red}} \cong GL_{a_1, \ldots, a_m}$ is \[\left(V^{\kappa_X^1(1)} \otimes \cdots \otimes V^{\kappa_X^{k_1}(1)}\right) \boxtimes \cdots \boxtimes \left(V^{\kappa_X^1(m)} \otimes \cdots \otimes V^{\kappa_X^{k_m}(m)}\right).\]

Therefore, to see that \[\dim \operatorname{Hom}_{G_{\alpha}^{\text{red}}} \left(V^{(\alpha, \kappa(X))}, V \right) > 0,\] it suffices to show that \[\dim \operatorname{Hom}_{GL_{a_t}}\left(V^{\kappa_X(t)}, V^{\kappa_X^1(t)} \otimes \cdots \otimes V^{\kappa_X^{k_t}(t)}\right) > 0\] for all $1 \leq t \leq m$.  

This is a consequence of the Parthasarathy--Ranga Rao--Varadarajan conjecture, first proved for complex semisimple algebraic groups (via sheaf cohomology) by Kumar \cite{Kumar} in 1988.  For complex general linear groups, a combinatorial proof via honeycombs is given in Knutson--Tao \cite{Knutson}, section 4.  
\end{proof}

\begin{rem}
In Achar's work, the corresponding claim is Proposition 4.4 in \cite{Acharj}.  Unfortunately, Achar's proof is incorrect: He implicitly assumes that the combinatorial map $\kappa \colon D_{\alpha} \rightarrow \Omega_{\alpha}$ lifts the action of a representation-theoretic map $\Lambda^+_{\alpha} \rightarrow \Omega_{\alpha}$, which he also denotes by $\kappa$, so that $\kappa(X) = \kappa(h(X))$.  This is manifestly untrue, for permuting the entries within a column of $X$ affects $\kappa(X)$ but leaves $h(X)$ unchanged.  

Thus, Achar's assertion: 
\begin{quote} ``\ldots the $G^{\alpha}$-submodule generated by the $\mu$-weight space of $V^L_{\mu}$ is a representation whose highest weight is the restriction of $\mu$, \textit{which is exactly what $E$ is}'' [emphasis added]
\end{quote}
is false unless $\kappa_X^{\text{ref}}$ coincides with $h(X)$ and $\kappa_X(t) = \sum_{j=1}^{k_t} \kappa_X^j(t)$ for all $1 \leq t \leq m$ --- in which case the $L_{\alpha}^{\text{ref}}$-subrepresentation of $W^{h(X)}$ generated by the highest weight space is isomorphic to $V$, and the highest weight of its restriction to $G_{\alpha}^{\text{red}}$ is $\kappa(X)$.  
\end{rem}

\begin{exam}
Set $\alpha := [2,2]$.  Note that $G_{[2,2]}^{\text{red}} \cong GL_{2}$ and $L_{[2,2]}^{\text{ref}} = L_{[2,2]} \cong (GL_{2})^2$.  Furthermore, $G_{[2,2]}^{\text{red}}$ is embedded in $L_{[2,2]}$ via the diagonal embedding $GL_2 \rightarrow (GL_2)^2$. 

Let $X_1$ and $X_2$ be the weight diagrams $\begin{smallmatrix} 1 & 1 \\ 0 & 0 \end{smallmatrix}$ and $\begin{smallmatrix} 1 & 0 \\ 0 & 1 \end{smallmatrix}$, respectively.  Then \[\kappa(X_1) = [2,0], \quad \kappa(X_2) = [1,1], \quad \text{and} \quad h(X_1) = h(X_2) = [1,0,1,0].\]  

The restriction of the $L_{[2,2]}$-representation \[W^{[1,0,1,0]} = W^{[1,0]} \boxtimes W^{[1,0]}\] to $G_{[2,2]}^{\text{red}}$ is \[W^{[1,0]} \otimes W^{[1,0]} \cong W^{[2,0]} \oplus W^{[1,1]}.\]

Hence Theorem~\ref{decomp} holds for $X_1$ and $X_2$.  

However, Achar's proof is valid for $X_1$ only.  To see this, let $v$ and $w$ be weight vectors of $W^{[1,0]}$ of weight $[1,0]$ and $[0,1]$, respectively.  Up to scaling, \[\lbrace v \otimes v, v \otimes w, w \otimes v, w \otimes w \rbrace\] is the unique basis of weight vectors for $W^{[1,0]} \boxtimes W^{[1,0]}$.  Whereas $v \otimes v$ and $w \otimes w$ each generates a $GL_2$-subrepresentation isomorphic to $W^{[2,0]}$, both $v \otimes w$ and $w \otimes v$ are cyclic vectors.  No weight space of $W^{[1,0]} \boxtimes W^{[1,0]}$ generates a $GL_2$-subrepresentation isomorphic to $W^{[1,1]}$ (instead, $W^{[1,1]}$ is generated by $v \otimes w - w \otimes v$).  
\end{exam}

\begin{cor} \label{decamp}
Let $\nu \in \Omega_{\alpha}$, and let $\lbrace \nu_{i,j} \rbrace_{\substack{1 \leq i \leq \ell \\ 1 \leq j \leq \alpha_i}}$ be a collection of integers such that \[\nu_i = \nu_{i, 1} + \cdots + \nu_{i, \alpha_i}\] for all $1 \leq i \leq \ell$.  For all $1 \leq j \leq s$, set $\mu^j := \operatorname{dom}([\nu_{1, j}, \ldots, \nu_{\alpha^*_j, j}])$.  Set $\mu$ to be the concatenation of $(\mu^1, \ldots, \mu^s)$.  Then $\mu \in \Lambda^+_{\alpha, \nu}$.  
\end{cor}

\begin{proof}
Let $X$ be the filling of the Young diagram of shape $\alpha$ for which $\nu_{i,j}$ is the entry in the $i^{\text{th}}$ row and $j^{\text{th}}$ column of $X$ for all $i, j$.  Then $\kappa(X) = \nu$, and $h(X) = \mu$.  Hence the result follows from Theorem~\ref{decomp}.    
\end{proof}

\begin{cor} \label{inside}
Let $\nu \in \Omega_{\alpha}$.  Then $\mathfrak{A}(\alpha, \nu) \in \Lambda^+_{\alpha, \nu}$.  
\end{cor}

\begin{proof}
By Remark~\ref{iter}, it suffices to show that $\mathfrak{A}_{\operatorname{iter}}(\alpha, \nu) \in \Lambda^+_{\alpha, \nu}$.  For all $1 \leq i \leq \ell$ and $1 \leq j \leq \alpha_i$, set $\nu_{i,j} := \mu^j_{\sigma^{j}(i)}$.  Then Corollary~\ref{decamp} implies the result.  
\end{proof}

\vfill \eject

\section{The Algorithm, Weight-Diagrams Version} 

\subsection{Overview}
In this section, we reengineer our algorithm from section 2.2 to output diagram pairs rather than weights.  Let $D_{\ell}$ be the set of weight diagrams, of any shape-class, with $\ell$ rows.  For a diagram $X \in D_{\ell}$, we denote by $X_{i,j}$ the entry of $X$ in the $i^{\text{th}}$ row and the $j^{\text{th}}$ column.  

We define a recursive algorithm $\mathcal{A}$ that computes a map \[\mathbb{N}^{\ell} \times \mathbb{Z}^{\ell} \times \lbrace \pm 1 \rbrace \rightarrow D_{\ell} \times D_{\ell}\]  by determining the entries in the first column of each diagram of its output and using recursion to ascertain the entries in the remaining columns.  Whenever we write $\mathcal{A}(\alpha, \nu)$, we refer to $\mathcal{A}(\alpha, \nu, -1)$.  

Let maps $p_1, p_2 \colon D_{\ell} \times D_{\ell} \rightarrow D_{\ell}$ be given by projection onto the first and second factors, respectively.  We refer to $p_1 \mathcal{A}(\alpha, \nu)$ as the \textit{left} diagram and to $p_2 \mathcal{A}(\alpha, \nu)$ as the \textit{right} diagram.  The algorithm $\mathcal{A}$ computes the Lusztig--Vogan bijection via $\gamma(\alpha, \nu) = \eta p_2 \mathcal{A}(\alpha, \nu)$.  

While $\mathcal{A}$ relies on the same functions as $\mathfrak{A}$ for its computations, it also requires companion versions of these functions that use floors rather than ceilings.  The \textit{candidate-floor} function $\mathcal{C}_1$, and the \textit{ranking-by-floors} and \textit{column-floors} algorithms $\mathcal{R}_1$ and $\mathcal{U}_1$, are analogous to the function $\mathcal{C}_{-1}$, and the algorithms $\mathcal{R}_{-1}$ and $\mathcal{U}_{-1}$, respectively, and we define them formally in section 4.2.  

More substantively, the recursive structure of $\mathcal{A}$ differs from that of $\mathfrak{A}$.  The integer-sequences version is singly recursive: On input $(\alpha, \nu)$, it reduces the task of determining the output to one sub-problem, namely, computing $\mathfrak{A}(\alpha', \nu')$.  In contrast, the weight-diagrams version is multiply recursive, and, depending on the input, it may require the solutions to several sub-problems to be assembled in order to return the output.  

After computing the first column of each output diagram, the weight-diagrams version creates a separate branch for each \textit{distinct} entry in the first column of the left diagram.  Then it attaches each branch's output diagrams to the first columns already computed to build the output diagrams of the whole recursion tree.  The attachment process is trivial; preparing each branch for its recursive call is not.\footnote{Thus, $\mathcal{A}$ deviates from the pattern of most prototypical divide-and-conquer algorithms, such as mergesort, for which dividing the residual input into branches is easier than combining the resulting outputs.  }  

On input $(\alpha, \nu, \epsilon)$, the algorithm $\mathcal{A}$ undertakes the following steps to compute $p_1 \mathcal{A}(\alpha, \nu, \epsilon)$ (the diagram $p_2 \mathcal{A}(\alpha, \nu, \epsilon)$ is computed simultaneously and similarly):
\begin{enumerate}
\item It computes $\sigma := \mathcal{R}_{\epsilon}(\alpha, \nu)$, which it construes as permuting the rows of a blank diagram of shape $\alpha$;\footnote{By a diagram of shape $\alpha$, we mean a diagram for which the $i^{\text{th}}$ row contains $\alpha_i$ boxes for all $1 \leq i \leq \ell$.  }  
\item It fills in the first column of the (permuted) diagram with the entries of $\iota := \mathcal{U}_{\epsilon}(\alpha, \nu, \sigma)$; 
\item For each row, it appeals to the \textit{row-survival function} to query whether the row \textit{survives} into the residual input (viz., is of length greater than $1$), and, if so, determine which branch of the residual input it is sorted into (and its position therein);  
\item For all $x$, it records the surviving rows in the $x^{\text{th}}$ branch in $\alpha^{(x)}$, and subtracts off the corresponding entries in $\iota$ from those in $\nu$ to obtain $\nu^{(x)}$;
\item For all $x$, it adjusts $\nu^{(x)}$ to $\hat{\nu}^{(x)}$ to reflect the data from the other branches;
\item For all $x$, it sets $X^{(x)} :=  p_1 \mathcal{A}(\alpha^{(x)}, \hat{\nu}^{(x)}, -\epsilon)$ and attaches $X^{(x)}$ to the first column.  
\end{enumerate} 

After the rows of a blank diagram of shape $\alpha$ have been permuted according to $\sigma \in \mathfrak{S}_{\ell}$, the $i^{\text{th}}$ row from the top is of length $\alpha_{\sigma^{-1}(i)}$.  Thus, the $i^{\text{th}}$ row \textit{survives} into the residual input if and only if $\alpha_{\sigma^{-1}(i)} > 1$.  Which branch it belongs to depends on its first-column entry.  

The first column of the permuted diagram is filled in with the entries of $\iota$.  Each distinct entry $\iota^{\circ}$ in $\iota$ gives rise to its own branch, comprising the surviving rows whose first-column entry is $\iota^{\circ}$ (a branch may be empty).  If the $i^{\text{th}}$ row does survive, it is sorted into the $x^{\text{th}}$ branch, where $x$ is the number of distinct entries in the subsequence $[\iota_1, \ldots, \iota_i]$; if, furthermore, exactly $i'$ rows among the first $i$ survive into the $x^{\text{th}}$ branch, then the $i^{\text{th}}$ row becomes the $i'^{\text{th}}$ row in the $x^{\text{th}}$ branch.    

To encompass these observations, we define the \textit{row-survival} function as follows.  
\begin{df}
For all $(\alpha, \sigma, \iota) \in \mathbb{N}^{\ell} \times \mathfrak{S}_{\ell} \times \mathbb{Z}^{\ell}_{\text{dom}}$, \[\mathcal{S}(\alpha, \sigma, \iota) \colon \lbrace 1, \ldots, \ell \rbrace \rightarrow \lbrace 1, \ldots, \ell \rbrace \times \lbrace 0, 1, \ldots, \ell \rbrace\] is given by \[\mathcal{S}(\alpha, \sigma, \iota)(i) := \big(|\lbrace \iota_{i'} : i' \leq i \rbrace|, |\lbrace i' \leq i : \iota_{i'} = \iota_i; \alpha_{\sigma^{-1}(i')} > 1 \rbrace| \cdot 1_{i} \big),\] where \[1_{i} := \begin{cases} 
1 & \alpha_{\sigma^{-1}(i)} > 1 \\ 0 & \alpha_{\sigma^{-1}(i)} = 1 
\end{cases}.\]
\end{df}

\begin{rem}
Suppose $\mathcal{S}(\alpha, \sigma, \iota)(i) = (x, i')$.  Assuming $i' > 0$, the $i^{\text{th}}$ row becomes the $i'^{\text{th}}$ row in the $x^{\text{th}}$ branch (if $i' = 0$, the row dies).  
\end{rem}

\begin{exam} \label{surv}
We revisit the input $(\alpha, \nu) := ([4,3,2,1,1], [15,14,9,4,4])$ from Example~\ref{acharexam}.  As noted therein, $\sigma := \mathcal{R}_{-1}(\alpha, \nu) = 42135$ and $\iota := \mathcal{U}_{-1}(\alpha, \nu, \sigma) = [4,4,4,4,4]$.  Thus, beginning with a blank diagram of shape $\alpha$, we see that the permuted diagram (with first column filled in) looks like 
\begin{figure}[h]
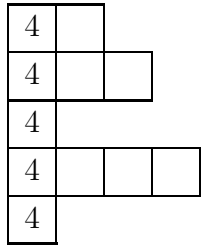

	\ytableausetup{centertableaux, textmode}
	\hspace{0.5in}
	\begin{ytableau}
		4 & \\
		4 &  & \\
		4\\
		4 &  &  & \\
		4
	\end{ytableau}
\caption{The left diagram after steps 1 and 2}
\end{figure}

From the picture, it is clear that there is exactly one branch, comprising the first, second, and fourth rows.  The row-survival function indicates the same, for \[\mathcal{S}(\alpha, \sigma, \iota)(1,2,3,4,5) = ((1,1), (1,2), (1,0), (1,3), (1,0)).\]  

We see later that $\mathcal{A}(\alpha, \nu) = (X, Y)$ in the notation of Example~\ref{thes}.  

\end{exam}  

\subsection{The algorithm}

Before we describe the algorithm, we define the preliminary functions that use floors.  

\begin{df}
Given a pair of integer sequences \[(\alpha, \nu) = ([\alpha_1, \ldots, \alpha_{\ell}], [\nu_1, \ldots, \nu_{\ell}]) \in \mathbb{N}^{\ell} \times \mathbb{Z}^{\ell},\] an integer $i \in \lbrace 1, \ldots, \ell \rbrace$, and an ordered pair of disjoint sets $(I_a, I_b)$ satisfying $I_a \cup I_b = \lbrace 1, \ldots, \ell \rbrace \setminus \lbrace i \rbrace$, we define the \textit{candidate-floor} function $\mathcal{C}$ as follows:

\[\mathcal{C}_1(\alpha, \nu, i , I_a, I_b) := \left \lfloor \frac{\nu_i - \sum_{j \in I_a} \min \lbrace \alpha_i, \alpha_j \rbrace + \sum_{j \in I_b} \min \lbrace \alpha_i, \alpha_j \rbrace}{\alpha_i} \right \rfloor.\]
\end{df}

\begin{df} 
The \textit{ranking-by-floors} algorithm $\mathcal{R}_1$ computes a function $\mathbb{N}^{\ell} \times \mathbb{Z}^{\ell} \rightarrow \mathfrak{S}_{\ell}$ iteratively over $\ell$ steps.  

Say $\mathcal{R}_1(\alpha, \nu) = \sigma$.  On the $i^{\text{th}}$ step of the algorithm, $\sigma^{-1}(\ell), \ldots, \sigma^{-1}(\ell-i+2)$ have already been determined.  Set \[J_i := \lbrace \sigma^{-1}(\ell), \ldots, \sigma^{-1}(\ell-i+2) \rbrace \quad \text{and} \quad J'_i := \lbrace 1, \ldots, \ell \rbrace \setminus J_i.\]  Then $\sigma^{-1}(\ell-i+1)$ is designated the numerically maximal $j \in J'_i$ among those for which \[(\mathcal{C}_1(\alpha, \nu, j, J'_i \setminus \lbrace j \rbrace, J_i), -\alpha_j, \nu_j)\] is lexicographically minimal.  
\end{df}

\begin{df}
The \textit{column-floors} algorithm $\mathcal{U}_1$ is iterative with $\ell$ steps and computes a function $\mathbb{N}^{\ell} \times \mathbb{Z}^{\ell} \times \mathfrak{S}_{\ell} \rightarrow \mathbb{Z}^{\ell}_{\text{dom}}$.  

Say $\mathcal{U}_1(\alpha, \nu, \sigma) = [\iota_1, \ldots, \iota_{\ell}]$.  On the $i^{\text{th}}$ step of the algorithm, $\iota_{\ell}, \ldots, \iota_{\ell - i + 2}$ have already been determined.  Then \[\iota_{\ell - i + 1} := \mathcal{C}_1(\alpha, \nu, \sigma^{-1}(\ell - i + 1), \sigma^{-1} \lbrace 1, \ldots, \ell - i \rbrace, \sigma^{-1} \lbrace \ell - i + 2, \ldots, \ell \rbrace) + \ell - 2i + 1\] unless the right-hand side is less than $\iota_{\ell - i +2}$, in which case $\iota_{\ell - i + 1} := \iota_{\ell - i + 2}$. 
\end{df}

We assemble these functions, together with the preliminary functions that use ceilings, and the row-survival function, into the recursive algorithm $\mathcal{A} \colon \mathbb{N}^{\ell} \times \mathbb{Z}^{\ell} \times \lbrace \pm 1 \rbrace \rightarrow D_{\ell} \times D_{\ell}$.  
 
On input $(\alpha, \nu, \epsilon)$, the algorithm sets \[\sigma := \mathcal{R}_{\epsilon}(\alpha, \nu) \quad \text{and} \quad \iota := \mathcal{U}_{\epsilon}(\alpha, \nu, \sigma).\]  

Next it sets \[X_{i,1} := \iota_{i} \quad \text{and} \quad Y_{i, 1} := \iota_i + \ell - 2i + 1\] for all $1 \leq i \leq \ell$.  

For all $(x, i')$ in the image of $\mathcal{S}(\alpha, \sigma, \iota)$ such that $i' > 0$, we write \[i_{(x, i')} := \mathcal{S}(\alpha, \sigma, \iota)^{-1}(x, i').\]

The algorithm sets $\mathcal{k} := |\lbrace \iota_1, \ldots, \iota_{\ell} \rbrace|$, which counts the number of branches.  

For all $1 \leq x \leq \mathcal{k}$, it sets \[\ell_x := \max \left\lbrace i' : (x, i') \in \mathcal{S}(\alpha, \sigma, \iota) \lbrace 1, \ldots, \ell \rbrace \right\rbrace.\]  

Note that $\ell_x$ counts the number of rows surviving into the $x^{\text{th}}$ branch; if $\ell_x = 0$, then the $x^{\text{th}}$ branch is empty.  

If $\ell_x > 0$, then the $x^{\text{th}}$ branch contains $\ell_x$ surviving rows, and the algorithm sets \[\alpha^{(x)} := \left [\alpha_{\sigma^{-1}\left(i_{(x, 1)}\right)} - 1, \ldots, \alpha_{\sigma^{-1}\left(i_{(x, \ell_x)}\right)} - 1 \right] \] and \[\nu^{(x)} = \left [\nu_{\sigma^{-1}\left(i_{(x, 1)}\right)} - \iota_{i_{(x, 1)}}, \ldots, \nu_{\sigma^{-1}\left(i_{(x, \ell_x)}\right)} - \iota_{i_{(x, \ell_x)}} \right].\]  

The algorithm does not call itself on $(\alpha^{(x)}, \nu^{(x)})$ because it has to adjust $\nu^{(x)}$ to reflect the data from the other branches, if any are present.  

For all $1 \leq i' \leq \ell_x$, it sets \[\hat{\nu}^{(x)}_{i'} := \nu^{(x)}_{i'} - \sum_{x' = 1}^{x-1} \sum_{i_0 = 1}^{\ell_{x'}} \min \left\lbrace \alpha^{(x)}_{i'}, \alpha^{(x')}_{i_0} \right\rbrace + \sum_{x' = x+1}^{\mathcal{k}} \sum_{i_0 = 1}^{\ell_{x'}} \min \left\lbrace \alpha^{(x)}_{i'}, \alpha^{(x')}_{i_0} \right\rbrace.\]

Then it sets $\hat{\nu}^{(x)} := \left[\hat{\nu}^{(x)}_1, \ldots, \hat{\nu}^{(x)}_{\ell_x}\right]$ and $\left(X^{(x)}, Y^{(x)} \right) := \mathcal{A}\left(\alpha^{(x)}, \hat{\nu}^{(x)}, -\epsilon \right)$.  

The algorithm fills in the rest of the entries of $X$ and $Y$ according to the following rule: For all $(i', j') \in \mathbb{N} \times \mathbb{N}$ such that $X^{(x)}$ and $Y^{(x)}$ each have an entry in the $i'^{\text{th}}$ row and $j'^{\text{th}}$ column, 
\begin{align} \label{attachx}
X_{i_{(x, i')}, j'+1} := X^{(x)}_{i', j'} + \sum_{x' = 1}^{x-1} (\alpha^{(x')})^*_{j'} - \sum_{x' = x+1}^{\mathcal{k}} (\alpha^{(x')})^*_{j'},
\end{align}
where $(\alpha^{(x')})^*_{j'} := |\lbrace i_0 : \alpha^{(x')}_{i_0} \geq j' \rbrace|$, and 
\begin{align} \label{attachy}
Y_{i_{(x,i')}, j'+1} := Y^{(x)}_{i', j'}.
\end{align}

Finally, it returns $(X, Y)$.  

Henceforward we adopt the notation of Equation~\ref{attachx} and denote $|\lbrace i : \alpha_i \geq j \rbrace|$ by $\alpha^*_j$ for all integer sequences $\alpha$, regardless of whether $\alpha$ is a partition.  

\begin{exam} \label{cont}
Maintain the notation of Example~\ref{surv}.  We proceed to compute $\mathcal{A}(\alpha, \nu)$.  

Since $\sigma := \mathcal{R}_{-1}(\alpha, \nu) = 42135$ and $\iota := \mathcal{U}_{-1}(\alpha, \nu, \sigma) = [4,4,4,4,4]$, we see that \[ [X_{1, 1}, X_{2, 1}, X_{3, 1}, X_{4, 1}, X_{5, 1}] = [4, 4, 4, 4, 4]\] and \[ [Y_{1, 1}, Y_{2, 1}, Y_{3, 1}, Y_{4, 1}, Y_{5, 1}] = [8, 6, 4, 2, 0].\]  
	
Set $f := \mathcal{S}(\alpha, \sigma, \iota)$.  Recall from Example~\ref{surv} that \[\big(f(1), f(2), f(3), f(4), f(5)\big) = \big((1,1), (1,2), (1,0), (1,3), (1,0)\big).\]  

Thus, $\mathcal{k} = 1$ and $\ell_1 = 3$.  Furthermore, $(i_{(1,1)}, i_{(1,2)}, i_{(1,3)}) = (1, 2, 4)$.  It follows that \[\alpha^{(1)} = [1, 2, 3] \quad \text{and} \quad \hat{\nu}^{(1)} = \nu^{(1)} = [5, 10, 11].\]

(Since the first branch is the only branch, no adjustment to $\nu^{(1)}$ is required and $\hat{\nu}^{(1)} = \nu^{(1)}$.)  

As it happens, we find that $X^{(1)}$ and $Y^{(1)}$ look as depicted in Figure~\ref{firstbranch}.  
\begin{figure}[h]
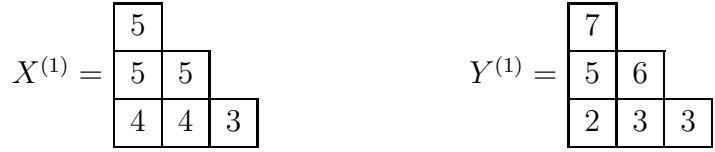

	\ytableausetup{centertableaux, textmode}
	\hspace{0.25in}
	$X^{(1)} =$ \begin{ytableau}
		5\\
		5 & 5\\
		4 & 4 & 3
	\end{ytableau}
	\hspace{1in}
	$Y^{(1)} = $ \begin{ytableau}
		7\\
		5 & 6\\
		2 & 3 & 3
	\end{ytableau}
	\caption{The diagram pair obtained from the first branch}
	\label{firstbranch}
\end{figure}

Finally, we ``attach'' $X^{(1)}$ and $Y^{(1)}$ to the first columns of $X$ and $Y$, respectively, to complete the output.  
	
\begin{figure}[h]
	\ytableausetup{centertableaux, textmode}
	\hspace{0.5in}
	$X =$ \begin{ytableau}
		4 & 5\\
		4 & 5 & 5\\
		4\\
		4 & 4 & 4 & 3\\
		4
	\end{ytableau}
	\hspace{1in}
	$Y = $ \begin{ytableau}
		8 & 7\\
		6 & 5 & 6\\
		4\\
		2 & 2 & 3 & 3\\
		0
	\end{ytableau}
	\caption{The diagram pair obtained from the recursion tree}
\end{figure}

\end{exam}

Since Example~\ref{cont} involves only one branch, it doesn't fully illustrate the contours of the algorithm.  For this reason, we also show how the algorithm computes $X^{(1)}$ and $Y^{(1)}$ in Example~\ref{cont}, during which we encounter multiple branches.  

\begin{exam}

Set $\alpha := [1,2,3]$ and $\nu := [5,10,11]$.  We compute $(\mathsf{X}, \mathsf{Y}) := \mathcal{A}(\alpha, \nu, 1)$.  

We find $\sigma := \mathcal{R}_1(\alpha, \nu) = 123$ and $\iota := \mathcal{U}_1(\alpha, \nu, \sigma) = [5,5,4]$, so \[[\mathsf{X}_{1,1}, \mathsf{X}_{2,1}, \mathsf{X}_{3,1}] = [5,5,4] \quad \text{and} \quad [\mathsf{Y}_{1,1}, \mathsf{Y}_{2,1}, \mathsf{Y}_{3,1}] = [7,5,2].\]
	
Set $\mathsf{f} := \mathcal{S}(\alpha, \sigma, \iota)$.  Note that \[\big(\mathsf{f}(1), \mathsf{f}(2), \mathsf{f}(3)\big) = \big((1,0), (1,1), (2,1)\big).\]
	
Thus, $\mathcal{k} = 2$ and $\ell_1 = \ell_2 = 1$.  Furthermore, $i_{(1,1)} = 2$ and $i_{(2,1)} = 3$.  It follows that \[(\alpha^{(1)}, \nu^{(1)}) = ([1], [5]) \quad \text{and} \quad (\alpha^{(2)}, \nu^{(2)}) = ([2], [7]).\]  

Hence \[(\alpha^{(1)}, \hat{\nu}^{(1)}) = ([1], [6]) \quad \text{and} \quad (\alpha^{(2)}, \hat{\nu}^{(2)}) = ([2], [6]).\]

We draw the diagram pairs $(\mathsf{X}^{(1)}, \mathsf{Y}^{(1)})$ and $(\mathsf{X}^{(2)}, \mathsf{Y}^{(2)})$ below.  

\begin{figure}[h]
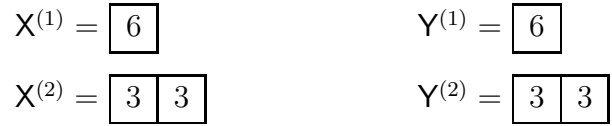

\ytableausetup{centertableaux, textmode}
\hspace{-0.25in}
$\mathsf{X}^{(1)} = $ \begin{ytableau}
	6
\end{ytableau}
\hspace{1.25in}
$\mathsf{Y}^{(1)} = $ \begin{ytableau}
	6
\end{ytableau} \\ \vspace{0.1in}
\ytableausetup{centertableaux, textmode}
\hspace{0.00in}
$\mathsf{X}^{(2)} = $ \begin{ytableau}
	3 & 3
\end{ytableau}
\hspace{1in}
$\mathsf{Y}^{(2)} = $ \begin{ytableau}
	3 & 3
\end{ytableau}
\caption{The diagram pairs obtained from the first and second branches}
\end{figure}

Finally, we ``attach'' these diagrams to the first columns computed above to complete the output.  

\begin{figure}[h]
	\ytableausetup{centertableaux, textmode}
	\hspace{0.25in}
	$\mathsf{X} = $ \begin{ytableau}
		5\\
		5 & 5\\
		4 & 4 & 3
	\end{ytableau}
	\hspace{1in}
	$\mathsf{Y} = $ \begin{ytableau}
		7\\
		5 & 6\\
		2 & 3 & 3
	\end{ytableau} 
\caption{The diagram pair obtained from the recursion tree}
\end{figure}

\textit{Nota bene.} Equation~\ref{attachx} dictates that the entries of $\mathsf{X}^{(1)}$ and $\mathsf{X}^{(2)}$ must be modified before they can be adjoined to $\mathsf{X}$, but the entries of $\mathsf{Y}^{(1)}$ and $\mathsf{Y}^{(2)}$ are adjoined to $Y$ as they are.  
\end{exam}

\subsection{Properties}

The following propositions delineate a few properties of $\mathcal{A}$.

\begin{prop} \label{permute}
Let $(\beta, \xi), (\alpha, \nu) \in \mathbb{N}^{\ell} \times \mathbb{Z}^{\ell}$, and suppose that the multisets \[\lbrace (\beta_1, \xi_1), \ldots, (\beta_{\ell}, \xi_{\ell}) \rbrace \quad \text{and} \quad \lbrace (\alpha_1, \nu_1), \ldots, (\alpha_{\ell}, \nu_{\ell}) \rbrace\] are coincident.  Then $\mathcal{A}(\beta, \xi, \pm 1) = \mathcal{A}(\alpha, \nu, \pm 1)$.  
\end{prop}

\begin{proof}
We prove $\mathcal{A}(\beta, \xi, -1) = \mathcal{A}(\alpha, \nu, -1)$.  

Set $\tau := \mathcal{R}_{-1}(\beta, \xi)$ and $\sigma := \mathcal{R}_{-1}(\alpha, \nu)$.  It suffices to show that 
\begin{align} \label{perm}
\beta_{\tau^{-1}(i)} = \alpha_{\sigma^{-1}(i)} \quad \text{and} \quad \xi_{\tau^{-1}(i)} = \nu_{\sigma^{-1}(i)}  
\end{align}
for all $1 \leq i \leq \ell$.  The proof is by induction on $i$; we show the inductive step.  In other words, we assume that Equation~\ref{perm} holds for all $1 \leq i_0 \leq i-1$ and show it holds for $i$.  

Set $J := \tau^{-1} \lbrace 1, \ldots, i-1 \rbrace$ and $J' := \lbrace 1, \ldots, \ell \rbrace \setminus J$.  Also set $I := \sigma^{-1} \lbrace 1, \ldots, i-1 \rbrace$ and $I' := \lbrace 1, \ldots, \ell \rbrace \setminus I$.  By the inductive hypothesis, the multisets \[\lbrace (\beta_j, \xi_j) \rbrace_{j \in J} \quad \text{and} \quad \lbrace (\alpha_j, \nu_j) \rbrace_{j \in I} \] are coincident.  

Therefore, the multisets \[\lbrace (\beta_j, \xi_j) \rbrace_{j \in J'} \quad \text{and} \quad \lbrace (\alpha_j, \nu_j) \rbrace_{j \in I'} \] are coincident, and there exists a bijection $\zeta \colon I' \rightarrow J'$ such that $(\beta_j, \xi_j) = (\alpha_{\zeta^{-1}(j)}, \nu_{\zeta^{-1}(j)})$ for all $j \in J'$.  

Then \[(\mathcal{C}_{-1}(\beta, \xi, j, J, J' \setminus \lbrace j \rbrace), \beta_j, \xi_j) = (\mathcal{C}_{-1}(\alpha, \nu, \zeta^{-1}(j), I, I' \setminus \lbrace \zeta^{-1}(j) \rbrace), \alpha_{\zeta^{-1}(j)}, \nu_{\zeta^{-1}(j)})\] for all $j \in J'$.  

Hence the lexicographically maximal value of the function \[j \mapsto (\mathcal{C}_{-1}(\beta, \xi, j, J, J' \setminus \lbrace j \rbrace), \beta_j, \xi_j)\] over the domain $J'$ coincides with the lexicographically maximal value of the function \[j \mapsto (\mathcal{C}_{-1}(\alpha, \nu, j, I, I' \setminus \lbrace j \rbrace), \alpha_j, \nu_j) \] over the domain $I'$.  

By definition of $\mathcal{R}_{-1}$, the former value is attained at $j = \tau^{-1}(i)$, and the latter value is attained at $j = \sigma^{-1}(i)$.  The result follows.   

\end{proof}

\begin{prop} \label{eworks}
Let $(\alpha, \nu, \epsilon) \in \mathbb{N}^{\ell} \times \mathbb{Z}^{\ell} \times \lbrace \pm 1 \rbrace$.  Then $\mathcal{A}(\alpha, \nu, \epsilon) \in \overline{E}(D_{\operatorname{dom}(\alpha)})$.    
\end{prop}

\begin{proof}
Set $(X, Y) := \mathcal{A}(\alpha, \nu, \epsilon)$.  We first show $(X, Y) \in D_{\operatorname{dom}(\alpha)} \times D_{\operatorname{dom}(\alpha)}$.  The proof is by induction on $\max \lbrace \alpha_1, \ldots, \alpha_{\ell} \rbrace$; we show the inductive step.  Since $X$ has an entry in the $i^{\text{th}}$ row and $j^{\text{th}}$ column if and only if $Y$ does for all $(i, j) \in \mathbb{N} \times \mathbb{N}$, it suffices to show $X \in D_{\operatorname{dom}(\alpha)}$.  By construction, the first column of $X$ has $\alpha^*_1$ entries.  Applying the inductive hypothesis, viz., $(X^{(x)}, Y^{(x)}) \in D_{\operatorname{dom}(\alpha^{(x)})} \times D_{\operatorname{dom}(\alpha^{(x)})}$, we find the $(j'+1)^{\text{th}}$ column of $X$ has $\sum_{x=1}^{\mathcal{k}} (\alpha^{(x)})^*_{j'} = \alpha^*_{j'+1}$ entries for all $1 \leq j' \leq \max \lbrace \alpha_1, \ldots, \alpha_{\ell} \rbrace -1$.  We conclude that $X$ is of shape $\alpha$.  

To see that $(X, Y) \in \overline{E}(D_{\operatorname{dom}(\alpha)})$, we show $EX = Y$.  Again the proof is by induction on $\max \lbrace \alpha_1, \ldots, \alpha_{\ell} \rbrace$, and we show the inductive step.  By construction, $X^1_i + \alpha^*_1 - 2i + 1 = Y^1_i$ for all $1 \leq i \leq \ell$.  Let $(i', j') \in \mathbb{N} \times \mathbb{N}$ such that $X^{(x)}$ and $Y^{(x)}$ each have an entry in the ${i'}^{\text{th}}$ row and ${j'}^{\text{th}}$ column, and set $\mathcal{i}$ so that $X^{(x)}_{i', j'} = (X^{(x)})^{j'}_{\mathcal{i}}$.  Note that \[X_{i_{(x, i')}, j'+1} = X^{j'+1}_{\mathcal{i} + \sum_{x'=1}^{x-1} (\alpha^{(x')})^*_{j'}}.\]

Therefore, 
\begin{align*}
EX_{i_{(x, i')}, j'+1} & = X_{i_{(x, i')}, j'+1} + \alpha^*_{j'+1} - 2\left(\mathcal{i} + \sum_{x'=1}^{x-1} (\alpha^{(x')})^*_{j'}\right) + 1 \\ & = X^{(x)}_{i', j'} + \sum_{x'=1}^{x-1} (\alpha^{(x')})^*_{j'} - \sum_{x' = x+1}^{\mathcal{k}} (\alpha^{(x')})^*_{j'} + \alpha^*_{j'+1} - 2\mathcal{i}- 2 \sum_{x'=1}^{x-1} (\alpha^{(x')})^*_{j'} + 1 \\ & = X^{(x)}_{i', j'} - \sum_{x'=1}^{x-1} (\alpha^{(x')})^*_{j'} - \sum_{x' = x+1}^{\mathcal{k}} (\alpha^{(x')})^*_{j'} +  \sum_{x'=1}^{\mathcal{k}} (\alpha^{(x')})^*_{j'} - 2\mathcal{i} + 1 \\ & = X^{(x)}_{i', j'} + (\alpha^{(x)})^*_{j'} - 2\mathcal{i} + 1 = Y^{(x)}_{i', j'} = Y_{i_{(x,i')}, j'+1},
\end{align*}
where the second-to-last equality follows from the inductive hypothesis.  
\end{proof}

For a diagram $X \in D_{\ell}$, let $\#(X,i)$ be the number of boxes in the $i^{\text{th}}$ row of $X$, and set $\Sigma(X, i) := \sum_{j=1}^{\#(X,i)} X_{i,j}$.  

\begin{prop} \label{multi}
Set $(X, Y) := \mathcal{A}(\alpha, \nu, \epsilon)$.  For all $1 \leq i \leq \ell$, set $\beta_i := \#(X,i)$ and $\xi_i := \Sigma(X,i)$.  Then the multisets \[\lbrace (\beta_1, \xi_1), \ldots, (\beta_{\ell}, \xi_{\ell}) \rbrace \quad \text{and} \quad \lbrace (\alpha_1, \nu_1), \ldots, (\alpha_{\ell}, \nu_{\ell}) \rbrace\] are coincident.  
\end{prop}

\begin{proof}
We prove the assertion for $\epsilon = -1$.  The proof is by induction on $\max \lbrace \alpha_1, \ldots, \alpha_{\ell} \rbrace$; we show the inductive step.  By Proposition~\ref{permute}, we may assume $\mathcal{R}_{-1}(\alpha, \nu) = \operatorname{id}$ without loss of generality.  

Let $i \in \lbrace 1, \ldots, \ell \rbrace$.  Set $\iota := \mathcal{U}_{-1}(\alpha, \nu, \operatorname{id})$.  We first claim that $\alpha_i = 1$ entails $\iota_i = \nu_i$.  To see this, suppose $\alpha_i = 1$.  Then $\mathcal{C}_{-1}(\alpha, \nu, i, I_a, I_b) =  \nu_i - |I_a| + |I_b|$.  Thus, 
\begin{align*}
\iota_i & = \mathcal{C}_{-1}(\alpha, \nu, i, \lbrace 1, \ldots, i-1 \rbrace, \lbrace i+1, \ldots, \ell \rbrace) - \ell + 2i - 1 \\ & = \nu_i - (i-1) + (\ell -i) - \ell + 2i -1 = \nu_i
\end{align*}
unless $\nu_i > \iota_{i-1}$.  If indeed $\nu_i > \iota_{i-1}$, then let $i_0$ be minimal such that $\iota_{i-1} = \iota_{i_0}$.  Since $\mathcal{R}_{-1}(\alpha, \nu) = \operatorname{id}$, 
\begin{align*}
\nu_i + \ell - 2i_0 + 1 & = \mathcal{C}_{-1}(\alpha, \nu, i, \lbrace 1, \ldots, i_0-1 \rbrace, \lbrace i_0, i_0+1, \ldots, i-1, i+1, \ldots, \ell \rbrace) \\ & \leq \mathcal{C}_{-1}(\alpha, \nu, i_0, \lbrace 1, \ldots, i_0-1 \rbrace, \lbrace i_0+1, \ldots, \ell \rbrace) \\ & = \iota_{i_0} + \ell - 2i_0 + 1 < \nu_i + \ell - 2i_0 + 1.
\end{align*}
This is a contradiction, so $\iota_i = \nu_i$ for all $i$ such that $\alpha_i = 1$.  It follows that \[(\alpha_i, \nu_i) = (1, \iota_i) = (\beta_i, \xi_i)\] for all $i$ such that $\alpha_i = 1$.  

Thus, it suffices to show, for all $1 \leq x \leq \mathcal{k}$, that the multisets \[\left \lbrace \big (\beta_{i_{(x, 1)}}, \xi_{i_{(x, 1)}} \big ), \ldots, \big (\beta_{i_{(x, \ell_x)}}, \xi_{i_{(x, \ell_x)}} \big) \right \rbrace \] and \[ \left \lbrace \big (\alpha_{i_{(x, 1)}}, \nu_{i_{(x, 1)}} \big ), \ldots, \big (\alpha_{i_{(x, \ell_x)}}, \nu_{i_{(x, \ell_x)}} \big ) \right \rbrace \] are coincident.  For all $1 \leq i' \leq \ell_x$, set $\beta^{(x)}_{i'} := \#(X^{(x)}, i')$ and $\xi^{(x)}_{i'} := \Sigma(X^{(x)}, i')$.  Note that 
\begin{align*}
\xi_{i_{(x, i')}} & = \iota_{i_{(x, i')}} + \xi^{(x)}_{i'} + \sum_{j' =1}^{\beta^{(x)}_{i'}} \left(\sum_{x' = 1}^{x-1} (\alpha^{(x')})^*_{j'} - \sum_{x' = x+1}^{\mathcal{k}} (\alpha^{(x')})^*_{j'}\right) \\ & = \iota_{i_{(x, i')}} + \xi^{(x)}_{i'} + \sum_{x' = 1}^{x-1} \sum_{i_0 = 1}^{\ell_{x'}} \min \left \lbrace \beta^{(x)}_{i'}, \alpha^{(x')}_{i_0} \right \rbrace - \sum_{x' = x+1}^{\mathcal{k}} \sum_{i_0 =1}^{\ell_{x'}} \min \left \lbrace \beta^{(x)}_{i'}, \alpha^{(x')}_{i_0} \right \rbrace.
\end{align*}

Therefore, as multisets, 
\begin{align*}
& \left \lbrace \big (\beta_{i_{(x, i')}}, \xi_{i_{(x, i')}} \big ) \right \rbrace_{i' =1}^{\ell_x} \\ & = \left \lbrace \Big (1 + \beta^{(x)}_{i'}, \iota_{i_{(x, i')}} + \xi^{(x)}_{i'} + \sum_{x' = 1}^{x-1} \sum_{i_0 = 1}^{\ell_{x'}} \min \left \lbrace \beta^{(x)}_{i'}, \alpha^{(x')}_{i_0} \right \rbrace - \sum_{x' = x+1}^{\mathcal{k}} \sum_{i_0 =1}^{\ell_{x'}} \min \left \lbrace \beta^{(x)}_{i'}, \alpha^{(x')}_{i_0} \right \rbrace \Big) \right \rbrace_{i'=1}^{\ell_x} \\ & = \left \lbrace \Big (1 + \alpha^{(x)}_{i'}, \iota_{i_{(x, i')}} + \hat{\nu}^{(x)}_{i'} + \sum_{x' = 1}^{x-1} \sum_{i_0 = 1}^{\ell_{x'}} \min \left \lbrace \alpha^{(x)}_{i'}, \alpha^{(x')}_{i_0} \right \rbrace - \sum_{x' = x+1}^{\mathcal{k}} \sum_{i_0 =1}^{\ell_{x'}} \min \left \lbrace \alpha^{(x)}_{i'}, \alpha^{(x')}_{i_0} \right \rbrace \Big) \right \rbrace_{i'=1}^{\ell_x} \\ & = \left \lbrace \big (\alpha_{i_{(x, i')}}, \iota_{i_{(x, i')}} + \nu^{(x)}_{i'} \big) \right \rbrace_{i'=1}^{\ell_x} = \left \lbrace \big (\alpha_{i_{(x, i')}}, \nu_{i_{(x, i')}} \big) \right \rbrace_{i'=1}^{\ell_x},
\end{align*}
where we obtain the second equality by recalling $\iota_{i_{(x, 1)}} = \cdots = \iota_{i_{(x, \ell_x)}}$ and applying the inductive hypothesis.  
\end{proof}

\begin{cor} \label{multiforward}
Set $\sigma := \mathcal{R}_{-1}(\alpha, \nu)$, and $\iota := \mathcal{U}_{-1}(\alpha, \nu, \sigma)$.  Then the multisets \[\left \lbrace (\beta_i - 1, \xi_i - \iota_i) : \beta_i > 1 \right \rbrace \quad \text{and} \quad \left \lbrace (\alpha_i - 1, \nu_i - \iota_{\sigma(i)}) : \alpha_i > 1 \right \rbrace \] are coincident.  
\end{cor}

\begin{proof}
Maintain the notation of Proposition~\ref{multi}.  By Proposition~\ref{multi}, for all $1 \leq x \leq \mathcal{k}$, \[\left \lbrace \big(\beta^{(x)}_{i'}, \xi^{(x)}_{i'}\big) \right \rbrace_{i'=1}^{\ell_x} = \left \lbrace \big(\alpha^{(x)}_{i'}, \hat{\nu}^{(x)}_{i'}\big) \right\rbrace_{i'=1}^{\ell_x}\] is an equality of multisets.  

Therefore, as multisets,
\begin{align*}
& \left \lbrace \big (\beta_{i_{(x, i')}} - 1, \xi_{i_{(x, i')}} - \iota_{i_{(x,i')}} \big ) \right \rbrace_{i' =1}^{\ell_x} \\ & = \left \lbrace \Big (\beta^{(x)}_{i'}, \xi^{(x)}_{i'} + \sum_{x' = 1}^{x-1} \sum_{i_0 = 1}^{\ell_{x'}} \min \left \lbrace \beta^{(x)}_{i'}, \alpha^{(x')}_{i_0} \right \rbrace - \sum_{x' = x+1}^{\mathcal{k}} \sum_{i_0 =1}^{\ell_{x'}} \min \left \lbrace \beta^{(x)}_{i'}, \alpha^{(x')}_{i_0} \right \rbrace \Big) \right \rbrace_{i'=1}^{\ell_x} \\ & = \left \lbrace \Big (\alpha^{(x)}_{i'}, \hat{\nu}^{(x)}_{i'} + \sum_{x' = 1}^{x-1} \sum_{i_0 = 1}^{\ell_{x'}} \min \left \lbrace \alpha^{(x)}_{i'}, \alpha^{(x')}_{i_0} \right \rbrace - \sum_{x' = x+1}^{\mathcal{k}} \sum_{i_0 =1}^{\ell_{x'}} \min \left \lbrace \alpha^{(x)}_{i'}, \alpha^{(x')}_{i_0} \right \rbrace \Big) \right \rbrace_{i'=1}^{\ell_x} \\ & = \left \lbrace \big (\alpha^{(x)}_{i'}, \nu^{(x)}_{i'} \big) \right \rbrace_{i'=1}^{\ell_x} = \left \lbrace \big (\alpha_{\sigma^{-1}(i_{(x, i')})} - 1, \nu_{\sigma^{-1}(i_{(x, i')})} - \iota_{i_{(x, i')}} \big) \right \rbrace_{i'=1}^{\ell_x}.  
\end{align*}

Taking the union of both sides over $1 \leq x \leq \mathcal{k}$, we obtain the equality of multisets \[\left \lbrace(\beta_i - 1, \xi_i - \iota_i) : \beta_i > 1 \right \rbrace = \left \lbrace (\alpha_{\sigma^{-1}(i)} -1, \nu_{\sigma^{-1}(i)} - \iota_i) : \alpha_{\sigma^{-1}(i)} > 1 \right \rbrace,\] whence the result follows.  
\end{proof}

Fix again a partition $\alpha = [\alpha_1, \ldots, \alpha_{\ell}]$ of $n$ with conjugate partition $\alpha^* = [\alpha^*_1, \ldots, \alpha^*_s]$.  Recall from section 3 that a diagram pair $(X, Y) \in \overline{E}(D_{\alpha})$ encodes three sequences related to objects in $\mathfrak{D}$ --- $\kappa(X)$, $h(X)$, and $\eta(Y)$.  If $(X,Y)$ is an output of the weight-diagrams version of the algorithm, then all three sequences carry crucial information: $\kappa(X)$ returns the input; $h(X)$ is a weight in $\Lambda^+_{\alpha, \nu}$ such that $||h(X) + 2 \rho_{\alpha}||$ is minimal, and $\eta(Y)$ is the output of the Lusztig--Vogan bijection.  

\begin{prop} \label{kap}
Let $\nu \in \Omega_{\alpha}$.  Then $\kappa p_1 \mathcal{A}(\alpha, \nu) = \nu$.  (Hence $hp_1\mathcal{A}(\alpha, \nu) \in \Lambda^+_{\alpha, \nu}$.)  
\end{prop}

\begin{proof}
This is a direct consequence of Proposition~\ref{multi}.  (The statement in parentheses follows from Theorem~\ref{decomp}.)    
\end{proof}

\begin{thm} \label{mini}
Let $\nu \in \Omega_{\alpha}$.  Then $||hp_1\mathcal{A}(\alpha, \nu) + 2 \rho_{\alpha}|| = \min \lbrace ||\bar{\mu} + 2 \rho_{\alpha}|| : \bar{\mu} \in \Lambda^+_{\alpha, \nu} \rbrace$.  
\end{thm}

\begin{cor} \label{etz}
Let $\nu \in \Omega_{\alpha}$.  Then $\operatorname{dom}(hp_1 \mathcal{A}(\alpha, \nu) + 2 \rho_{\alpha}) = \gamma(\alpha, \nu)$.  
\end{cor}

\begin{proof}
This follows from (the parenthetical statement in) Proposition~\ref{kap} and Theorem~\ref{mini}, as discussed in the introduction (cf. footnote 3).  
\end{proof}

\begin{cor} \label{eta}
Let $\nu \in \Omega_{\alpha}$.  Then $\eta p_2 \mathcal{A}(\alpha, \nu) = \gamma(\alpha, \nu)$.  
\end{cor}

\begin{proof}
Recall from Equation~\ref{compat} that $\eta p_2 \mathcal{A}(\alpha, \nu) = \operatorname{dom}(hp_1 \mathcal{A}(\alpha, \nu) + 2 \rho_{\alpha})$.\footnote{Equation~\ref{compat} holds under the assumption that the entries of $X$ are weakly decreasing down each column, which is certainly the case if $X$ is distinguished (cf. condition (4) of Definition~\ref{dis}).  We prove that $p_1 \mathcal{A}(\alpha, \nu)$ is distinguished in the next section (cf. Corollary~\ref{dist}).  } 
\end{proof}  

Corollaries~\ref{etz} and ~\ref{eta} express that $\mathcal{A}$ computes the Lusztig--Vogan bijection.  Just as Corollary~\ref{etz} follows from Proposition~\ref{kap} and Theorem~\ref{mini}, that $\mathfrak{A}$ computes the Lusztig--Vogan bijection (as expressed in Theorem~\ref{main}) follows from Corollary~\ref{inside} and the following theorem, which we deduce from Theorem~\ref{mini}.    

\begin{thm} \label{thesame}
Let $\nu \in \Omega_{\alpha}$.  Then $||\mathfrak{A}(\alpha, \nu) + 2 \rho_{\alpha}|| = \min \lbrace ||\bar{\mu} + 2 \rho_{\alpha} || : \bar{\mu} \in \Lambda^+_{\alpha, \nu} \rbrace$.  
\end{thm}

\begin{proof}
Assume Theorem~\ref{mini} holds.  Set $\mu := \mathfrak{A}(\alpha, \nu)$ and $\breve{\mu} := hp_1 \mathcal{A}(\alpha, \nu)$.  Recall from Corollary~\ref{inside} that $\mu \in \Lambda^+_{\alpha, \nu}$ and from Proposition~\ref{kap} that $\breve{\mu} \in \Lambda^+_{\alpha, \nu}$.  Note that 
\begin{align} \label{first}
[\mu_1, \ldots, \mu_{\ell}] = [\breve{\mu}_1, \ldots, \breve{\mu}_{\ell}].
\end{align}

In other words, the first $\ell$ entries of $\mathfrak{A}(\alpha, \nu)$ agree with the first $\ell$ entries of $hp_1 \mathcal{A}(\alpha, \nu)$; both $\mathfrak{A}$ and $hp_1\mathcal{A}$ assign the same weight to the first factor $GL_{\alpha^*_1}$ of $L_{\alpha}$.  

To prove the theorem, we induct on $s$.  For the inductive step, set $\mu' := [\mu_{\ell + 1}, \ldots, \mu_n]$ and $\breve{\mu}' := [\breve{\mu}_{\ell+1}, \ldots, \breve{\mu}_n]$.

Maintain the notation of sections 2.1 and 2.2.  By Corollary~\ref{inside}, $\mu' \in \Lambda^+_{\alpha', \nu'}$.  We claim that $\breve{\mu}' \in \Lambda^+_{\alpha',  \nu'}$.  Indeed, since $(\alpha_i - 1, \nu_i - \mu_{\sigma(i)}) = (\alpha'_i, \nu'_i)$ for all $1 \leq i \leq \alpha^*_2$, Corollary~\ref{multiforward} (in view of Equation~\ref{first}) tells us that the multisets \[\lbrace \big(\beta_i - 1, \xi_i - \breve{\mu}_i \big) : \beta_i > 1 \rbrace \quad \text{and} \quad \lbrace (\alpha'_i, \nu'_i) \rbrace_{i=1}^{\alpha^*_2}\] are coincident.  

Therefore, there exists a function $\zeta \colon \lbrace 1, \ldots, \alpha^*_2 \rbrace \rightarrow \lbrace 1, \ldots, \ell \rbrace$ such that $\beta_{\zeta(i)} > 1$ and $(\beta_{\zeta(i)} - 1, \xi_{\zeta(i)} - \breve{\mu}_{\zeta(i)}) = (\alpha'_i, \nu'_i)$ for all $1 \leq i \leq \alpha^*_2$.  For all $1 \leq i \leq \alpha^*_2$ and $1 \leq j \leq \alpha'_i$, set $\nu'_{i,j} := X_{\zeta(i), j+1}$.  Then the claim follows from Corollary~\ref{decamp}.  

Thus, by the inductive hypothesis, \[||\mu' + 2 \rho_{\alpha'}|| = \min \lbrace || \bar{\mu}' + 2 \rho_{\alpha'}|| : \bar{\mu}' \in \Lambda^+_{\alpha', \nu'} \rbrace \leq ||\breve{\mu}' + 2 \rho_{\alpha'}||.\]

It follows that
\begin{align*}
||\mu + 2 \rho_{\alpha}||^2 & = ||[\mu_1 + \ell - 1, \ldots, \mu_{\ell} + 1 - \ell]||^2 + ||\mu' + 2 \rho_{\alpha'}||^2 \\ & \leq ||[\breve{\mu}_1 + \ell - 1, \ldots, \breve{\mu}_{\ell} + 1 - \ell]||^2 + ||\breve{\mu}' + 2 \rho_{\alpha'}||^2 \\ & = ||\breve{\mu} + 2 \rho_{\alpha}||^2 \\ & = \min \lbrace ||\bar{\mu} + 2 \rho_{\alpha}||^2 : \bar{\mu} \in \Lambda^+_{\alpha, \nu} \rbrace \\ & \leq ||\mu + 2 \rho_{\alpha}||^2,
\end{align*}
where the first inequality follows from Equation~\ref{first} and the third equality follows from Theorem~\ref{mini}.  

We conclude that $||\mu + 2 \rho_{\alpha}|| = \min \lbrace ||\bar{\mu} + 2 \rho_{\alpha}|| : \bar{\mu} \in \Lambda^+_{\alpha, \nu} \rbrace$, as desired.  
\end{proof}

It remains is to prove Theorem~\ref{mini}.  In the next section, we make good on our pledge to prove that the weight-diagrams version of our algorithm encompasses Achar's algorithm; in particular, we prove the following theorem, whence Theorem~\ref{mini} follows immediately.  

\begin{thm} \label{mata}
Let $\nu \in \Omega_{\alpha}$.  Then $p_1 \mathcal{A}(\alpha, \nu) = \mathsf{A}(\alpha, \nu)$.  
\end{thm}

\begin{cor}
Theorem~\ref{mini} holds.  	
\end{cor}

\begin{proof}
Note that $||h\mathsf{A}(\alpha, \nu) + 2 \rho_{\alpha}|| = \min \lbrace ||\bar{\mu} + 2 \rho_{\alpha}|| : \bar{\mu} \in \Lambda^+_{\alpha, \nu} \rbrace$ (cf. Achar \cite{Acharj},  Corollary 8.9).  
\end{proof}

\vfill \eject

\section{Proof of Theorem~\ref{mata}}

The crux of the proof is a simple characterization of the diagram pairs that occur as outputs of the algorithm $\mathcal{A}$.  These \textit{distinguished} diagram pairs are images under $\overline{E}$ of the \textit{distinguished} diagrams of Achar \cite{Acharj}.  We start by defining distinguished diagrams and diagram pairs.  

\begin{df}
Let $Y$ be a diagram of shape-class $\alpha$.  The entry $Y^j_i$ is \textit{$E$-raisable} if $i=1$, or if $i > 1$ and $Y^j_{i-1} > Y^j_i + 2$.  The entry $Y^j_i$ is \textit{$E$-lowerable} if $i = \alpha^*_j$, or if $i < \alpha^*_j$ and $Y^j_{i+1} < Y^j_i - 2$.  
\end{df}

\begin{df} \label{dis}
Let $X \in D_{\alpha}$, and set $Y := EX$.  Then the diagram $X$ and the diagram pair $(X,Y)$ are \textit{odd-distinguished} if the following four conditions hold.
\begin{enumerate}
	\item $Y_{i,j+1} - Y_{i,j} \in \lbrace 0, (-1)^j \rbrace$.
	\item For all $1 \leq j < j' \leq s$ such that $j \equiv j' \pmod 2$:
	\begin{enumerate}
		\item If $j$ and $j'$ are odd and $Y_{i,j} \leq Y_{i,j'} - 1$, then $Y_{i,j}$ is not $E$-raisable;
		\item If $j$ and $j'$ are even and $Y_{i,j} \geq Y_{i, j'} + 1$, then $Y_{i,j}$ is not $E$-lowerable.  
	\end{enumerate}
	\item For all $1 \leq j < j' \leq s$:
	\begin{enumerate}
		\item If $Y_{i,j} \leq Y_{i,j'} - 2$, then $Y_{i,j}$ is not $E$-raisable;
		\item If $Y_{i,j} \geq Y_{i,j'} + 2$, then $Y_{i,j}$ is not $E$-lowerable.  
	\end{enumerate}
	\item $Y^j_i - Y^j_{i+1} \geq 2$.
\end{enumerate}
\end{df}

\begin{df}
	Let $X \in D_{\alpha}$, and set $Y := EX$.  Then the diagram $X$ and the diagram pair $(X,Y)$ are \textit{even-distinguished} if the following four conditions hold.
	\begin{enumerate}
		\item $Y_{i,j+1} - Y_{i,j} \in \lbrace 0, (-1)^{j+1} \rbrace$.
		\item For all $1 \leq j < j' \leq s$ such that $j \equiv j' \pmod 2$:
		\begin{enumerate}
			\item If $j$ and $j'$ are even and $Y_{i,j} \leq Y_{i,j'} - 1$, then $Y_{i,j}$ is not $E$-raisable;
			\item If $j$ and $j'$ are odd and $Y_{i,j} \geq Y_{i, j'} + 1$, then $Y_{i,j}$ is not $E$-lowerable.  
		\end{enumerate}
		\item For all $1 \leq j < j' \leq s$:
		\begin{enumerate}
			\item If $Y_{i,j} \leq Y_{i,j'} - 2$, then $Y_{i,j}$ is not $E$-raisable;
			\item If $Y_{i,j} \geq Y_{i,j'} + 2$, then $Y_{i,j}$ is not $E$-lowerable.  
		\end{enumerate}
		\item $Y^j_{i} - Y^j_{i+1} \geq 2$.
	\end{enumerate}	
\end{df}

We refer to odd-distinguished diagrams and diagram pairs as just \textit{distinguished}.  

\begin{rem} \label{weak}
The definition of distinguished diagram in Achar \cite{Acharj} is weaker than ours inasmuch as it requires $Y^j_i - Y^j_{i+1} \geq 1$ rather than $Y^j_{i} - Y^j_{i+1} \geq 2$.  However, Achar's definition of the $E$ map differs slightly from ours, so it does not suffice for our purposes to copy his definition wholesale.  Our definition of distinguished ensures that, if $(X,Y)$ is distinguished, then $Y = EX$ under Achar's definition as well as ours --- so it guarantees that any diagram distinguished by our reckoning is distinguished by Achar's \textit{a fortiori}.  
\end{rem}

To simplify our analysis of the algorithm, we define the \textit{row-partition} function, which is similar to the row-survival function, but does not discriminate between surviving and non-surviving rows.  

\begin{df}
For all $(\alpha, \iota) \in \mathbb{N}^{\ell} \times \mathbb{Z}^{\ell}_{\operatorname{dom}}$, \[\mathcal{P}(\alpha, \iota) \colon \lbrace 1, \ldots, \ell \rbrace \rightarrow \lbrace 1, \ldots, \ell \rbrace \times \lbrace 1, \ldots, \ell \rbrace\] is given by \[\mathcal{P}(\alpha, \iota)(i) :=\big(|\lbrace \iota_{i'} : i' \leq i \rbrace|, |\lbrace i' \leq i : \iota_{i'} = \iota_i \rbrace | \big).\]
\end{df}

\begin{rem}
After the rows of a blank diagram of shape $\alpha$ have been permuted according to $\sigma \in \mathfrak{S}_{\ell}$ and the first column of the permuted diagram is filled in with the entries of $\iota$, the row-survival function $\mathcal{S}(\alpha, \sigma, \iota)$ tells us, for each surviving row, which branch it belongs to, and its position within that branch.  (For each row that does not survive, the row-survival function still records a branch, but records its position within that branch as $0$.)  

If we construe each branch as comprising all rows with a particular first-column entry, not merely the surviving such rows, then, for each row, the row-partition function $\mathcal{P}(\alpha, \iota)$ tells us which branch it belongs to, and its position among all rows within that branch.  
\end{rem}

\begin{df}
Given diagrams $Z^{(1)}, \ldots, Z^{(k)}$ such that $Z^{(x)} \in D_{\ell_x}$ for all $1 \leq x \leq k$, construct $Z \in D_{\ell_1 + \cdots + \ell_k}$ as follows: For all $(i, j) \in \mathbb{N} \times \mathbb{N}$ such that $Z^{(x)}$ has an entry in the $i^{\text{th}}$ row and $j^{\text{th}}$ column, $Z_{i + \sum_{x'=1}^{x-1} \ell_{x'}, j} := Z^{(x)}_{i, j}$.  The \textit{diagram-concatenation} function $\operatorname{Cat} \colon D_{\ell_1} \times \cdots \times D_{\ell_k} \rightarrow D_{\ell_1 + \cdots + \ell_k}$ is given by $(Z^{(1)}, \ldots, Z^{(k)}) \mapsto Z$.  
\end{df}

\begin{rem}
Diagram concatenation is transitive: If $Z^{(x)} = \operatorname{Cat} (Z^{(x, 1)}, \ldots, Z^{(x, \omega_x)})$ for all $x$, then \[Z = \operatorname{Cat}\big(Z^{(1, 1)}, \ldots, Z^{(1, \omega_1)}, \ldots, Z^{(k, 1)}, \ldots, Z^{(k, \omega_k)}\big).\]   
\end{rem}

Let $(\alpha, \nu, \epsilon) \in \mathbb{N}^{\ell} \times \mathbb{Z}^{\ell} \times \lbrace \pm 1 \rbrace$.  Set $\sigma := \mathcal{R}_{\epsilon}(\alpha, \nu)$.  Set $\iota := \mathcal{U}_{\epsilon}(\alpha, \nu, \sigma)$.  For all $(x, i')$ in the image of $\mathcal{P}(\alpha, \iota)$, set $p_{(x,i')} := \mathcal{P}(\alpha, \iota)^{-1}(x, i')$.  

For all $1 \leq x \leq \mathcal{k}$, set \[\ell^{\circ}_x := \max \lbrace i' : (x,i') \in \mathcal{P}(\alpha, \iota) \lbrace 1, \ldots, \ell \rbrace \rbrace. \]

Set \[\mathcal{a}^{(x)} := \left[\alpha_{\sigma^{-1}(p_{(x,1)})}, \ldots, \alpha_{\sigma^{-1}(p_{(x,\ell^{\circ}_x)})}\right] \quad \text{and} \quad \mathcal{n}^{(x)} := \left[\nu_{\sigma^{-1}(p_{(x,1)})}, \ldots, \nu_{\sigma^{-1}(p_{(x,\ell^{\circ}_x)})}\right].\]

For all $1 \leq i' \leq \ell^{\circ}_x$, set \[\hat{\mathcal{n}}^{(x)}_{i'} := \mathcal{n}^{(x)}_{i'} - \sum_{x'=1}^{x-1} \sum_{i_0 = 1}^{\ell^{\circ}_{x'}} \min \left \lbrace \mathcal{a}^{(x)}_{i'}, \mathcal{a}^{(x')}_{i_0} \right \rbrace + \sum_{x'=x+1}^{\mathcal{k}} \sum_{i_0 = 1}^{\ell^{\circ}_{x'}} \min \left \lbrace \mathcal{a}^{(x)}_{i'}, \mathcal{a}^{(x')}_{i_0} \right \rbrace.\]

Then set $\hat{\mathcal{n}}^{(x)} := \left[ \hat{\mathcal{n}}^{(x)}_{1}, \ldots, \hat{\mathcal{n}}^{(x)}_{\ell^{\circ}_x} \right]$ and $\left (\mathcal{X}^{(x)}, \mathcal{Y}^{(x)} \right) := \mathcal{A} \left (\mathcal{a}^{(x)}, \hat{\mathcal{n}}^{(x)}, \epsilon \right)$.  

\begin{lem} \label{collapse}
Set $(X,Y) := \mathcal{A}(\alpha, \nu, \epsilon)$.  Then $Y = \operatorname{Cat}(\mathcal{Y}^{(1)}, \ldots, \mathcal{Y}^{(\mathcal{k})})$.  
\end{lem}  

\begin{proof}
We show the claim for $\epsilon = -1$ only.  By Proposition~\ref{permute}, we may assume $\sigma = \operatorname{id}$ without loss of generality.  

Fix $1 \leq x \leq \mathcal{k}$.  Set $\mathcal{s}_{(x)} := \mathcal{R}_{-1}(\mathcal{a}^{(x)}, \hat{\mathcal{n}}^{(x)})$ and $\mathcal{m}_{(x)} := \mathcal{U}_{-1}(\mathcal{a}^{(x)}, \hat{\mathcal{n}}^{(x)}, \mathcal{s}_{(x)})$.  We claim that $\mathcal{s}_{(x)} = \operatorname{id}$, which is to say that $\mathcal{s}_{(x)}(i') = i'$ for all $1 \leq i' \leq \ell^{\circ}_x$.  The proof is by induction on $i'$; we show the inductive step.  

Set $J := \lbrace 1, \ldots, i'-1 \rbrace$ and $J' := \lbrace 1, \ldots, \ell^{\circ}_x \rbrace \setminus J$.  For all $I \subset \lbrace 1, \ldots, \ell^{\circ}_x \rbrace$, set $p^{(x)}_I := \lbrace p_{(x, i_0)} : i_0 \in I \rbrace$.  Also set $\mathcal{Q}_{<x} := \lbrace p_{(x', i_0)} : x' < x \rbrace$ and $\mathcal{Q}_{>x} := \lbrace p_{(x', i_0)} : x' > x \rbrace$.  

For all $j \in J'$, 
\begin{align*}
& \mathcal{n}^{(x)}_{j} - \sum_{x'=1}^{x-1} \sum_{i_0 = 1}^{\ell^{\circ}_{x'}} \min \left \lbrace \mathcal{a}^{(x)}_{j}, \mathcal{a}^{(x')}_{i_0} \right \rbrace + \sum_{x'=x+1}^{\mathcal{k}} \sum_{i_0 = 1}^{\ell^{\circ}_{x'}} \min \left \lbrace \mathcal{a}^{(x)}_{j}, \mathcal{a}^{(x')}_{i_0} \right \rbrace \\ & - \sum_{i_0 \in J} \min \left \lbrace \mathcal{a}^{(x)}_{j}, \mathcal{a}^{(x)}_{i_0} \right \rbrace + \sum_{i_0 \in J' \setminus \lbrace j \rbrace} \min \left \lbrace \mathcal{a}^{(x)}_{j}, \mathcal{a}^{(x)}_{i_0} \right \rbrace \\ & = \nu_{p_{(x,j)}} - \sum_{i \in \mathcal{Q}_{<x} \cup p^{(x)}_J} \min \lbrace \alpha_{p_{(x,j)}}, \alpha_i \rbrace + \sum_{i \in \mathcal{Q}_{>x} \cup p^{(x)}_{J' \setminus \lbrace j \rbrace}} \min \lbrace \alpha_{p_{(x,j)}}, \alpha_i \rbrace. 
\end{align*}

Therefore, 
\begin{align} \label{pro}
\mathcal{C}_{-1}\big(\mathcal{a}^{(x)}, \hat{\mathcal{n}}^{(x)}, j, J, J' \setminus \lbrace j \rbrace\big) = \mathcal{C}_{-1}\big(\alpha, \nu, p_{(x,j)}, \mathcal{Q}_{<x} \cup p^{(x)}_J, \mathcal{Q}_{>x} \cup p^{(x)}_{J' \setminus \lbrace j \rbrace}\big).
\end{align}

Note that $\mathcal{Q}_{<x} \cup p^{(x)}_J = \lbrace 1, \ldots, p_{(x,i')} - 1 \rbrace$.  Since $\sigma(p_{(x,i')}) = p_{(x,i')}$, it follows that the lexicographically maximal value of the function \[j \mapsto \left(\mathcal{C}_{-1}\big(\alpha, \nu, p_{(x,j)}, \mathcal{Q}_{<x} \cup p^{(x)}_J, \mathcal{Q}_{>x} \cup p^{(x)}_{J' \setminus \lbrace j \rbrace}\big), \alpha_{p_{(x,j)}}, \nu_{p_{(x,j)}} \right)\] over the domain $J'$ is attained at $j = i'$.  

Thus, the lexicographically maximal value of the function \[j \mapsto \left(\mathcal{C}_{-1}\big(\mathcal{a}^{(x)}, \hat{\mathcal{n}}^{(x)}, j, J, J' \setminus \lbrace j \rbrace\big), \mathcal{a}^{(x)}_{j}, \mathcal{n}^{(x)}_{j} \right)\] over the domain $J'$ is attained at $j = i'$.  

Since $\hat{\mathcal{n}}^{(x)}_j - \mathcal{n}^{(x)}_j$ as a function of $j$ depends only on $\mathcal{a}^{(x)}_j$, the lexicographically maximal value of the function \[j \mapsto \left(\mathcal{C}_{-1}\big(\mathcal{a}^{(x)}, \hat{\mathcal{n}}^{(x)}, j, J, J' \setminus \lbrace j \rbrace\big), \mathcal{a}^{(x)}_{j}, \hat{\mathcal{n}}^{(x)}_{j} \right)\] over the domain $J'$ is also attained at $j = i'$.  

Since $i' \in J'$ is numerically minimal, it follows that $\mathcal{s}^{(x)}(i') = i'$, as desired. 

We next claim that 
\begin{align} \label{art}
X_{p_{(x,i')}, 1} = \mathcal{X}^{(x)}_{i', 1} + \sum_{x'=1}^{x-1} \ell^{\circ}_{x'} - \sum_{x'=x+1}^{\mathcal{k}} \ell^{\circ}_{x'}.
\end{align}

Again the proof is by induction on $i'$.  For all $1 \leq i' \leq \ell^{\circ}_x$, we see from Equation~\ref{pro} that
\begin{align*}
& \mathcal{C}_{-1}(\mathcal{a}^{(x)}, \hat{\mathcal{n}}^{(x)}, i', \lbrace 1, \ldots, i'-1 \rbrace, \lbrace i'+1, \ldots, \ell^{\circ}_x \rbrace) \\ & = \mathcal{C}_{-1}(\alpha, \nu, p_{(x,i')}, \lbrace 1, \ldots, p_{(x,i')} -1 \rbrace, \lbrace p_{(x,i')} + 1, \ldots, \ell \rbrace).
\end{align*}

Denote the value $\iota_{p_{(x,1)}} = \cdots = \iota_{p_{(x, \ell^{\circ}_x)}}$ by $\iota^{\circ}_x$.  
 
Note that \[\iota^{\circ}_x  = \mathcal{C}_{-1}(\alpha, \nu, p_{(x,1)}, \lbrace 1, \ldots, p_{(x,1)} -1 \rbrace, \lbrace p_{(x,1)} + 1, \ldots, \ell \rbrace) - \ell + 2p_{(x,1)} - 1,\] else $\iota^{\circ}_{x} = \iota^{\circ}_{x-1}$, which contradicts the definition of $\mathcal{P}(\alpha, \iota)$.  

Hence
\begin{align*}
\mathcal{X}_{1,1}^{(x)} & = \mathcal{C}_{-1}(\mathcal{a}^{(x)}, \hat{\mathcal{n}}^{(x)}, 1, \varnothing, \lbrace 2, \ldots, \ell^{\circ}_x \rbrace) - \ell^{\circ}_x + 1 \\ &  = \mathcal{C}_{-1}(\alpha, \nu, p_{(x,1)}, \lbrace 1, \ldots, p_{(x,1)} -1 \rbrace, \lbrace p_{(x,1)} + 1, \ldots, \ell \rbrace) - \ell^{\circ}_x + 1 \\ & = \iota^{\circ}_x + \ell - 2(p_{(x,1)} - 1) - \ell^{\circ}_x \\ & = X_{p_{(x,1)}, 1} - \sum_{x'=1}^{x-1} \ell^{\circ}_{x'} + \sum_{x' = x+1}^{\mathcal{k}} \ell^{\circ}_{x'}, 
\end{align*}
which proves the base case.  

For the inductive step, note that \[\iota^{\circ}_x  \leq \mathcal{C}_{-1}(\alpha, \nu, p_{(x,i')}, \lbrace 1, \ldots, p_{(x,i')} -1 \rbrace, \lbrace p_{(x,i')} + 1, \ldots, \ell \rbrace) - \ell + 2p_{(x,i')} - 1.\]

Thus, 
\begin{align*}
& \mathcal{C}_{-1}(\mathcal{a}^{(x)}, \hat{\mathcal{n}}^{(x)}, i', \lbrace 1, \ldots, i' -1 \rbrace, \lbrace i'+1, \ldots, \ell^{\circ}_x \rbrace) - \ell^{\circ}_x + 2i' - 1 \\ & = \mathcal{C}_{-1}(\alpha, \nu, p_{(x,i')}, \lbrace 1, \ldots, p_{(x,i')} -1 \rbrace, \lbrace p_{(x,i')} + 1, \ldots, \ell \rbrace) - \ell^{\circ}_x + 2i' - 1 \\ & \geq \iota^{\circ}_x + \ell - 2(p_{(x,i')} -i') - \ell^{\circ}_x \\ & = X_{p_{(x,i'-1)}, 1} - \sum_{x'=1}^{x-1} \ell^{\circ}_{x'} + \sum_{x' = x+1}^{\mathcal{k}} \ell^{\circ}_{x'} \\ & = \mathcal{X}^{(x)}_{i'-1, 1},
\end{align*}
where the last inequality follows from the inductive hypothesis.  

We conclude that
\begin{align*}
\mathcal{X}^{(x)}_{i', 1} = \mathcal{X}^{(x)}_{i'-1, 1} & = X_{p_{(x,i'-1)}, 1} - \sum_{x'=1}^{x-1} \ell^{\circ}_{x'} + \sum_{x' = x+1}^{\mathcal{k}} \ell^{\circ}_{x'} \\ & = X_{p_{(x,i')}, 1} - \sum_{x'=1}^{x-1} \ell^{\circ}_{x'} + \sum_{x' = x+1}^{\mathcal{k}} \ell^{\circ}_{x'}.
\end{align*}

This establishes Equation~\ref{art}, which implies $Y_{p_{(x,i')}, 1} = \mathcal{Y}^{(x)}_{i', 1}$, proving the result for the first column of $Y$.

We turn now to the successive columns.  By Equation~\ref{art}, \[\mathcal{m}^{(x)}_1 = \cdots = \mathcal{m}^{(x)}_{\ell^{\circ}_x} = \iota^{\circ}_x - \sum_{x'=1}^{x-1} \ell^{\circ}_{x'} + \sum_{x' = x+1}^{\mathcal{k}} \ell^{\circ}_{x'}.\]

Set $\mathcal{f}_x := \mathcal{S}(\mathcal{a}^{(x)}, \operatorname{id}, \mathcal{m}^{(x)})$ and $\mathcal{f} := \mathcal{S}(\alpha, \operatorname{id}, \iota)$.  Suppose that $\ell_x > 0$.  Note that \[p_{(x, \mathcal{f}_x^{-1}(1,i'))} = \mathcal{f}^{-1}(x, i')\] for all $1 \leq i' \leq \ell_x$. 

Set \[({\mathcal{a}^{(x)}})' := \left  [\mathcal{a}^{(x)}_{\mathcal{f}_x^{-1}(1,1)} - 1, \ldots, \mathcal{a}^{(x)}_{\mathcal{f}_x^{-1}(1, \ell_x)} - 1 \right ]\] and \[({\hat{\mathcal{n}}^{(x)}})' := \left [\hat{\mathcal{n}}^{(x)}_{\mathcal{f}_x^{-1}(1,1)} - \mathcal{m}^{(x)}_{\mathcal{f}_x^{-1}(1,1)}, \ldots, \hat{\mathcal{n}}^{(x)}_{\mathcal{f}_x^{-1}(1,\ell_x)} - \mathcal{m}^{(x)}_{\mathcal{f}_x^{-1}(1,\ell_x)}\right].\]

Then set \[\left((\mathcal{X}^{(x)})', (\mathcal{Y}^{(x)})'\right) := \mathcal{A}_{1} \left(({\mathcal{a}^{(x)}})', ({\hat{\mathcal{n}}^{(x)}})' \right).\]

Since \[({\mathcal{a}^{(x)}})'_{i'} = \mathcal{a}^{(x)}_{\mathcal{f}_x^{-1}(1,i')} - 1 = \alpha_{\mathcal{f}^{-1}(x,i')} - 1 = \alpha^{(x)}_{i'}\] and
\begin{align*}
& ({\hat{\mathcal{n}}^{(x)}})'_{i'} = \hat{\mathcal{n}}^{(x)}_{\mathcal{f}_x^{-1}(1,i')} - \mathcal{m}^{(x)}_{\mathcal{f}_x^{-1}(1,i')} \\ & = \nu_{\mathcal{f}^{-1}(x,i')} - \sum_{x'=1}^{x-1} \sum_{i_0 = 1}^{\ell^{\circ}_{x'}} \min \left \lbrace \alpha_{\mathcal{f}^{-1}(x,i')}, \alpha_{p_{(x',i_0)}} \right \rbrace + \sum_{x'=x+1}^{\mathcal{k}} \sum_{i_0 = 1}^{\ell^{\circ}_{x'}} \min \left \lbrace \alpha_{\mathcal{f}^{-1}(x,i')}, \alpha_{p_{(x',i_0)}} \right \rbrace \\ & - \iota^{\circ}_x + \sum_{x'=1}^{x-1} \ell^{\circ}_{x'} - \sum_{x'=x+1}^{\mathcal{k}} \ell^{\circ}_{x'} \\ & = \nu^{(x)}_{i'} - \sum_{x'=1}^{x-1} \sum_{i_0 = 1}^{\ell^{\circ}_{x'}} \min \left \lbrace \alpha_{\mathcal{f}^{-1}(x,i')} - 1, \alpha_{p_{(x',i_0)}} - 1 \right \rbrace + \sum_{x'=x+1}^{\mathcal{k}} \sum_{i_0 = 1}^{\ell^{\circ}_{x'}} \min \left \lbrace \alpha_{\mathcal{f}^{-1}(x,i')} - 1, \alpha_{p_{(x',i_0)}} - 1 \right \rbrace \\ & = \nu^{(x)}_{i'} - \sum_{x'=1}^{x-1} \sum_{i_0 =1}^{\ell_x} \min \left \lbrace \alpha^{(x)}_{i'}, \alpha_{\mathcal{f}^{-1}(x', i_0)} - 1 \right \rbrace + \sum_{x'=x+1}^{\mathcal{k}} \sum_{i_0 = 1}^{\ell_x} \min \left \lbrace \alpha^{(x)}_{i'}, \alpha_{\mathcal{f}^{-1}(x', i_0)} - 1\right \rbrace \\ & = \nu^{(x)}_{i'} - \sum_{x'=1}^{x-1} \sum_{i_0 =1}^{\ell_x} \min \left \lbrace \alpha^{(x)}_{i'}, \alpha^{(x')}_{i_0} \right \rbrace + \sum_{x'=x+1}^{\mathcal{k}} \sum_{i_0 = 1}^{\ell_x} \min \left \lbrace \alpha^{(x)}_{i'}, \alpha^{(x')}_{i_0} \right \rbrace = \hat{\nu}^{(x)}_{i'},
\end{align*}
it follows that $((\mathcal{X}^{(x)})', (\mathcal{Y}^{(x)})') = (X^{(x)}, Y^{(x)})$.  

Thus, $Y_{p_{(x, \mathcal{f}_x^{-1}(1,i'))}, j' + 1} = Y_{\mathcal{f}^{-1}(x,i'), j' + 1} = Y^{(x)}_{i',j'} = (\mathcal{Y}^{(x)})'_{i',j'} = \mathcal{Y}^{(x)}_{\mathcal{f}_x^{-1}(1, i'), j' + 1}$ for all $j' \geq 1$.  The result follows.  
\end{proof}

\begin{thm} \label{differences}
Let $(\alpha, \nu, \epsilon) \in \mathbb{N}^{\ell} \times \mathbb{Z}^{\ell} \times \lbrace \pm 1 \rbrace$.  Set $(X,Y) := \mathcal{A}(\alpha, \nu, \epsilon)$.  Then $Y_{i, j+1} - Y_{i,j} \in \lbrace 0, \epsilon (-1)^{j+1} \rbrace$.  
\end{thm}

\begin{proof}
The proof is by induction on $\max \lbrace \alpha_1, \ldots, \alpha_{\ell} \rbrace$.  We show the inductive step for $\epsilon = -1$ only.  Thus, we set $(X,Y) := \mathcal{A}(\alpha, \nu, -1)$ and prove $Y_{i, 2} - Y_{i, 1} \in \lbrace 0, -1 \rbrace$.  The rest follows from the inductive hypothesis.  To see this, set $\sigma := \mathcal{R}_{-1}(\alpha, \nu)$ and $\mu := \mathcal{U}_{-1}(\alpha, \nu, \sigma)$.  Then note that \[Y_{i_{(x,i')}, j'+2} - Y_{i_{(x,i')},j'+1} = Y^{(x)}_{i', j'+1} - Y^{(x)}_{i', j'} \in \lbrace 0, (-1)^{j'+1} \rbrace\] for all $(x,i')$ in the image of $\mathcal{S}(\alpha, \sigma, \mu)$ such that $i' > 0$.

By Lemma~\ref{collapse}, we may assume $\mu_1 = \cdots = \mu_{\ell}$.  Denote the common value by $\mu^{\circ}$.  Furthermore, by Proposition~\ref{permute}, we may assume without loss of generality that $\sigma = \operatorname{id}$.  Hence 
\begin{align*}
\mu^{\circ} & = \mathcal{C}_{-1}(\alpha, \nu, 1, \varnothing, \lbrace 2, \ldots, \ell \rbrace) - \ell + 1 \\ & = \left \lceil \frac{\nu_1 + \sum_{i = 2}^{\ell} \min \lbrace \alpha_1, \alpha_i \rbrace}{\alpha_1} \right \rceil - \ell + 1 \\ & = \left \lceil \frac{\nu_1 + \alpha^*_1 + \cdots + \alpha^*_{\alpha_1}}{\alpha_1} \right \rceil - \ell.
\end{align*}

It follows that \[Y_{i,1} = X_{i,1} + \ell - 2i + 1 = \left \lceil \frac{\nu_1 + \alpha^*_1 + \cdots + \alpha^*_{\alpha_1}}{\alpha_1} \right \rceil - 2i + 1.\]

Set $\mathcal{f} := \mathcal{S}(\alpha, \mu, 1)$.  Set $\ell' := \alpha^*_2$.  For all $1 \leq i' \leq \ell'$, set $i_{i'} := \mathcal{f}^{-1}(1, i')$.  Set \[\alpha' := \left [\alpha_{i_1} - 1, \ldots, \alpha_{i_{\ell'}} - 1 \right] \quad \text{and} \quad \nu' := \left [\nu_{i_1} - \mu^{\circ}, \ldots, \nu_{i_{\ell'}} - \mu^{\circ} \right].\]

Set $\tau := \mathcal{R}_1 (\alpha', \nu')$.  Set $\mu' := \mathcal{U}_1(\alpha', \nu', \tau)$.  Additionally, set $(X', Y') := \mathcal{A}(\alpha', \nu', 1)$.  Note that \[Y_{i_{i'},2} = Y'_{i',1} = X'_{i',1} + \ell' - 2i' + 1 = \mu'_{i'} + \ell' - 2i' + 1.\]

We claim that $Y_{i_{i'}, 2} - Y_{i_{i'}, 1} \in \lbrace 0, -1 \rbrace$ for all $1 \leq i' \leq \ell'$.  The proof is by (backwards) induction on $i'$.  For the inductive step, assume that the claim holds for all $i' + 1 \leq i_0 \leq \ell'$.  Set $I_b := \tau^{-1} \lbrace i' + 1, \ldots, \ell' \rbrace$ and $I'_b := \lbrace 1, \ldots, \ell' \rbrace \setminus I_b$.  To show the claim holds for $i'$, we split into two cases.  

\begin{enumerate}
	\item If $i' = \ell'$, or $i' < \ell'$ and $\alpha_{i_{i'} + 1} = 1$, we show: 
	\begin{enumerate}
		\item For all $\mathcal{i} \in I'_b$, \[\mathcal{C}_1 \left (\alpha', \nu', \mathcal{i}, I'_b \setminus \lbrace \mathcal{i} \rbrace, I_b \right) \geq \left \lceil \frac{\nu_1 + \alpha^*_1 + \cdots + \alpha^*_{\alpha_1}}{\alpha_1} \right \rceil - 2i_{i'}.\]
		\item There exists $\mathcal{i} \in I'_b$ such that  \[\mathcal{C}_1 \left (\alpha', \nu', \mathcal{i}, I'_b \setminus \lbrace \mathcal{i} \rbrace, I_b \right) \leq \left \lceil \frac{\nu_1 + \alpha^*_1 + \cdots + \alpha^*_{\alpha_1}}{\alpha_1} \right \rceil - 2i_{i'} + 1.\]
	\end{enumerate}
	\item If $i' < \ell'$ and $\alpha_{i_{i'} + 1} \geq 2$, we show (b) only.  
\end{enumerate}

We first prove that the properties indicated are sufficient to obtain the desired; then we show that they indeed hold.  

\begin{enumerate}
	\item Suppose $i' = \ell'$, or $i' < \ell'$ and $\alpha_{i_{i'}+1} = 1$, and suppose (a) and (b) hold.  We first claim that 
	\begin{align} \label{echo}
	\mu'_{i'} = \mathcal{C}_1(\alpha', \nu', \tau^{-1}(i'), I'_b \setminus \lbrace \tau^{-1}(i') \rbrace, I_b) -\ell' + 2i' - 1.
	\end{align}
	If $i' = \ell'$, the claim follows immediately.  If $i' < \ell'$ and $\alpha_{i_{i'}+1} = 1$, it suffices to show \[\mathcal{C}_1(\alpha', \nu', \tau^{-1}(i'), I'_b \setminus \lbrace \tau^{-1}(i') \rbrace, I_b) -\ell' + 2i' - 1 \geq \mu'_{i'+1}.\]  Applying (a) and the inductive hypothesis, we find
	\begin{align*}
	\mathcal{C}_1(\alpha', \nu', \tau^{-1}(i'), I'_b \setminus \lbrace \tau^{-1}(i') \rbrace, I_b) & \geq \left \lceil \frac{\nu_1 + \alpha^*_1 + \cdots + \alpha^*_{\alpha_1}}{\alpha_1} \right \rceil - 2i_{i'}  \\ & = \mu^{\circ} + \ell - 2i_{i'} \\ & = X_{i_{i' + 1}, 1} + \ell - 2i_{i'} \\ & = Y_{i_{i'+1},1} + 2i_{i'+1} - 2i_{i'}- 1 \\ & \geq Y_{i_{i'+1}, 1} + 3 \\ & \geq Y_{i_{i'+1}, 2} + 3 \\ & = \mu'_{i'+1} + \ell' - 2i' + 2.
	\end{align*}
	From Equation~\ref{echo}, invoking (a) yields
	\begin{align*}
	Y_{i_{i'}, 2} & = \mu'_{i'} + \ell' - 2i'  + 1 \\ & = \mathcal{C}_1(\alpha', \nu', \tau^{-1}(i'), I'_b \setminus \lbrace \tau^{-1}(i') \rbrace, I_b) \\ & \geq \mu^{\circ} + \ell - 2i_{i'} \\ & = X_{i_{i'}, 1} + \ell - 2i_{i'} \\ & = Y_{i_{i'}, 1} - 1.
	\end{align*}
	Since the minimum value of the function given by \[\mathcal{i} \mapsto \mathcal{C}_1(\alpha', \nu', \mathcal{i}, I'_b \setminus \lbrace \mathcal{i} \rbrace, I_b) \] is attained at $\mathcal{i} = \tau^{-1}(i')$, invoking (b) yields
	\begin{align*}
		Y_{i_{i'}, 2} & = \mathcal{C}_1(\alpha', \nu', \tau^{-1}(i'), I'_b \setminus \lbrace \tau^{-1}(i') \rbrace, I_b) \\ & \leq \mu^{\circ} + \ell - 2{i_{i'}} + 1 \\ & = X_{i_{i'}, 1} + \ell -2i_{i'} + 1 \\ & = Y_{i_{i'}, 1}.
	\end{align*}
	It follows that $Y_{i_{i'}, 2} - Y_{i_{i'}, 1} \in \lbrace 0, -1 \rbrace$.  
	\item Suppose $i' < \ell'$ and $\alpha_{i_{i'} + 1} \geq 2$, and suppose (b) holds.  Note that $i_{i'+1} = i_{i'} + 1$.  Thus,
	\begin{align*}
	Y_{i_{i'}, 2} & = \mu'_{i'} + \ell' - 2i' + 1 \\ & \geq \mu'_{i'+1} + \ell' - 2i' + 1 \\ & = Y_{i_{i'+1}, 2} + 2 \\ & = Y_{i_{i'} +1, 2} + 2 \\ & \geq Y_{i_{i'} + 1, 1} + 1 \\ & = X_{i_{i'} + 1, 1} + \ell - 2i_{i'}\\ & = X_{i_{i'}, 1} + \ell - 2i_{i'} \\ & = Y_{i_{i'}, 1} - 1.   
	\end{align*}
	If Equation~\ref{echo} holds, then $Y_{i_{i'}, 2} \leq Y_{i_{i'}, 1}$ follows from invoking (b) as above.  Otherwise, $\mu'_{i'} = \mu'_{i' + 1}$, and 
	\begin{align*}
	Y_{i_{i'}, 2} & = \mu'_{i'} + \ell' - 2i' + 1 \\ & = \mu'_{i'+1} + \ell' - 2i' + 1 \\ & = Y_{i_{i'+1}, 2} + 2 \\ & = Y_{i_{i'} + 1, 2} + 2 \\ & \leq Y_{i_{i'} + 1, 1} + 2 \\ & = X_{i_{i'} +1, 1} + \ell - 2i_{i'} + 1 \\ & = X_{i_{i'}, 1} + \ell - 2i_{i'} + 1 \\ & = Y_{i_{i'}, 1}.
	\end{align*}
\end{enumerate}

Therefore, it suffices to show (i) that (a) holds if $i' = \ell'$, or $i' < \ell'$ and $\alpha_{i_{i'} + 1} = 1$, and (ii) that (b) holds always.  We prove these claims subject to the following assumption: For all $1 \leq i \leq \ell$ such that $\alpha_i = 1$, the set $I_i := \lbrace i_0 \in \lbrace 1, \ldots, \ell' \rbrace : i_{i_0} > i \rbrace$ is preserved under $\tau$.  Finally, we justify the assumption.  

\begin{enumerate}
	\item[(i)] Suppose $i' = \ell'$, or $i' < \ell'$ and $\alpha_{i_{i'} + 1} = 1$.  Assume for the sake of contradiction that there exists $\mathcal{i} \in I'_b$ such that \[\mathcal{C}_1 \left (\alpha', \nu', \mathcal{i}, I'_b \setminus \lbrace \mathcal{i} \rbrace, I_b \right) < \mu^{\circ} + \ell - 2i_{i'}.\]
	
	For all $1 \leq i, j \leq \ell$, set $m_{i,j} := \min \lbrace \alpha_i, \alpha_j \rbrace$.  
	Note that
	\begin{align*}
	\mathcal{C} \left (\alpha', \nu', \mathcal{i}, I'_b \setminus \lbrace \mathcal{i} \rbrace, I_b \right) & = \nu'_{\mathcal{i}} - \sum_{i_0 \in I'_b \setminus \lbrace \mathcal{i} \rbrace} \min \left \lbrace \alpha'_{\mathcal{i}}, \alpha'_{i_0} \right \rbrace + \sum_{i_0 \in I_b} \min \left \lbrace \alpha'_{\mathcal{i}}, \alpha'_{i_0} \right \rbrace \\ & = \nu_{i_{\mathcal{i}}} - \mu^{\circ} - (\alpha^*_2 + \cdots + \alpha^*_{\alpha_{i_{\mathcal{i}}}}) + \alpha_{i_{\mathcal{i}}} - 1 - 2 (\ell' - i') + 2 \sum_{i_0 \in I_b} m_{i_{\mathcal{i}}, i_{i_0}}.
	\end{align*}
	Hence \[\mathcal{C}_1(\alpha', \nu', \mathcal{i}, I'_b \setminus \lbrace \mathcal{i} \rbrace, I_b) = \left \lfloor \frac{ \nu_{i_{\mathcal{i}}} - \mu^{\circ} - (\alpha^*_2 + \cdots + \alpha^*_{\alpha_{i_{\mathcal{i}}}}) - 2(\ell'-i') + 2 \sum_{i_0 \in I_b} m_{i_{\mathcal{i}}, i_{i_0}}}{\alpha_{i_{\mathcal{i}}} - 1} \right \rfloor + 1.\]
	Thus, \[\left \lfloor \frac{ \nu_{i_{\mathcal{i}}} - \mu^{\circ} - (\alpha^*_2 + \cdots + \alpha^*_{\alpha_{i_{\mathcal{i}}}}) - 2(\ell'-i') + 2 \sum_{i_0 \in I_b} m_{i_{\mathcal{i}}, i_{i_0}}}{\alpha_{i_{\mathcal{i}}} - 1} \right \rfloor < \mu^{\circ} + \ell - 2i_{i'} - 1.\]
	Since the right-hand side is an integer, it follows that
	\begin{align*}
	& \frac{ \nu_{i_{\mathcal{i}}} - \mu^{\circ} - (\alpha^*_2 + \cdots + \alpha^*_{\alpha_{i_{\mathcal{i}}}}) - 2(\ell'-i') + 2 \sum_{i_0 \in I_b} m_{i_{\mathcal{i}}, i_{i_0}}}{\alpha_{i_{\mathcal{i}}} - 1} < \mu^{\circ} + \ell - 2i_{i'} - 1 \\ & \Longleftrightarrow \frac{\nu_{i_{\mathcal{i}}} - (\alpha^*_2 + \cdots + \alpha^*_{\alpha_{i_{\mathcal{i}}}}) - 2(\ell'-i') + 2 \sum_{i_0 \in I_b} m_{i_{\mathcal{i}}, i_{i_0}} - (\ell - 2i_{i'} - 1)(\alpha_{i_{\mathcal{i}}} - 1)}{\alpha_{i_{\mathcal{i}}} - 1} \\ & < \mu^{\circ} \left( 1 + \frac{1}{\alpha_{i_{\mathcal{i}}} - 1} \right) \\ & \Longleftrightarrow \frac{\nu_{i_{\mathcal{i}}} - (\alpha^*_2 + \cdots + \alpha^*_{\alpha_{i_{\mathcal{i}}}}) - 2(\ell'-i') + 2 \sum_{i_0 \in I_b} m_{i_{\mathcal{i}}, i_{i_0}} - (\ell - 2i_{i'} - 1)(\alpha_{i_{\mathcal{i}}} - 1)}{\alpha_{i_{\mathcal{i}}}} < \mu^{\circ} \\ & \Longleftrightarrow \frac{\nu_{i_{\mathcal{i}}} - (\alpha^*_1 + \cdots + \alpha^*_{\alpha_{i_{\mathcal{i}}}}) + 2 \sum_{i_0 \in I_b} m_{i_{\mathcal{i}}, i_{i_0}} + 2\ell - 2 \ell' + 2 i' - 2i_{i'} - 1}{\alpha_{i_{\mathcal{i}}}} - \ell + 2i_{i'} + 1 < \mu^{\circ} \\ & \Longleftrightarrow \frac{\nu_{i_{\mathcal{i}}} + (\alpha^*_1 + \cdots + \alpha^*_{\alpha_{i_{\mathcal{i}}}}) - \alpha_{i_{\mathcal{i}}} - 2 \sum_{i_0 = 1}^{i_{\mathcal{i}} - 1} m_{i_{\mathcal{i}}, i_0} - 1}{\alpha_{i_{\mathcal{i}}}} \\ & + \frac{- 2 \sum_{i_0 = i_{\mathcal{i}} + 1}^{\ell} m_{i_{\mathcal{i}}, i_0} + 2\sum_{i_0 \in I_b} m_{i_{\mathcal{i}}, i_{i_0}} + 2\ell - 2\ell' + 2i' - 2i_{i'}}{\alpha_{i_{\mathcal{i}}}} - \ell + 2i_{i'} < \mu^{\circ} \\ & \Longleftrightarrow \frac{\mathcal{C}(\alpha, \nu, i_{\mathcal{i}}, \lbrace 1, \ldots, i_{\mathcal{i}}-1 \rbrace, \lbrace i_{\mathcal{i}} +1, \ldots, \ell \rbrace) -1}{\alpha_{i_{\mathcal{i}}}} \\ & + \frac{- 2 \sum_{i_0 = i_{\mathcal{i}} + 1}^{\ell} m_{i_{\mathcal{i}}, i_0} + 2\sum_{i_0 \in I_b} m_{i_{\mathcal{i}}, i_{i_0}} + 2\ell - 2\ell' + 2i' - 2i_{i'}}{\alpha_{i_{\mathcal{i}}}} + 2i_{i'} - 2i_{\mathcal{i}} < \mu^{\circ} + \ell - 2i_{\mathcal{i}}.  
	\end{align*}
	
	We observe that $I_b = \lbrace i'+1, \ldots, \ell' \rbrace$.  If $i' = \ell'$, this holds vacuously; otherwise, it follows from the assumption indicated above, for $\alpha_{i_{i'} + 1} =1$ implies $\lbrace i'+1, \ldots, \ell' \rbrace = I_{i_{i'} +1}$ is preserved under $\tau$.  Since $\mathcal{i} \in I'_b$, we see also that $\mathcal{i} \leq i'$.  
	
	Thus, 
	\begin{align*}
	& \frac{- 2 \sum_{i_0 = i_{\mathcal{i}} + 1}^{\ell} m_{i_{\mathcal{i}}, i_0} + 2\sum_{i_0 \in I_b} m_{i_{\mathcal{i}}, i_{i_0}} + 2\ell - 2\ell' + 2i' - 2i_{i'}}{\alpha_{i_{\mathcal{i}}}} + 2i_{i'} - 2i_{\mathcal{i}} \\ & = \frac{-2 \sum_{i_0 = i_{\mathcal{i}} + 1}^{i_{i'}} m_{i_{\mathcal{i}}, i_0} - 2 \sum_{i_0 = i_{i'} + 1}^{\ell} m_{i_{\mathcal{i}}, i_0} + 2 \sum_{i_0 = i' +1}^{\ell'} m_{i_{\mathcal{i}}, i_{i_0}} + 2\ell - 2\ell' + 2i' - 2i_{i'}}{\alpha_{i_{\mathcal{i}}}} + 2i_{i'} - 2i_{\mathcal{i}} \\ & = \frac{-2 \sum_{i_0=i_{\mathcal{i}} + 1}^{i_{i'}} m_{i_{\mathcal{i}}, i_0}}{\alpha_{i_{\mathcal{i}}}} + 2i_{i'} - 2i_{\mathcal{i}}  \geq 0.  
	\end{align*}
	
	Furthermore, \[\mu^{\circ} = \mu_{i_{\mathcal{i}}} \leq \mathcal{C}_{-1}(\alpha, \nu, i_{\mathcal{i}}, \lbrace 1, \ldots, i_{\mathcal{i}} - 1 \rbrace, \lbrace i_{\mathcal{i}} + 1, \ldots, \ell \rbrace) - \ell + 2i_{\mathcal{i}} - 1.\]
	
	Set \[\mathcal{c} := \mathcal{C}(\alpha, \nu, i_{\mathcal{i}}, \lbrace 1, \ldots, i_{\mathcal{i}} - 1 \rbrace, \lbrace i_{\mathcal{i}} + 1, \ldots, \ell \rbrace).\]
	
	Then 
	\begin{align*}
	& \frac{\mathcal{c}-1}{\alpha_{i_{\mathcal{i}}}} < \left \lceil \frac{\mathcal{c}}{\alpha_{i_{\mathcal{i}}}} \right \rceil - 1 \\ & \Longleftrightarrow \mathcal{c} -1 < \alpha_{i_{\mathcal{i}}} \left \lceil \frac{\mathcal{c}}{\alpha_{i_{\mathcal{i}}}} \right \rceil - \alpha_{i_{\mathcal{i}}} \\ & \Longleftrightarrow \mathcal{c} \leq \alpha_{i_{\mathcal{i}}} \left \lceil \frac{\mathcal{c}}{\alpha_{i_{\mathcal{i}}}} \right \rceil - \alpha_{i_{\mathcal{i}}} \\ & \Longrightarrow \mathcal{c} < \alpha_{i_{\mathcal{i}}} \left( \frac{\mathcal{c}}{\alpha_{i_{\mathcal{i}}}} + 1 \right) - \alpha_{i_{\mathcal{i}}} = \mathcal{c},
	\end{align*}
	which is a contradiction.  
	
	\item[(ii)] If $\alpha_j > 1$ for all $1 \leq j < i_{\mathcal{i'}}$, set $j_0 := 0$.  Otherwise, let $j_0 < i_{i'}$ be maximal such that $\alpha_{j_0} = 1$.  Analogously, if $\alpha_j > 1$ for all $i_{\mathcal{i'}} < j \leq \ell$, set $j_1 := \ell + 1$.  Otherwise, let $j_1 > i_{i'}$ be minimal such that $\alpha_{j_1} = 1$.  Set $I_c := I_{j_0} \setminus I_{j_1}$.  By assumption, $I_c$ is preserved under $\tau$.  Hence $I_c \cap I'_b \neq \varnothing$, else $I_c \subset I_b$, meaning $I_c \subset \lbrace i' + 1, \ldots, \ell' \rbrace$, which is impossible because $i' \in I_c$.  
	
	Let $\mathcal{i} \in I_c \cap I'_b$ be chosen so that $\alpha_{i_{\mathcal{i}}}$ is minimal.  We claim that \[\mathcal{C}_1 \left (\alpha', \nu', \mathcal{i}, I'_b \setminus \lbrace \mathcal{i} \rbrace, I_b \right) \leq \mu^{\circ} + \ell - 2i_{i'} + 1.\]
	
	Assume for the sake of contradiction that \[\mathcal{C}_1 \left (\alpha', \nu', \mathcal{i}, I'_b \setminus \lbrace \mathcal{i} \rbrace, I_b \right) \geq \mu^{\circ} + \ell - 2i_{i'} + 2.\] 
	
	As above, \[\mathcal{C}_1(\alpha', \nu', \mathcal{i}, I'_b \setminus \lbrace \mathcal{i} \rbrace, I_b) = \left \lfloor \frac{ \nu_{i_{\mathcal{i}}} - \mu^{\circ} - (\alpha^*_2 + \cdots + \alpha^*_{\alpha_{i_{\mathcal{i}}}}) - 2(\ell'-i') + 2 \sum_{i_0 \in I_b} m_{i_{\mathcal{i}}, i_{i_0}}}{\alpha_{i_{\mathcal{i}}} - 1} \right \rfloor + 1.\]
	
	Thus,  
	\begin{align*}
	& \frac{ \nu_{i_{\mathcal{i}}} - \mu^{\circ} - (\alpha^*_2 + \cdots + \alpha^*_{\alpha_{i_{\mathcal{i}}}}) - 2(\ell'-i') + 2 \sum_{i_0 \in I_b} m_{i_{\mathcal{i}}, i_{i_0}}}{\alpha_{i_{\mathcal{i}}} - 1} \geq \mu^{\circ} + \ell - 2i_{\mathcal{i'}} + 1  \\ & \Longleftrightarrow \frac{\nu_{i_{\mathcal{i}}} - (\alpha^*_2 + \cdots + \alpha^*_{\alpha_{i_{\mathcal{i}}}}) - 2(\ell'-i') + 2 \sum_{i_0 \in I_b} m_{i_{\mathcal{i}}, i_{i_0}} - (\ell - 2i_{i'} + 1)(\alpha_{i_{\mathcal{i}}} - 1)}{\alpha_{i_{\mathcal{i}}} - 1} \\ & \geq \mu^{\circ} \left( 1 + \frac{1}{\alpha_{i_{\mathcal{i}}} - 1} \right) \\ & \Longleftrightarrow \frac{\nu_{i_{\mathcal{i}}} - (\alpha^*_2 + \cdots + \alpha^*_{\alpha_{i_{\mathcal{i}}}}) - 2(\ell'-i') + 2 \sum_{i_0 \in I_b} m_{i_{\mathcal{i}}, i_{i_0}} - (\ell - 2i_{i'} + 1)(\alpha_{i_{\mathcal{i}}} - 1)}{\alpha_{i_{\mathcal{i}}}} \geq \mu^{\circ} \\ & \Longleftrightarrow \frac{\nu_{i_{\mathcal{i}}} - (\alpha^*_1 + \cdots + \alpha^*_{\alpha_{i_{\mathcal{i}}}}) + 2 \sum_{i_0 \in I_b} m_{i_{\mathcal{i}}, i_{i_0}} + 2 \ell - 2 \ell' + 2 i' - 2i_{i'}+1}{\alpha_{i_{\mathcal{i}}}} - \ell + 2i_{i'} - 1 \geq \mu^{\circ} \\ & \Longleftrightarrow \frac{\nu_{i_{\mathcal{i}}} + (\alpha^*_1 + \cdots + \alpha^*_{\alpha_{i_{\mathcal{i}}}}) - \alpha_{i_{\mathcal{i}}} - 2 \sum_{i_0 = 1}^{j_0} m_{i_{\mathcal{i}}, i_0} + 1}{\alpha_{i_{\mathcal{i}}}} \\ & + \frac{- 2 \sum_{i_0 = j_0+1}^{\ell} m_{i_{\mathcal{i}}, i_0} + 2\sum_{i_0 \in I_b} m_{i_{\mathcal{i}}, i_{i_0}} + 2\ell - 2\ell' + 2i' - 2i_{i'}}{\alpha_{i_{\mathcal{i}}}} - \ell + 2i_{i'} \geq \mu^{\circ} \\ & \Longleftrightarrow \frac{\nu_{i_{\mathcal{i}}} + (\alpha^*_1 + \cdots + \alpha^*_{\alpha_{i_{\mathcal{i}}}}) - \alpha_{i_{\mathcal{i}}} - 2 \sum_{i_0 = 1}^{j_0} m_{i_{\mathcal{i}}, i_0} + 1}{\alpha_{i_{\mathcal{i}}}} \\ & + \frac{- 2 \sum_{i_0 = j_0 + 1}^{\ell} m_{i_{\mathcal{i}}, i_0} + 2\sum_{i_0 \in I_b} m_{i_{\mathcal{i}}, i_{i_0}} + 2\ell - 2\ell' + 2i' - 2i_{i'}}{\alpha_{i_{\mathcal{i}}}} + 2i_{i'} - 2j_0 \geq \mu^{\circ} + \ell - 2 j_0.
	\end{align*}
	
	From the inclusions $I_{j_1} \subset \lbrace i'+1, \ldots, \ell' \rbrace \subset I_{j_0}$, we see that $I_{j_1} \subset I_b \subset I_{j_0}$, so $I_{j_1} = I_b \setminus I_c$.  Furthermore, $|I_c \cap I'_b| = |I_c \cap \lbrace 1, \ldots, i' \rbrace| = i_{i'} - j_0$.  Thus,
	\begin{align*}
	& \frac{- 2 \sum_{i_0 = j_0 + 1}^{\ell} m_{i_{\mathcal{i}}, i_0} + 2\sum_{i_0 \in I_b} m_{i_{\mathcal{i}}, i_{i_0}} + 2\ell - 2\ell' + 2i' - 2i_{i'}}{\alpha_{i_{\mathcal{i}}}} + 2i_{i'} - 2j_0 \\ & =  \frac{- 2 \sum_{i_0 = j_0 + 1}^{j_1 - 1} m_{i_{\mathcal{i}}, i_0}  + 2\sum_{i_0 \in I_c \cap I_b} m_{i_{\mathcal{i}}, i_{i_0}} }{\alpha_{i_{\mathcal{i}}}} + 2i_{i'} - 2j_0 \\ & + \frac{- 2 \sum_{i_0 = j_1}^{\ell} m_{i_{\mathcal{i}}, i_0} + 2 \sum_{i_0 \in I_{j_1}} m_{i_{\mathcal{i}}, i_{i_0}} + 2\ell - 2\ell' + 2i' - 2i_{i'}}{\alpha_{i_{\mathcal{i}}}} \\ & = \frac{- 2\sum_{i_0 \in I_c \cap I'_b} m_{i_{\mathcal{i}}, i_{i_0}}}{\alpha_{i_{\mathcal{i}}}} + 2i_{i'} - 2j_0 \\ & + \frac{-2 \sum_{i_0 = i_{i'}}^{\ell} m_{i_{\mathcal{i}}, i_0} + 2 \sum_{i_0 = i'}^{\ell'} m_{i_{\mathcal{i}}, i_{i_0}} + 2\ell - 2\ell' + 2i' - 2i_{i'}}{\alpha_{i_{\mathcal{i}}}} \\ & = 0 + 0 = 0.
	\end{align*}
	
	Hence \[ \frac{\nu_{i_{\mathcal{i}}} + (\alpha^*_1 + \cdots + \alpha^*_{\alpha_{i_{\mathcal{i}}}}) - \alpha_{i_{\mathcal{i}}} - 2 \sum_{i_0 = 1}^{j_0} m_{i_{\mathcal{i}}, i_0} + 1}{\alpha_{i_{\mathcal{i}}}} \geq \mu^{\circ} + \ell - 2j_0.\]
	
	If $j_0 = 0$, then 
	\begin{align*}
	\frac{\mathcal{C}(\alpha, \nu, i_{\mathcal{i}}, \varnothing, \lbrace 1, \ldots, \ell \rbrace \setminus \lbrace i_{\mathcal{i}} \rbrace) + 1}{\alpha_{i_{\mathcal{i}}}} & = \frac{\nu_{i_{\mathcal{i}}} + (\alpha^*_1 + \cdots + \alpha^*_{\alpha_{i_{\mathcal{i}}}}) - \alpha_{i_{\mathcal{i}}} + 1}{\alpha_{i_{\mathcal{i}}}} \\ & \geq \mu^{\circ} + \ell \\ & = \mathcal{C}_{-1}(\alpha, \nu, 1, \varnothing, \lbrace 2, \ldots, \ell \rbrace) + 1 \\ & \geq \mathcal{C}_{-1}(\alpha, \nu, i_{\mathcal{i}}, \varnothing, \lbrace 1, \ldots, \ell \rbrace \setminus \lbrace i_{\mathcal{i}} \rbrace) + 1 \\ & \geq \frac{\mathcal{C}(\alpha, \nu, i_{\mathcal{i}}, \varnothing, \lbrace 1, \ldots, \ell \rbrace \setminus \lbrace i_{\mathcal{i}} \rbrace)}{\alpha_{i_{\mathcal{i}}}} + 1,
	\end{align*}
	which is impossible because $\alpha_{i_{\mathcal{i}}} > 1$.  
	
	Thus, $j_0 \geq 1$.  From Proposition~\ref{multi}, it follows that $\nu_{j_0} = \mu^{\circ}$.  Hence \[\mathcal{C}_{-1}(\alpha, \nu, j_0, \lbrace 1, \ldots, j_0 - 1 \rbrace, \lbrace j_0 + 1, \ldots, \ell \rbrace) = \mu^{\circ} + \ell - 2j_0 + 1.\]
		
	Since $j_0 < i_{\mathcal{i}}$ and $\alpha_{j_0} < \alpha_{i_{\mathcal{i}}}$, it follows that \[\mathcal{C}_{-1}(\alpha, \nu, j_0, \lbrace 1, \ldots, j_0 - 1 \rbrace, \lbrace j_0 + 1, \ldots, \ell \rbrace) > \mathcal{C}_{-1}(\alpha, \nu, i_{\mathcal{i}}, \lbrace 1, \ldots, j_0 - 1 \rbrace, \lbrace j_0, \ldots, \ell \rbrace \setminus \lbrace i_{\mathcal{i}} \rbrace).\]
	
	Therefore, \[\mu^{\circ} + \ell - 2j_0 \geq \mathcal{C}_{-1}(\alpha, \nu, i_{\mathcal{i}}, \lbrace 1, \ldots, j_0 - 1 \rbrace, \lbrace j_0, \ldots, \ell \rbrace \setminus \lbrace i_{\mathcal{i}} \rbrace).\]
	
	Then
	\begin{align*}
	& \frac{\mathcal{C}(\alpha, \nu, i_{\mathcal{i}}, \lbrace 1, \ldots, j_0 - 1 \rbrace, \lbrace j_0, \ldots, \ell \rbrace \setminus \lbrace i_{\mathcal{i}} \rbrace) - 1}{\alpha_{i_{\mathcal{i}}}} \\ & = \frac{\nu_{i_{\mathcal{i}}} + (\alpha^*_1 + \cdots + \alpha^*_{\alpha_{i_{\mathcal{i}}}}) - \alpha_{i_{\mathcal{i}}} - 2 \sum_{i_0 = 1}^{j_0} m_{i_{\mathcal{i}}, i_0} + 1}{\alpha_{i_{\mathcal{i}}}} \\ & \geq \mu^{\circ} + \ell - 2j_0 \\ & \geq \mathcal{C}_{-1}(\alpha, \nu, i_{\mathcal{i}}, \lbrace 1, \ldots, j_0 - 1 \rbrace, \lbrace j_0, \ldots, \ell \rbrace \setminus \lbrace i_{\mathcal{i}} \rbrace) \\ & \geq \frac{\mathcal{C}(\alpha, \nu, i_{\mathcal{i}}, \lbrace 1, \ldots, j_0 - 1 \rbrace, \lbrace j_0, \ldots, \ell \rbrace \setminus \lbrace i_{\mathcal{i}} \rbrace)}{\alpha_{i_{\mathcal{i}}}},
	\end{align*}
	which is a contradiction.  
	\end{enumerate}
	
	It remains to justify the assumption that $I_i$ is preserved under $\tau$ for all $1 \leq i \leq \ell$ such that $\alpha_i = 1$.  Given a subset $J \subset \lbrace 1, \ldots, \ell' \rbrace$, set $i_J := \lbrace i_j : j \in J \rbrace$.  Given a subset $I \subset \lbrace 1, \ldots, \ell \rbrace$, let $m(I)$ be its minimal element, and let $M(I)$ be its maximal element.  Say that $I$ is \textit{consecutive} if $M(I) - m(I) + 1 = |I|$.  Partition $\lbrace 1, \ldots, \ell' \rbrace$ into disjoint blocks $J_1, \ldots, J_k$ such that $i_{J_r}$ is consecutive for all $1 \leq r \leq k$ and $m(i_{J_{r+1}}) - M(i_{J_r}) > 1$ for all $1 \leq r \leq k-1$.  
	
	We claim that $J_r$ is preserved under $\tau$ for all $1 \leq r \leq k$.  The proof is by (backwards) induction on $r$.  For the inductive step, suppose the claim holds for all $r + 1 \leq r_0 \leq k$.  Let $c$ be the cardinality of $J_r$, and let $j^r_1, \ldots, j^r_c$ be the elements of $J_r$, arranged in increasing order.  
	
	The claim for $r$ is then that $\tau^{-1}(j^r_{w}) \in J_r$ for all $1 \leq w \leq c$.  If $r=1$, this follows immediately from the inductive hypothesis, so we may assume $r \geq 1$.  Set $q := i_{j^r_1} - 1$.  Then $\alpha_q = 1$ and $i_{j^r_w} = q + w$.  We prove the claim by (backwards) induction on $w$.  
	
	Suppose $\tau^{-1}(j^r_{w_0}) \in J_r$ for all $w + 1 \leq w_0 \leq c$.  Set $j_0 := \tau^{-1}(j^r_{w})$.  Assume for the sake of contradiction that $j_0 \notin J_r$.  By the inductive hypothesis (on $r$), we see that $j_0 \in J_1 \cup \cdots \cup J_{r-1}$.  Thus, $i_{j_0} < q$.  
	
	Note that 
	\begin{align*}
	\mu^{\circ} = \mu_{i_{j_0}} & \leq \mathcal{C}_{-1}(\alpha, \nu, i_{j_0}, \lbrace 1, \ldots, i_{j_0} - 1 \rbrace, \lbrace i_{j_0} + 1, \ldots, \ell \rbrace) - \ell + 2i_{j_0} - 1 \\ & = \left \lceil \frac{\nu_{i_{j_0}} - \sum_{i_0=1}^{i_{j_0}-1} m_{i_{j_0}, i_0} + \sum_{i_0 = i_{j_0} + 1}^{q-1} m_{i_{j_0}, i_0} + 1 + \sum_{i_0 = q+1}^{\ell} m_{i_{j_0}, i_0}}{\alpha_{i_{j_0}}} \right \rceil - \ell + 2i_{j_0} - 1 \\ & \leq \left \lceil \frac{\nu_{i_{j_0}} - \sum_{i_0=1}^{i_{j_0}-1} m_{i_{j_0}, i_0} - \sum_{i_0 = i_{j_0} + 1}^{q-1} m_{i_{j_0}, i_0} + 1 + \sum_{i_0 = q+1}^{\ell} m_{i_{j_0}, i_0}}{\alpha_{i_{j_0}}}\right \rceil - \ell + 2q - 3 \\ & = \left \lceil \frac{\nu_{i_{j_0}} - \sum_{i_0=1}^{q-1} m_{i_{j_0}, i_0} + 1 + \sum_{i_0 = q+1}^{\ell} m_{i_{j_0}, i_0}}{\alpha_{i_{j_0}}}\right \rceil - \ell + 2q - 2.
	\end{align*}
	
	Thus, 
	\begin{align}
	& \left \lfloor \frac{\nu_{i_{j_0}} - \mu^{\circ} - \sum_{i_0 = 1}^{q-1} m_{i_{j_0}, i_0} + \sum_{i_0 = q+1}^{\ell} m_{i_{j_0}, i_0} - \ell + 2q - 1}{\alpha_{i_{j_0}} - 1}\right \rfloor \\ & \geq \left \lfloor \frac{\nu_{i_{j_0}} - \sum_{i_0 = 1}^{q-1} m_{i_{j_0}, i_0} + 1 + \sum_{i_0 = q+1}^{\ell} m_{i_{j_0}, i_0} - \left \lceil \frac{\nu_{i_{j_0}} - \sum_{i_0=1}^{q-1} m_{i_{j_0}, i_0} + 1 + \sum_{i_0 = q+1}^{\ell} m_{i_{j_0}, i_0}}{\alpha_{i_{j_0}}}\right \rceil}{\alpha_{i_{j_0}} -1} \right \rfloor \\ & = \left \lfloor \frac{\left \lfloor \left(\nu_{i_{j_0}} - \sum_{i_0=1}^{q-1} m_{i_{j_0}, i_0} + 1 + \sum_{i_0 = q+1}^{\ell} m_{i_{j_0}, i_0} \right) \left( \frac{\alpha_{i_{j_0}} - 1}{\alpha_{i_{j_0}}} \right)\right \rfloor}{\alpha_{i_{j_0}} - 1}\right \rfloor \\ & = \left \lfloor \frac{\nu_{i_{j_0}} - \sum_{i_0=1}^{q-1} m_{i_{j_0}, i_0} + 1 + \sum_{i_0 = q+1}^{\ell} m_{i_{j_0}, i_0}}{\alpha_{i_{j_0}}}\right \rfloor \\ & \geq \left \lceil \frac{\nu_{i_{j_0}} - \sum_{i_0=1}^{q-1} m_{i_{j_0}, i_0} + 1 + \sum_{i_0 = q+1}^{\ell} m_{i_{j_0}, i_0}}{\alpha_{i_{j_0}}}\right \rceil - 1 \\ & \geq \mu^{\circ} + \ell - 2q + 1.
	\end{align}
	
	Set \[J_r^w := \left \lbrace w_0 \in \lbrace 1, \ldots, c \rbrace : j^r_{w_0} \in J_r \setminus \tau^{-1} \lbrace j^r_{w+1}, \ldots, j^r_{c} \rbrace \right \rbrace.\]  Let $w' \in J_r^w$ be chosen so that $\alpha_{i_{j^r_{w'}}}$ is minimal.
	
	Since $q < q + w'$ and $\alpha_q = 1 < \alpha_{q+w'}$, it follows that
	\begin{align*}
	&\mathcal{C}_{-1}(\alpha, \nu, q + w', \lbrace 1, \ldots, q-1 \rbrace, \lbrace q, \ldots, \ell \rbrace \setminus \lbrace q+w' \rbrace) \\ & < \mathcal{C}_{-1}(\alpha, \nu, q, \lbrace 1, \ldots, q-1 \rbrace, \lbrace q+1, \ldots, \ell \rbrace) \\ & = \nu_q + \ell - 2q + 1 \\ & = \mu^{\circ} + \ell - 2q + 1,
	\end{align*}
	where the last equality follows from Proposition~\ref{multi}.  
	
	Hence
	\begin{align*}
	\mu^{\circ} + \ell - 2q + 1 & \geq \mathcal{C}_{-1}(\alpha, \nu, q + w', \lbrace 1, \ldots, q-1 \rbrace, \lbrace q, \ldots, \ell \rbrace \setminus \lbrace q+w' \rbrace) + 1 \\ & = \left \lceil \frac{\nu_{q+w'} - \sum_{i_0=1}^{q-1} m_{q+w', i_0} + 1 + \sum_{i_0 = q+1}^{\ell} m_{q+w', i_0}}{\alpha_{q + w'}} \right \rceil.  
	\end{align*}
	
	Therefore, 
	\begin{align}
	& \left \lfloor \frac{\nu_{q+w'} - \mu^{\circ} - \sum_{i_0 = 1}^{q-1} m_{q+w', i_0} + \sum_{i_0 = q+1}^{\ell} m_{q+w', i_0} - \ell + 2q - 1}{\alpha_{q+w'}-1} \right \rfloor \\ & \leq \left \lfloor \frac{\nu_{q+w'} - \sum_{i_0 = 1}^{q-1} m_{q+w', i_0} + \sum_{i_0 = q+1}^{\ell} m_{q+w', i_0} - \left \lceil \frac{\nu_{q+w'} - \sum_{i_0=1}^{q-1} m_{q+w', i_0} + 1 + \sum_{i_0 = q+1}^{\ell} m_{q+w', i_0}}{\alpha_{q + w'}} \right \rceil}{\alpha_{q+w'}-1} \right \rfloor \\ & \leq \left \lfloor \frac{\nu_{q+w'} - \sum_{i_0 = 1}^{q-1} m_{q+w', i_0} + 1 + \sum_{i_0 = q+1}^{\ell} m_{q+w', i_0} - \left \lceil \frac{\nu_{q+w'} - \sum_{i_0=1}^{q-1} m_{q+w', i_0} + 1 + \sum_{i_0 = q+1}^{\ell} m_{q+w', i_0}}{\alpha_{q + w'}} \right \rceil}{\alpha_{q+w'}-1} \right \rfloor \\ & = \left \lfloor \frac{\left \lfloor \left(\nu_{q+w'} - \sum_{i_0=1}^{q-1} m_{q+w', i_0} + 1 + \sum_{i_0 = q+1}^{\ell} m_{q+w', i_0} \right) \left(\frac{\alpha_{q+w'}-1}{\alpha_{q+w'}} \right) \right \rfloor}{\alpha_{q+w'}-1} \right \rfloor \\ & = \left \lfloor \frac{\nu_{q+w'} - \sum_{i_0=1}^{q-1} m_{q+w', i_0} + 1 + \sum_{i_0 = q+1}^{\ell} m_{q+w', i_0}}{\alpha_{q + w'}} \right \rfloor \\ & \leq \left \lceil \frac{\nu_{q+w'} - \sum_{i_0=1}^{q-1} m_{q+w', i_0} + 1 + \sum_{i_0 = q+1}^{\ell} m_{q+w', i_0}}{\alpha_{q + w'}} \right \rceil \\ & \leq \mu^{\circ} + \ell - 2q + 1.
	\end{align}
	
	Combining (5.4) -- (5.9) and (5.10) -- (5.16), we see that $(5.4) \geq (5.10)$, with equality if and only if all the inequalities are in fact equalities.  However, if $(5.14) = (5.15)$, then \[z:= \frac{\nu_{q+w'} - \sum_{i_0=1}^{q-1} m_{q+w', i_0} + 1 + \sum_{i_0 = q+1}^{\ell} m_{q+w', i_0}}{\alpha_{q + w'}} \in \mathbb{Z},\] in which case \[\frac{\nu_{q+w'} - \sum_{i_0=1}^{q-1} m_{q+w', i_0} + 1 + \sum_{i_0 = q+1}^{\ell} m_{q+w', i_0}-z}{\alpha_{q + w'} -1} = z \in \mathbb{Z},\] so $(5.12) = z$ and $(5.11) = \left \lfloor z - \frac{1}{\alpha_{q+w'} - 1} \right \rfloor = z-1$.  It follows that $(5.4) > (5.10)$. 
	
	For all $1 \leq i, j \leq \ell$, set $m'_{i,j} := m_{i,j} - 1$.  Note that $m'_{i,j} = 0$ unless $\alpha_i, \alpha_j > 1$.  Additionally, \[\frac{-2\sum_{w_0 \in J_r^w} m'_{i_{j_0}, q+w_0}}{\alpha_{i_{j_0}} - 1} \geq -2w = \frac{-2\sum_{w_0 \in J_r^w} m'_{q+w', q+w_0}}{\alpha_{q+w'} - 1}.\]  
	
	Thus,
	\begin{align*}
	& \mathcal{C}_1(\alpha', \nu', j_0, (J_1 \cup \cdots \cup J_r) \setminus \tau^{-1} \lbrace j^r_{w}, \ldots, j^r_c \rbrace, \tau^{-1} \lbrace j^r_{w+1}, \ldots, j^r_c \rbrace \cup J_{r+1} \cup \cdots \cup J_k) \\ & = \left \lfloor \frac{\nu'_{j_0} - \sum_{i_0 = 1}^{q-1} m'_{i_{j_0}, i_0} - \sum_{w_0 \in J_r^w} m'_{i_{j_0}, q + w_0} + \sum_{w_0 \in \lbrace 1, \ldots, c \rbrace \setminus J_r^w} m'_{i_{j_0}, q + w_0} + \sum_{i_0 = q + c + 1}^{\ell} m'_{i_{j_0}, i_0}}{\alpha_{i_{j_0}} -1} \right \rfloor + 1 \\ & = \left \lfloor \frac{\nu_{i_{j_0}} - \mu^{\circ} - \sum_{i_0 = 1}^{q-1} m'_{i_{j_0}, i_0} + \sum_{i_0 = q + 1}^{\ell} m'_{i_{j_0}, i_0}}{\alpha_{i_{j_0}} -1} + \frac{-2\sum_{w_0 \in J_r^w} m'_{i_{j_0}, q + w_0}}{\alpha_{i_{j_0}} - 1} \right \rfloor + 1 \\ & \geq \left \lfloor \frac{\nu_{i_{j_0}} - \mu^{\circ} - \sum_{i_0 = 1}^{q-1} m_{i_{j_0}, i_0} + \sum_{i_0 = q + 1}^{\ell} m_{i_{j_0}, i_0} - \ell + 2q - 1}{\alpha_{i_{j_0}} -1} \right \rfloor - 2w + 1 \\ & > \left \lfloor \frac{\nu_{q+w'} - \mu^{\circ} - \sum_{i_0 = 1}^{q-1} m_{q+w', i_0} + \sum_{i_0 = q+1}^{\ell} m_{q+w', i_0} - \ell + 2q - 1}{\alpha_{q+w'}-1} \right \rfloor - 2w + 1 \\ & = \left \lfloor \frac{\nu_{q+w'} - \mu^{\circ} - \sum_{i_0 = 1}^{q-1} m'_{q+w', i_0} + \sum_{i_0 = q+1}^{\ell} m'_{q+w', i_0}}{\alpha_{q+w'}-1} + \frac{-2 \sum_{w_0 \in J_r^w} m'_{q+w', q+ w_0}}{\alpha_{q+w'}-1} \right \rfloor + 1 \\ & = \left \lfloor \frac{\nu'_{j^r_{w'}} - \sum_{i_0=1}^{q-1} m'_{i_{j^r_{w'}}, i_0} - \sum_{w_0 \in J_r^w} m'_{i_{j^r_{w'}}, q + w_0} + \sum_{w_0 \in \lbrace 1, \ldots, c \rbrace \setminus J_r^w} m'_{i_{j^r_{w'}}, q + w_0} + \sum_{i_0 = q+c+1}^{\ell} m'_{i_{j^r_{w'}}, i_0}}{\alpha_{q+w'}-1} \right \rfloor + 1 \\ & = \mathcal{C}_1(\alpha', \nu', j^r_{w'}, (J_1 \cup \cdots \cup J_r) \setminus (\lbrace j^r_{w'} \rbrace \cup \tau^{-1} \lbrace j^r_{w+1}, \ldots, j^r_{c} \rbrace), \tau^{-1} \lbrace j^r_{w+1}, \ldots, j^r_c \rbrace \cup J_{r+1} \cup \cdots \cup J_k).  
	\end{align*}
	
	Set $J := \tau^{-1} \lbrace j^r_{w+1}, \ldots, j^r_c \rbrace \cup J_{r+1} \cup \cdots \cup J_k$ and $J' := \lbrace 1, \ldots, \ell' \rbrace \setminus J$.  By the inductive hypothesis, \[J = \tau^{-1} \lbrace j \in \lbrace 1, \ldots, \ell' \rbrace : j > j^r_w \rbrace.\]  
	
	From our work above, we see that \[\mathcal{C}_1(\alpha', \nu', j_0, J' \setminus \lbrace j_0 \rbrace, J) > \mathcal{C}_1(\alpha', \nu', j^r_{w'}, J' \setminus \lbrace j^r_{w'} \rbrace, J),\] which means that the function given by \[j \mapsto \mathcal{C}_1(\alpha', \nu', j, J' \setminus \lbrace j \rbrace, J)\] does not attain its minimal value over the domain $j \in J'$ at $j = j_0 = \tau^{-1}(j^r_w)$.  This contradicts the definition of $\tau$.  
\end{proof}

\begin{df}
Given a diagram $X$ and a positive integer $j$, the diagram $\mathcal{T}_j(X)$ is obtained from $X$ by removing the leftmost $j-1$ columns of $X$, and then removing the empty rows from the remaining diagram.  
\end{df}

\begin{rem}
We refer to $\mathcal{T}_j$ as the \textit{column-reduction} function.  Inductively, we see that $\mathcal{T}_j \mathcal{T}_{j'}(X) = \mathcal{T}_{j+j'-1}(X)$ for all $j, j' \in \mathbb{N}$.  
\end{rem}

\begin{lem} \label{bigentry}
Let $(\alpha, \nu) \in \mathbb{N}^{\ell} \times \mathbb{Z}^{\ell}$.  Set $(X,Y) := \mathcal{A}(\alpha, \nu, -1)$.  Suppose $X_{1,1} = \cdots = X_{\ell, 1}$.  Then $Y_{1,1} \geq Y_{i,j}$ for all $(i, j) \in \mathbb{N} \times \mathbb{N}$ such that $Y$ has an entry in the $i^{\text{th}}$ row and $j^{\text{th}}$ column.  
\end{lem}

\begin{proof}
The proof is by induction on $M := \max \lbrace \alpha_1, \ldots, \alpha_{\ell} \rbrace$.  Clearly, $Y_{1, 1} \geq Y_{i, 1}$ for all $i$, and it follows from Theorem~\ref{differences} that $Y_{i, 1} \geq Y_{i, 2}$ for all $i$, so $Y_{1,1} \geq Y_{i,2}$.  Thus, we may assume $M \geq 3$.  

Maintain the notation from the proof of Theorem~\ref{differences} (continue to assume without loss of generality that $\sigma = \operatorname{id}$).  Set $\mathcal{f}' := \mathcal{S}(\alpha', \tau, \mu')$.  For all $(x, i_0)$ in the image of $\mathcal{f}'$ such that $i_0 > 0$, set $i'_{(x,i_0)} := {\mathcal{f}'}^{-1}(x,i_0)$.  Also set $\mathcal{k}' := |\lbrace \mu'_1, \ldots, \mu'_{\ell'} \rbrace|$.  For all $1 \leq x \leq \mathcal{k}'$, set \[\ell'_x := \max \lbrace i_0 : (x, i_0) \in \mathcal{f}' \lbrace 1, \ldots, \ell' \rbrace \rbrace.\]

If $\ell'_x > 0$, then set \[{\alpha'}^{(x)} := \left [\alpha'_{{\tau}^{-1}\left(i'_{(x, 1)}\right)} - 1, \ldots, \alpha'_{{\tau}^{-1}\left(i'_{(x, \ell'_x)}\right)} - 1 \right] \] and \[{\nu'}^{(x)} = \left [\nu'_{{\tau}^{-1}\left(i'_{(x, 1)}\right)} - \mu'_{i'_{(x, 1)}}, \ldots, \nu'_{{\tau}^{-1}\left(i'_{(x, \ell'_x)}\right)} - \mu'_{i'_{(x, \ell'_x)}} \right].\]  

For all $1 \leq i_0 \leq \ell'_x$, set \[\widehat{\nu'}^{(x)}_{i_0} := {\nu'}^{(x)}_{i_0} - \sum_{x' = 1}^{x-1} \sum_{i_1 = 1}^{\ell'_{x'}} \min \left\lbrace {\alpha'}^{(x)}_{i_0}, {\alpha'}^{(x')}_{i_1} \right\rbrace + \sum_{x' = x+1}^{\mathcal{k}'} \sum_{i_1 = 1}^{\ell'_{x'}} \min \left\lbrace {\alpha'}^{(x)}_{i_0}, {\alpha'}^{(x')}_{i_1} \right\rbrace.\]
Then set $\widehat{\nu'}^{(x)} := \left[\widehat{\nu'}^{(x)}_1, \ldots, \widehat{\nu'}^{(x)}_{\ell'_x}\right]$ and $\left({X'}^{(x)}, {Y'}^{(x)} \right) := \mathcal{A}\left({\alpha'}^{(x)}, \widehat{\nu'}^{(x)}, -1 \right)$.

By construction $\mathcal{T}_2(Y) = Y'$ and \[\mathcal{T}_3(Y) = \mathcal{T}_2 (Y') = \operatorname{Cat}\left({Y'}^{(1)}, \ldots, {Y'}^{(\mathcal{k}')}\right).\]

Note that $\mathcal{T}_2(Y)_{i, 1} \geq \mathcal{T}_2(Y)_{i+1, 1} + 1$.  To see this, suppose the $i^{\text{th}}$ row of $\mathcal{T}_2(Y)$ is contained within the $j^{\text{th}}$ row of $Y$, and the $(i+1)^{\text{th}}$ row of $\mathcal{T}_2(Y)$ is contained within the $k^{\text{th}}$ row of $Y$.  Then $k - j \geq 1$, so $Y_{j, 1} \geq Y_{k, 1} + 2$, and it follows from Theorem~\ref{differences} that $Y_{j, 2} \geq Y_{k, 2} + 1$.  

Applying exactly the same reasoning, we find $\mathcal{T}_2(Y')_{i, 1} \geq \mathcal{T}_2(Y')_{i+1, 1}$, so the entries of $\mathcal{T}_3(Y) = \mathcal{T}_2 (Y')$ down the first column are weakly decreasing.  In particular, ${Y'}^{(1)}_{1,1} \geq {Y'}^{(x)}_{1, 1}$ for all $1 \leq x \leq \mathcal{k}'$.  By the inductive hypothesis, we see that ${Y'}^{(1)}_{1,1} \geq \mathcal{T}_3(Y)_{i,j}$ for all $(i, j)$.  

Thus, it suffices to show $Y_{1,1} \geq {Y'}^{(1)}_{1,1}$.  Assume for the sake of contradiction that ${Y'}^{(1)}_{1,1} > Y_{1,1}$.  Set $\phi := \mathcal{R}_{-1}({\alpha'}^{(1)}, \widehat{\nu'}^{(1)})$.  Set $i_0 := \phi^{-1}(1)$.  Also set $i'_0 := \tau^{-1}\left(i'_{(1,i_0)}\right)$.  Note that
\begin{align*}
{Y'}^{(1)}_{1,1} & = {X'}^{(1)}_{1,1} + \ell'_1 - 1 \\ & = \mathcal{C}_{-1}({\alpha'}^{(1)}, \widehat{\nu'}^{(1)}, i_0, \varnothing, \lbrace 1, \ldots, \ell'_1 \rbrace \setminus \lbrace i_0 \rbrace) \\ & = \left \lceil \frac{\widehat{\nu'}^{(1)}_{i_0} + \sum_{i_1 = 1}^{\ell'_1} \min \left \lbrace {\alpha'}^{(1)}_{i_0}, {\alpha'}^{(1)}_{i_1} \right \rbrace}{{\alpha'}^{(1)}_{i_0}} \right \rceil - 1 \\ & = \left \lceil \frac{{\nu'}^{(1)}_{i_0} + \sum_{i_1 = 1}^{\ell'_1} \min \left \lbrace {\alpha'}^{(1)}_{i_0}, {\alpha'}^{(1)}_{i_1} \right \rbrace + \sum_{x' = 2}^{\mathcal{k}'} \sum_{i'=1}^{\ell'_{x'}} \min \left \lbrace {\alpha'}^{(1)}_{i_0},{\alpha'}^{(x')}_{i_1} \right \rbrace}{{\alpha'}^{(1)}_{i_0}} \right \rceil - 1 \\ & = \left \lceil \frac{\nu'_{i'_0} - \mu'_{i'_{(1,i_0)}} + \sum_{x = 1}^{\mathcal{k}'} \sum_{i'=1}^{\ell'_{x}} \min \left \lbrace {\alpha'}_{i'_0} - 1, {\alpha'}_{\tau^{-1}(i'_{(x, i_1)})} - 1 \right \rbrace}{{\alpha'}_{i'_0} - 1} \right \rceil - 1 \\ & = \left \lceil \frac{\nu'_{i'_0} - \mu'_{i'_{(1,i_0)}} + \sum_{i' = 1}^{\ell'} \min \left \lbrace {\alpha'}_{i'_0} - 1, {\alpha'}_{i'} - 1 \right \rbrace}{{\alpha'}_{i'_0} - 1} \right \rceil - 1 \\ & = \left \lceil \frac{\nu'_{i'_0} - \mu'_{i'_{(1,i_0)}} + \left({\alpha'}^*_1 + \cdots + {\alpha'}^*_{{\alpha'}_{i'_0}} \right) - \ell'}{{\alpha'}_{i'_0} - 1} \right \rceil - 1 \\ & = \left \lceil \frac{\nu_{i_{i'_0}} - \mu^{\circ} - \mu'_{i'_{(1,i_0)}} + \left({\alpha}^*_1 + \cdots + {\alpha}^*_{{\alpha}_{i_{i'_0}}} \right) - \ell - \ell'}{{\alpha}_{i_{i'_0}} - 2} \right \rceil - 1.
\end{align*}  

Suppose that the topmost row of $\mathcal{T}_3(Y)$ is contained within the $p^{\text{th}}$ row of $Y$.  If $p > 1$, then \[\mathcal{T}_3(Y)_{1,1} = Y_{p, 3} \leq Y_{p, 2} + 1 \leq Y_{p, 1} + 1 \leq Y_{1,1} - 1.\]  It follows that $p = 1$, so there are at least three boxes in the first row of $Y$.  Thus, the first row of $Y'$ is contained within the first row of $Y$, and, furthermore, there are at least two boxes in the first row of $Y'$.  Hence $\alpha'_{\tau^{-1}(1)} > 1$, whence $\mathcal{f}'(1) = (1, 1)$.  

Since $\mathcal{T}_3(Y)_{1,1} = Y_{1, 3}$ and $Y_{1, 3} \leq Y_{1, 2} + 1 \leq Y_{1, 1} + 1$, the assumption $Y_{1,3} > Y_{1,1}$ entails $Y_{1,2} = Y_{1,1}$.  Note that $Y_{1,1} = \mu^{\circ} + \ell - 1$ and $Y_{1, 2} = Y'_{1,1} = \mu'_1 + \ell' - 1$, so $\mu^{\circ} + \ell = \mu'_1 + \ell'$.  Then $\mu^{\circ} + \ell = \mu'_{i'_{(1,i_0)}} + \ell'$ because $\mu'_{i'_{(1,i_0)}} = \mu'_{i'_{(1,1)}} = \mu'_1$.

Thus, \[{Y'}^{(1)}_{1,1} = \left \lceil \frac{\nu_{i_{i'_0}} + \left({\alpha}^*_1 + \cdots + {\alpha}^*_{{\alpha}_{i_{i'_0}}} \right) - 2(\mu^{\circ} + \ell)}{{\alpha}_{i_{i'_0}} - 2} \right \rceil - 1.\]

From ${Y'}^{(1)}_{1,1} > Y_{1,1} = \mu^{\circ} + \ell - 1$, we obtain
\begin{align*}
& \left \lceil \frac{\nu_{i_{i'_0}} + \left({\alpha}^*_1 + \cdots + {\alpha}^*_{{\alpha}_{i_{i'_0}}} \right) - 2(\mu^{\circ} + \ell)}{{\alpha}_{i_{i'_0}} - 2} \right \rceil > \mu^{\circ} + \ell \\ & \Longleftrightarrow \frac{\nu_{i_{i'_0}} + \left({\alpha}^*_1 + \cdots + {\alpha}^*_{{\alpha}_{i_{i'_0}}} \right) - 2(\mu^{\circ} + \ell)}{{\alpha}_{i_{i'_0}} - 2} > \mu^{\circ} + \ell \\ & \Longleftrightarrow \frac{\nu_{i_{i'_0}} + \left({\alpha}^*_1 + \cdots + {\alpha}^*_{{\alpha}_{i_{i'_0}}} \right)}{{\alpha}_{i_{i'_0}} - 2} > (\mu^{\circ} + \ell) \left(1 + \frac{2}{\alpha_{i_{i'_0}} - 2} \right) \\ & \Longleftrightarrow \frac{\nu_{i_{i'_0}} + \left({\alpha}^*_1 + \cdots + {\alpha}^*_{{\alpha}_{i_{i'_0}}} \right)}{{\alpha}_{i_{i'_0}}} > \mu^{\circ} + \ell \\ & \Longleftrightarrow \left \lceil \frac{\nu_{i_{i'_0}} + \left({\alpha}^*_1 + \cdots + {\alpha}^*_{{\alpha}_{i_{i'_0}}} \right)}{{\alpha}_{i_{i'_0}}} \right \rceil - 1 > \mu^{\circ} + \ell - 1 \\ & \Longleftrightarrow \mathcal{C}_{-1}(\alpha, \nu, i_{i'_0}, \varnothing, \lbrace 1, \ldots, \ell \rbrace \setminus \lbrace i_{i'_0} \rbrace) > \mathcal{C}_{-1}(\alpha, \nu, 1, \varnothing, \lbrace 2, \ldots, \ell \rbrace). 
\end{align*} 

However, the function given by \[i \mapsto \mathcal{C}_{-1}(\alpha, \nu, i, \varnothing, \lbrace 1, \ldots, \ell \rbrace \setminus \lbrace i \rbrace)\] attains its maximal value over the domain $i \in \lbrace 1, \ldots, \ell \rbrace$ at $i = \sigma^{-1}(1) = 1$. 

Therefore, \[\mathcal{C}_{-1}(\alpha, \nu, i_{i'_0}, \varnothing, \lbrace 1, \ldots, \ell \rbrace \setminus \lbrace i_{i'_0} \rbrace) \leq \mathcal{C}_{-1}(\alpha, \nu, 1, \varnothing, \lbrace 2, \ldots, \ell \rbrace).\]  This is a contradiction.  
\end{proof}

\begin{lem} \label{smallentry}
Let $(\alpha, \nu) \in \mathbb{N}^{\ell} \times \mathbb{Z}^{\ell}$.  Set $(X,Y) := \mathcal{A}(\alpha, \nu, 1)$.  Suppose $X_{1,1} = \cdots = X_{\ell, 1}$.  Then $Y_{\ell,1} \leq Y_{i,j}$ for all $(i, j) \in \mathbb{N} \times \mathbb{N}$ such that $Y$ has an entry in the $i^{\text{th}}$ row and $j^{\text{th}}$ column.  
\end{lem}

\begin{proof}
The proof is analogous to that of Lemma~\ref{bigentry}.  
\end{proof}

\begin{thm} \label{puttog}
Let $(\alpha, \nu) \in \mathbb{N}^{\ell} \times \mathbb{Z}^{\ell}$.  Then $\mathcal{A}(\alpha, \nu, -1)$ is odd-distinguished, and $\mathcal{A}(\alpha, \nu, 1)$ is even-distinguished.  
\end{thm}

\begin{proof}
The proof is by induction on $\max \lbrace \alpha_1, \ldots, \alpha_{\ell} \rbrace$.  We show the inductive step for the former statement only.  

Maintain the notation following the definitions of the row-survival and row-partition functions.  Recall that \[Y = \operatorname{Cat}(\mathcal{Y}^{(1)}, \ldots, \mathcal{Y}^{(\mathcal{k})}) \quad \text{and} \quad \mathcal{T}_2(Y) = \operatorname{Cat}(Y^{(1)}, \ldots, Y^{(\mathcal{k})}).\]  By the inductive hypothesis, $Y^{(x)}$ is even-distinguished for all $1 \leq x \leq \mathcal{k}$.  To see that $Y$ is odd-distinguished, we prove that $Y$ satisfies the four conditions delineated in Definition~\ref{dis}.

\begin{enumerate}
	\item This follows immediately from Theorem~\ref{differences}.  
	\item \begin{enumerate}
		\item Suppose $j < j'$ are odd and $Y_{i,j} \leq Y_{i, j'} - 1$.  We split into two cases.  
		
		If $j > 1$, note that there are at least two boxes in the $i^{\text{th}}$ row of $Y$.  Setting $(x, i') := \mathcal{S}(\alpha, \sigma, \iota)$, we obtain $i' > 0$, so $i = i_{(x,i')}$.  Thus, \[Y^{(x)}_{i', j-1} = Y_{i,j} \leq Y_{i, j'} - 1 = Y^{(x)}_{i', j' - 1} - 1.\]  Since $Y^{(x)}$ is even-distinguished, it follows that $Y^{(x)}_{i', j-1}$ is not $E$-raisable in $Y^{(x)}$, so $Y_{i,j}$ is not $E$-raisable in $Y$.  
		
		If $j = 1$, assume for the sake of contradiction that $Y_{i,j}$ is $E$-raisable.  Setting $(x, i') := \mathcal{P}(\alpha, \iota)(i)$, we obtain $i' = 1$, so $i = p_{(x,1)}$.  Then \[\mathcal{Y}^{(x)}_{1, 1} = Y_{i, 1} \leq Y_{i, j'} - 1 = \mathcal{Y}^{(x)}_{1, j'} - 1,\] which contradicts Lemma~\ref{bigentry}.  
		\item Suppose $j < j'$ are even and $Y_{i, j} \geq Y_{i, j'} + 1$.  As above, note that there are at least two boxes in the $i^{\text{th}}$ row of $Y$.  Setting $(x, i') := \mathcal{S}(\alpha, \sigma, \iota)(i)$, we again obtain $i' > 0$, so $i = i_{(x,i')}$.  Thus, \[Y^{(x)}_{i', j-1} = Y_{i,j} \geq Y_{i, j'} + 1 = Y^{(x)}_{i', j'-1} + 1.\]  Since $Y^{(x)}$ is even-distinguished, it follows that $Y^{(x)}_{i', j-1}$ is not $E$-lowerable in $Y^{(x)}$, so $Y_{i,j}$ is not $E$-lowerable in $Y$.  
	\end{enumerate}
	\item For both parts of this condition, it suffices to address the case $j = 1$; the condition holds for $j > 1$ by the inductive hypothesis (as in the proof of condition (2)).  
	\begin{enumerate}
		\item Suppose $Y_{i,1} \leq Y_{i,j'} - 2$.  If $j'$ is odd, then $Y_{i,1}$ is not $E$-raisable by condition (2).  Otherwise, $j'$ is even, and we see from Theorem~\ref{differences} that \[Y_{i, 1} \leq Y_{i, j'} - 2 \leq Y_{i, j'-1} - 2.\]  Thus, $j' -1 > 1$, and, again invoking condition (2), we find $Y_{i,1}$ is not $E$-raisable.  
		\item Suppose $Y_{i,1} \geq Y_{i,j'} + 2$, and assume for the sake of contradiction that $Y_{i,1}$ is $E$-lowerable.  Setting $(x,i') := \mathcal{P}(\alpha, \iota)(i)$, we obtain $i' = \ell^{\circ}_x$, so $i = p_{(x, \ell^{\circ}_x)}$.  By Theorem~\ref{differences},
		\begin{align} \label{compare}
		\mathcal{Y}^{(x)}_{\ell^{\circ}_x, 2} = Y_{i, 2} \geq Y_{i, 1} - 1 \geq Y_{i, j'} + 1 = \mathcal{Y}^{(x)}_{\ell^{\circ}_x, j'} + 1.
		\end{align}
		
		Note that $\mathcal{Y}^{(x)}_{\ell^{\circ}_x, 2}$ is $E$-lowerable in $\mathcal{Y}^{(x)}$ (\textit{even if} $Y_{i,2}$ is \textit{not} $E$-lowerable in $Y$).  Thus, if $j'$ is even, then $j' > 2$, and Equation~\ref{compare} contradicts condition (2).  Otherwise, $j'$ is odd, and \[\mathcal{Y}^{(x)}_{\ell^{\circ}_x, 2} \geq \mathcal{Y}^{(x)}_{\ell^{\circ}_x,j'} + 1 \geq \mathcal{Y}^{(x)}_{\ell^{\circ}_x,j'-1} + 1,\] which means $j' - 1 > 2$, and again yields a contradiction with condition (2).   
	\end{enumerate}
	\item Clearly this condition holds for $j=1$.  Therefore, since $Y^{(x)}$ is even-distinguished for all $1 \leq x \leq \mathcal{k}$, it suffices to show $x < x'$ implies $Y^{(x)}_{i, j} \geq Y^{(x')}_{i', j} + 2$.
	
	We claim that $Y^{(x)}_{i, j} \geq Y^{(x)}_{\ell'_x, 1}$.  To see this, note that Lemma~\ref{collapse} tells us that there exists a positive integer $\mathcal{k}_x$ and pairs of integer sequences \[\left(\alpha^{(x;1)}, \nu^{(x;1)}\right), \ldots, \left(\alpha^{(x;\mathcal{k}_x)}, \nu^{(x; \mathcal{k}_x)}\right)\] such that the diagram pairs $\left(X^{(x;y)}, Y^{(x;y)}\right) := \mathcal{A}_1 \left(\alpha^{(x;y)}, \nu^{(x;y)} \right)$ satisfy the following conditions: For all $1 \leq y \leq \mathcal{k}_x$, the entries in the first column of $X^{(x;y)}$ are all equal, and \[Y^{(x)} = \operatorname{Cat}\left(Y^{(x;1)}, \ldots, Y^{(x; \mathcal{k}_x)}\right).\]
	
	Let $y$ be chosen so that the $i^{\text{th}}$ row of $Y^{(x)}$ is contained in $Y^{(x;y)}$.  By Lemma~\ref{smallentry}, $Y^{(x)}_{i,j}$ is greater than or equal to the bottommost entry in the first column of $Y^{(x;y)}$.  This entry belongs to the first column of $Y^{(x)}$, so it is itself greater than or equal to $Y^{(x)}_{\ell'_x,1}$, which proves the claim.  
	
	By Theorem~\ref{differences}, \[Y^{(x)}_{\ell'_x, 1} = Y_{i_{(x,\ell'_x)}, 2} \geq Y_{i_{(x,\ell'_x)},1} - 1 \geq Y_{p_{(x,\ell^{\circ}_x)}, 1} - 1.\]  
	
	Furthermore, \[Y_{p_{(x', 1)}, 1} + 2 = \mathcal{Y}^{(x')}_{1,1} + 2 \geq Y^{(x')}_{i',j} + 2,\] where the inequality follows from Lemma~\ref{bigentry} because $Y^{(x')}$ is contained in $\mathcal{Y}^{(x')}$, as shown in Lemma~\ref{collapse}.  
	
	Thus, it suffices to show $Y_{p_{(x,\ell^{\circ}_x)}, 1} - 1 \geq Y_{p_{(x',1)}, 1} + 2$.  By definition, \[\mathcal{P}(\alpha, \iota)(p_{(x,\ell^{\circ}_x)}) = (x, \ell^{\circ}_x) \quad \text{and} \quad \mathcal{P}(\alpha, \iota)(p_{(x',1)}) = (x', 1).\]  Hence $x < x'$ implies $X_{p_{(x,\ell^{\circ}_x)}, 1} > X_{p_{(x',1)}, 1}$, so $X_{p_{(x,\ell^{\circ}_x)}, 1} - 1 \geq X_{p_{(x',1)}, 1}$, whence the result follows.  
\end{enumerate}
\end{proof}

\begin{cor} \label{dist}
Let $\alpha \vdash n$, and let $\nu \in \Omega_{\alpha}$.  Then $\mathcal{A}(\alpha, \nu)$ is distinguished.  
\end{cor}

Corollary~\ref{dist}, in view of Proposition~\ref{kap}, suffices to prove Theorem~\ref{mata} --- thanks to the following theorem.  

\begin{thm}[Achar \cite{Acharj}, Theorem 8.8] \label{achar}
Let $\alpha \vdash n$, and let $\nu \in \Omega_{\alpha}$.  Then $\mathsf{A}(\alpha, \nu)$ is the unique distinguished diagram of shape-class $\alpha$ in $\kappa^{-1}(\nu)$.  
\end{thm}

\begin{rem}
Again (cf. Remark~\ref{weak}), Achar's definition of distinguished is weaker than ours, but it doesn't matter: $p_1 \mathcal{A}(\alpha, \nu)$ is distinguished by our definition, so it is distinguished by Achar's definition, so $p_1\mathcal{A}(\alpha, \nu) = \mathsf{A}(\alpha, \nu)$.  
\end{rem}

This completes the proof of Theorem~\ref{mata}.  \qed

\vfill \eject

\appendix

\section{Afterword}

Achar's algorithm $\mathsf{A}$ computes a map \[\Omega \rightarrow \bigcup_{\ell = 1}^n D_{\ell}.\]  On input $(\alpha, \nu) \in \Omega$, the corresponding dominant weight $\gamma(\alpha, \nu) \in \Lambda^+$ is obtained from the output by taking $\eta E \mathsf{A}(\alpha, \nu)$.  

Achar's algorithm for $\gamma^{-1}$, which we denote by $\mathsf{B}$, computes a map \[\Lambda^+ \rightarrow \bigcup_{\ell = 1}^n D_{\ell}.\]  On input $\lambda \in \Lambda^+$, the corresponding pair $\gamma^{-1}(\lambda) \in \Omega$ is obtained from the output by taking $(\delta E^{-1} \mathsf{B}(\lambda), \kappa E^{-1} \mathsf{B}(\lambda))$, where the map $\delta$ sends a diagram to its shape-class.  

Consider the following diagram.  

\[\Omega \xleftarrow{(\delta p_1, \kappa p_1)} \bigcup_{\ell=1}^n D_{\ell} \times D_{\ell} \xrightarrow{\eta p_2} \Lambda^+\]

The algorithms $\mathsf{A}$ and $\mathsf{B}$ yield sections $(\mathsf{A}, E \mathsf{A})$ and $(E^{-1}\mathsf{B}, \mathsf{B})$ of the projections $(\delta p_1, \kappa p_1)$ and $\eta p_2$ onto $\Omega$ and $\Lambda^+$, respectively, for which \[\eta p_2 \circ (\mathsf{A}, E \mathsf{A}) = \gamma \quad \text{and} \quad (\delta p_1, \kappa p_1) \circ (E^{-1} \mathsf{B}, \mathsf{B}) = \gamma^{-1}.\]  That the maps computed by $\mathsf{A}$ and $\mathsf{B}$ play symmetric roles suggests that the algorithms themselves should exhibit structural symmetry.  Unfortunately, they do not.  

We address this incongruity by introducing the algorithm $\mathcal{A}$, which computes the section $(\mathsf{A}, E \mathsf{A})$, yet has the same recursive structure as $\mathsf{B}$: Both $\mathcal{A}$ and $\mathsf{B}$ recur after determining the entries in the first column of their output diagram(s).\footnote{Achar \cite{Acharj} phrases the instructions for $\mathsf{B}$ to use iteration rather than recursion, but they amount to carrying out the same computations.  } Thus, the weight-diagrams version of our algorithm achieves structural parity with $\mathsf{B}$; the integer-sequences version $\mathfrak{A}$ is a singly recursive simplification that sidesteps weight diagrams altogether.  

Having established that our algorithm is correct, we are mindful that Einstein's admonition, ``Everything should be made as simple as possible, but not simpler,'' will have the last word.  We believe the appeal of our weight-by-weight, column-by-column approach is underscored by its consonance with Achar's algorithm $\mathsf{B}$, which has stood the test of time.  To demonstrate the complementarity between $\mathcal{A}$ and $\mathsf{B}$, we offer the following description of $\mathsf{B}$.  More details can be found in Achar \cite{Acharj}, section 6.  

\subsubsection*{The algorithm}

We define a recursive algorithm $\mathsf{B}$ that computes a map \[\mathbb{Z}^n_{\operatorname{dom}} \times \lbrace \pm 1 \rbrace \rightarrow \bigcup_{\ell = 1}^n D_{\ell}\] by filling in the entries in the first column of its output diagram and using recursion to fill in the entries in the remaining columns.  Whenever we write $\mathsf{B}(\lambda)$, we refer to $\mathsf{B}(\lambda, -1)$.  

The algorithm $\mathsf{B}$ is multiply recursive and begins by dividing a weakly decreasing integer sequence into \textit{clumps}.  From each clump, it builds a diagram by first extracting a maximal-length \textit{majuscule} sequence to comprise the entries of the first column, and then calling itself on the clump's remains.  To obtain the output, it concatenates the diagrams constructed from all the clumps.   

\begin{df}
A subsequence of a weakly decreasing integer sequence is \textit{clumped} if no two of its consecutive entries differ by more than $1$.  A clumped subsequence is a \textit{clump} if it is not contained in a longer clumped subsequence.  
\end{df}

\begin{df}
An integer sequence $\iota = [\iota_1, \ldots, \iota_{\ell}]$ is \textit{majuscule} if $\iota_i - \iota_{i+1} \geq 2$ for all $1 \leq i \leq \ell - 1$.  
\end{df}

We are ready to describe $\mathsf{B}$.  

On input $(\lambda, \epsilon)$, the algorithm designates $\mathsf{c}$ the number of distinct clumps in $\lambda$.  For all $1 \leq x \leq \mathsf{c}$, it designates $n_x$ the number of entries in the $x^{\text{th}}$ clump.

Let $\lambda^{(x)}$ denote the $x^{\text{th}}$ clump, and $\mathcal{Y}^{(x)}$ the diagram to be built from $\lambda^{(x)}$.  For all $1 \leq x \leq \mathsf{c}$, the algorithm obtains a majuscule sequence $\iota^{(x)}$ from $\lambda^{(x)}$ as follows:
\begin{itemize}
\item If $\epsilon = -1$, then $\iota^{(x)}$ is the maximal-length majuscule sequence contained in $\lambda^{(x)}$ that begins with $\lambda^{(x)}_1$; 
\item If $\epsilon = 1$, then $\iota^{(x)}$ is the maximal-length majuscule sequence contained in $\lambda^{(x)}$ that ends with $\lambda^{(x)}_{n_x}$.  
\end{itemize}

Then it sets \[\mathcal{Y}^{(x)}_{i, 1} := \iota^{(x)}_i\] for all $1 \leq i \leq \ell_x$, where $\ell_x$ is the length of $\iota^{(x)}$.  This determines the entries in the first column of $\mathcal{Y}^{(x)}$.  

If $\iota^{(x)} = \lambda^{(x)}$, the diagram $\mathcal{Y}^{(x)}$ is complete.  Otherwise, the algorithm arranges the elements of the (multiset) difference $\lambda^{(x)} \setminus \iota^{(x)}$ in weakly decreasing order, leaving a weakly decreasing integer sequence $\bar{\lambda}^{(x)}$, and it sets \[Y^{(x)} := \mathsf{B}(\bar{\lambda}^{(x)}, -\epsilon).\]  

It proceeds to attach $Y^{(x)}$ to the first column of $\mathcal{Y}^{(x)}$.  For all $i'$ such that $Y^{(x)}$ has an $i'^{\text{th}}$ row, the algorithm finds the unique $i \in \lbrace 1, \ldots, \ell_x \rbrace$ such that $Y^{(x)}_{i',1} - \mathcal{Y}^{(x)}_{i,1} \in \lbrace 0, \epsilon \rbrace$.  Then, for all $j'$ such that $Y^{(x)}$ has an entry in the $i'^{\text{th}}$ row and $j'^{\text{th}}$ column, it sets \[\mathcal{Y}^{(x)}_{i,j'+1} := Y^{(x)}_{i',j'}.\]

Finally, it sets \[Y := \operatorname{Cat}\left (\mathcal{Y}^{(1)}, \ldots, \mathcal{Y}^{(\mathcal{c})} \right )\] and returns $Y$.  

\begin{exam}
In Example~\ref{acharexam}, we found \[\gamma([4,3,2,1,1], [15,14,9,4,4]) = [8,7,6,6,5,4,3,3,2,2,0].\]

Here we set $\lambda := [8,7,6,6,5,4,3,3,2,2,0]$ and compute $\mathsf{B}(\lambda)$.  

We start by observing that $\lambda$ has $2$ clumps, \[\lambda^{(1)} = [8,7,6,6,5,4,3,3,2,2] \quad \text{and} \quad 
\lambda^{(2)} = [0].\]  

The maximal-length majuscule sequence contained in $\lambda^{(1)}$ that begins with $\lambda^{(1)}_1 = 8$ is $\iota^{(1)} = [8,6,4,2]$.  Hence \[\big[\mathcal{Y}^{(1)}_{1,1}, \mathcal{Y}^{(1)}_{2,1}, \mathcal{Y}^{(1)}_{3,1}, \mathcal{Y}^{(1)}_{4,1} \big] = [8,6,4,2].\]

Upon removing $\iota^{(1)}$ from $\lambda^{(1)}$, we see that \[\bar{\lambda}^{(1)} = [7,6,5,3,3,2].\]  

As it happens, $\mathsf{B}(\bar{\lambda}^{(1)}, 1)$ looks as follows.  

\begin{figure}[h]
	\ytableausetup{centertableaux, textmode}
	$Y^{(1)} = $ \begin{ytableau}
		7\\
		5 & 6\\
		2 & 3 & 3
	\end{ytableau}
	\caption{The diagram obtained from the remains of the first clump}
\end{figure}  

We complete $\mathcal{Y}^{(1)}$ by attaching $Y^{(1)}$ to the first column of $\mathcal{Y}^{(1)}$.  

\begin{figure}[h]
	\ytableausetup{centertableaux, textmode}
	$\mathcal{Y}^{(1)} = $ \begin{ytableau}
		8 & 7\\
		6 & 5 & 6\\
		4 \\
		2 & 2 & 3 & 3
	\end{ytableau}
	\caption{The diagram obtained from the first clump}
\end{figure}  

Since $\lambda^{(2)}$ consists of a single entry, it follows that $\iota^{(2)} = \lambda^{(2)}$, so $\mathcal{Y}^{(2)}$ consists of a single box.  

\begin{figure}[h]
	\ytableausetup{centertableaux, textmode}
	$\mathcal{Y}^{(2)} = $ \begin{ytableau}
		0 
	\end{ytableau}
	\caption{The diagram obtained from the second clump}
\end{figure}

Concatenating $\mathcal{Y}^{(1)}$ and $\mathcal{Y}^{(2)}$, we arrive at $Y$.  

\begin{figure}[h]
	\ytableausetup{centertableaux, textmode}
	$Y = $ \begin{ytableau}
		8 & 7\\
		6 & 5 & 6\\
		4 \\
		2 & 2 & 3 & 3 \\
		0
	\end{ytableau}
	\caption{The diagram obtained from the recursion}
\end{figure}  
\end{exam}

Comparing our result with that of Example~\ref{cont}, we see d\'ej\`a vu --- $\mathcal{A}(\alpha, \nu) = (E^{-1}Y, Y)$.  This corroborates that the sections $\mathcal{A}$ and $(E^{-1} \mathsf{B}, \mathsf{B})$ send $(\alpha, \nu)$ and $\lambda$, respectively, to the same diagram pair whenever $(\alpha, \nu)$ and $\lambda$ correspond under the Lusztig--Vogan bijection, as in this case they do.  
\vfill \eject

\section*{Acknowledgments}

This research is based on the author's doctoral dissertation at the Massachusetts Institute of Technology.  His gratitude to his advisor David A. Vogan Jr. --- for suggesting the problem and offering invaluable help and encouragement along the way to a solution --- cannot be overstated.  Thanks are also due to Roman Bezrukavnikov and George Lusztig for serving on the author's thesis committee and offering thoughtful remarks.  

John W. Chun, the author's late grandfather, emigrated from Korea and fell in love with the author's grandmother and American literature.  He earned a doctorate in English at the Ohio State University, becoming the first member of the author's family to be awarded a Ph.D.  This article is dedicated to his memory, fondly and with reverence.  

Throughout his graduate studies, the author was supported by the US National Science Foundation Graduate Research Fellowship Program.


\begin{thebibliography}{14}
\bibitem{Achart} P. Achar, Equivariant coherent sheaves on the nilpotent cone for complex reductive Lie groups, Ph.D. thesis, Massachusetts Institute of Technology (2001).
\bibitem{Acharj} \underline{\hspace{0.5in}}, On the equivariant $K$-theory of the nilpotent cone in the general linear group, \emph{Represent. Theory} \textbf{8} (2004), 180--211.  
\bibitem{Bezrukav} R. Bezrukavnikov, Quasi-exceptional sets and equivariant coherent sheaves on the nilpotent cone, \emph{Represent. Theory} \textbf{7} (2003), 1--18.  
\bibitem{Jantzen} A. Jantzen, Nilpotent orbits in representation theory, \emph{Lie Theory: Lie algebras and representations}, ed. J.-P. Anker and B. Orsted, New York: Springer Science+Business Media (2004), 1--211.  
\bibitem{Knutson} A. Knutson and T. Tao, The honeycomb model of $GL_n(\mathbb{C})$ tensor products I: Proof of the saturation conjecture, \emph{J. Amer. Math. Soc.} \textbf{12} (1999), 1055--1090.  
\bibitem{Kumar} S. Kumar, Proof of the Parthasarathy--Ranga Rao--Varadarajan conjecture, \emph{Invent. Math.} \textbf{93} (1988), 117--130.  
\bibitem{Lusztig1} G. Lusztig, Cells in affine Weyl groups, \emph{Algebraic Groups and Related Topics}, ed. R. Hotta, \emph{Adv. Stud. Pure Math.} \textbf{6}, Tokyo: Kinokuniya (1985), 255--287.  
\bibitem{Lusztig2} \underline{\hspace{0.5in}}, Cells in affine Weyl groups, II, \emph{J. Algebra} \textbf{109} (1987), 536--548.  
\bibitem{Lusztig3} \underline{\hspace{0.5in}}, Cells in affine Weyl groups, III, \emph{J. Fac. Sci. Univ. Tokyo} \textbf{34} (1987), 223--243.  
\bibitem{Lusztig4} \underline{\hspace{0.5in}}, Cells in affine Weyl groups, IV, \emph{J. Fac. Sci. Univ. Tokyo} \textbf{36} (1989), 297--328.  
\bibitem{Ostrik} V. Ostrik, On the equivariant $K$-theory of the nilpotent cone, \emph{Represent. Theory} \textbf{4} (2000), 296--305. 
\bibitem{Rush} D. B. Rush, Computing the Lusztig--Vogan bijection, Ph.D. thesis, Massachusetts Institute of Technology (2017). 
\bibitem{Vogan} D. Vogan, Associated varieties and unipotent representations, \emph{Harmonic Analysis on Reductive Groups}, ed. W. Barker and P. Sally, New York: Springer Science+Business Media (1991), 315--388.  
\bibitem{Xi} N. Xi, The based ring of two-sided cells of affine Weyl groups of type $\widetilde A_{n-1}$, \emph{Mem. Amer. Math. Soc.} \textbf{157} (2002), no. 749.  
\end{thebibliography}
\end{document}